\documentclass[a4paper,11pt,reqno]{amsart}

\usepackage[margin=1in,includehead,includefoot]{geometry}
\usepackage[utf8]{inputenc}
\usepackage[T1]{fontenc}
\usepackage{lmodern,microtype}
\usepackage{upref}
\usepackage{mathtools,amssymb,amsthm,amsmath}
\usepackage[foot]{amsaddr}
\usepackage{enumitem}
\usepackage[disable]{todonotes}
\usepackage{graphicx}
\usepackage{hyperref}
\usepackage{wrapfig}
\usepackage{xcolor}
\usepackage{array}
\usepackage[margin=1cm]{caption}
\usepackage{bookmark}
\usepackage{mycommands}

\graphicspath{{./graphics/}}

\makeatletter
\let\old@setaddresses\@setaddresses
\def\@setaddresses{\bigskip{\parindent 0pt\let\scshape\relax\let\ttfamily\relax\old@setaddresses}}
\makeatother

\newtheorem{theorem}{Theorem}

\newtheorem{proposition}[theorem]{Proposition}
\newtheorem{lemma}[theorem]{Lemma}
\theoremstyle{remark}


\hypersetup{
  pdftitle={On a combinatorial generation problem of Knuth},
  pdfauthor={Arturo Merino, Ondrej Micka, Torsten M\"utze}
}

\title{On a combinatorial generation problem of Knuth}

\author{Arturo Merino}
\address[Arturo Merino]{Department of Mathematics, TU Berlin, Germany}
\email{merino@math.tu-berlin.de}

\author{Ond\v{r}ej Mi\v{c}ka}
\address[Ond\v{r}ej Mi\v{c}ka]{Department of Theoretical Computer Science and Mathematical Logic, Charles University, Prague, Czech Republic}
\email{micka@ktiml.mff.cuni.cz}

\author{Torsten M\"utze}
\address[Torsten M\"utze]{Department of Computer Science, University of Warwick, United Kingdom \& Department of Theoretical Computer Science and Mathematical Logic, Charles University, Prague, Czech Republic}
\email{torsten.mutze@warwick.ac.uk}

\thanks{This work was supported by Czech Science Foundation grant GA~19-08554S and by German Science Foundation grant~413902284. Arturo Merino was also supported by ANID Becas Chile 2019-72200522.}

\thanks{An extended abstract of this paper has been accepted for presentation at the 32nd Annual ACM-SIAM Symposium on Discrete Algorithms (SODA 2021).}

\begin{document}

\begin{abstract}
The well-known middle levels conjecture asserts that for every integer $n\geq 1$, all binary strings of length $2(n+1)$ with exactly $n+1$ many 0s and 1s can be ordered cyclically so that any two consecutive strings differ in swapping the first bit with a complementary bit at some later position.
In his book `The Art of Computer Programming Vol.~4A' Knuth raised a stronger form of this conjecture (Problem~56 in Section~7.2.1.3), which requires that the sequence of positions with which the first bit is swapped in each step of such an ordering has $2n+1$ blocks of the same length, and each block is obtained by adding $s=1$ (modulo $2n+1$) to the previous block.
In this work, we prove Knuth's conjecture in a more general form, allowing for arbitrary shifts $s\geq 1$ that are coprime to~$2n+1$.
We also present an algorithm to compute this ordering, generating each new bitstring in $\cO(n)$ time, using $\cO(n)$ memory in total.
\end{abstract}


\maketitle

\section{Introduction}

In computer science and mathematics we frequently encounter various fundamental classes of combinatorial objects such as subsets, permutations, combinations, partitions, trees etc.
There are essentially three recurring algorithmic tasks we want to perform with such objects, namely counting (how many objects are there?), random generation (pick one object uniformly at random), and exhaustive generation (generate every object exactly once).
The focus of this paper is on the latter of these tasks, namely algorithms for exhaustively generating a class of combinatorial objects.
This research area has flourished tremendously, in particular since the advent of powerful computers, and many of the gems it has produced are treated in depth in the most recent volume of Knuth's seminal series `The Art of Computer Programming'~\cite{MR3444818} (see also the classical book by Nijenhuis and Wilf~\cite{MR0396274}).

\subsection{Combination generation}

One of the basic classes of combinatorial objects we want to generate are \emph{$(k,\ell)$-combinations}, i.e., all ways of choosing a subset of a fixed size~$k$ from the ground set~$[n]:=\{1,\ldots,n\}$ where $n:=k+\ell$.
In a computer we conveniently encode every set by a bitstring of length~$n$ with exactly $k$ many 1s, where the $i$th bit is~1 if and only if the element~$i$ is contained in the set.
For instance, all 2-element subsets of the 4-element ground set $\{1,2,3,4\}$ are $12$, $13$, $14$, $23$, $24$, $34$, where we omit curly brackets and commas for simplicity, and the corresponding bitstrings are $1100,1010,1001,0110,0101,0011$.
As we are concerned with fast generation algorithms, a natural approach is to generate a class of objects in an order such that any two consecutive objects differ only by a small amount, i.e., we aim for a \emph{Gray code} ordering.
In general, a combinatorial Gray code is a minimum change ordering of objects for some specified closeness criterion, and fast algorithms for generating such orderings have been discovered for a large variety of combinatorial objects of interest (see~\cite{MR1491049,MR3444818}).
For combinations, we aim for an ordering where any two consecutive sets differ only in exchanging a single element, such as $(12,13,14,24,34,23)=(1\ul{1}\ul{0}0,10\ul{1}\ul{0},\ul{1}\ul{0}01,0\ul{1}\ul{0}1,0\ul{0}1\ul{1},\ul{0}1\ul{1}0)$.
As we can see, this corresponds to swapping a 0-bit with a 1-bit in the bitstring representation in every step, where the two swapped bits are underlined in the example.

\subsection{The middle levels conjecture}

In the 1980s, Buck and Wiedemann~\cite{MR737262} conjectured that all $(n+1,n+1)$-combinations can be generated by \emph{star transpositions} for every $n\geq 1$, i.e., the element~1 either enters or leaves the set in each step.
In terms of bitstrings, this means that in every step the first bit is swapped with a complementary bit at a later position.
The ordering is also required to be cyclic, i.e., this transition rule must also hold when going from the last combination back to the first.
The corresponding \emph{flip sequence~$\alpha$} records the position of the bit with which the first bit is swapped in each step, where positions are indexed by~$0,\ldots,2n+1$, so the entries of~$\alpha$ are from the set~$\{1,\ldots,2n+1\}$ and $\alpha$ has length~$N:=\binom{2(n+1)}{n+1}$.
For example, a cyclic star transposition ordering of $(2,2)$-combinations is $(12,23,13,34,14,24)=(\ul{1}1\ul{0}0,\ul{0}\ul{1}10,\ul{1}01\ul{0},\ul{0}0\ul{1}1,\ul{1}\ul{0}01,\ul{0}10\ul{1})$, and the corresponding flip sequence is~$\alpha=213213$.
Buck and Wiedemann's conjecture was raised independently by Havel~\cite{MR737021} and became known as \emph{middle levels conjecture}.
This name originates from an equivalent formulation of the problem, which asks for a Hamilton cycle through the middle two levels of the $(2n+1)$-dimensional hypercube.
This conjecture received considerable attention in the literature (see \cite{MR1350586, MR1329390, MR2046083, MR2195731, MR2609124, MR962224, MR962223, MR1268348}), as it lies at the heart of several related combinatorial generation problems.
It is also mentioned in the popular books by Winkler~\cite{MR2034896} and by Diaconis and Graham~\cite{MR2858033}, and in Gowers' survey~\cite{MR3584100}.
Eventually, the middle levels conjecture was solved by M\"utze~\cite{MR3483129} and a simplified proof appeared in~\cite{gregor-muetze-nummenpalo:18}.
Moreover, a constant-time algorithm for computing a star transposition ordering for $(n+1,n+1)$-combinations for every $n\geq 1$ was presented in~\cite{MR4075363}.

\subsection{Knuth's stronger conjecture}

In Problem~56 in Section~7.2.1.3 of his book~\cite{MR3444818} (page~735), which was ranked as the hardest open problem in the book with a difficulty rating of~49/50, Knuth raised a stronger version of the middle levels conjecture, which requires additional symmetry in the flip sequence.
Specifically, Knuth conjectured that there is a star transposition ordering of $(n+1,n+1)$-combinations for every $n\geq 1$ such that the flip sequence~$\alpha$ has a block structure $\alpha=(\alpha_0,\alpha_1,\ldots,\alpha_{2n})$, where each block~$\alpha_i$ has the same length $N/(2n+1)$ and is obtained from the initial block~$\alpha_0$ by element-wise addition of~$i$ modulo~$2n+1$ for all $i=1,\ldots,2n$.
As the entries of~$\alpha$ are from~$\{1,\ldots,2n+1\}$, the numbers $1,\ldots,2n+1$ are chosen as residue class representatives for this addition, rather than $0,\ldots,2n$.
In other words, such a flip sequence~$\alpha$ has cyclic symmetry and the initial block~$\alpha_0$ alone encodes the entire flip sequence~$\alpha$ by a factor of~$2n+1$ more compactly.
The compression factor~$2n+1$ is best possible, and it arises from the fact that every bitstring obtained by removing the first bit of an $(n+1,n+1)$-combination has exactly $2n+1$ distinct cyclic rotations.
Also note that $N/(2n+1)=2C_n$, where $C_n:=\frac{1}{n+1}\binom{2n}{n}$ is the $n$th Catalan number.
For instance, for $n=2$ we have $N=20$, and all $(3,3)$-combinations can be generated from~$111000$ by the flip sequence $(4134\;5245\;1351\;2412\;3523)$, i.e., with initial block $\alpha_0:=4134$.
Similarly, for $n=3$ we have $N=70$, and all $(4,4)$-combinations can be generated from~$11110000$ by the flip sequence defined by the initial block $\alpha_0:=6253462135$.
The entire ordering of combinations obtained for this example is shown in the first column in Figure~\ref{fig:c44}.
In fact, the compact encoding of the flip sequence required in Knuth's problem was the main tool researchers used in tackling the middle levels conjecture experimentally, as it allows restricting the search space by a factor of~$2n+1$ (which yields an exponential speedup for brute-force searches).
This approach was already employed by Buck and Wiedemann~\cite{MR737262} for $n=3,4,5$, and was later refined and implemented on powerful computers by Shields, Shields, and Savage~\cite{MR2548541} for values up to~$n\leq 17$ and by Shimada and Amano \cite{shimada-amano} for $n=18,19$.

As mentioned before, the middle levels conjecture is equivalent to asking for a Hamilton cycle through the middle two levels of the $(2n+1)$-dimensional hypercube.
This subgraph of the hypercube is vertex-transitive, as is the entire graph.
Knuth's problem therefore asks for a highly symmetric Hamilton cycle in a highly symmetric graph.
Specifically, such a Hamilton cycle is invariant under a non-trivial automorphism of the graph, in this case rotation of the bitstrings, and the entire cycle of length~$N$ is the image of a path of length~$N/(2n+1)$ under repeated applications of this automorphism.

\subsection{Our results}

Unfortunately, none of the flip sequences constructed in~\cite{MR3483129,gregor-muetze-nummenpalo:18,MR4075363} to solve the middle levels conjecture satisfies the stronger symmetry requirements of Knuth's problem.
The main contribution of this work is to solve Knuth's symmetric version of the middle levels conjecture in the following more general form, allowing for arbitrary shifts; see Figure~\ref{fig:c44} for illustration.

\begin{figure}
\makebox[0cm]{ 
\includegraphics{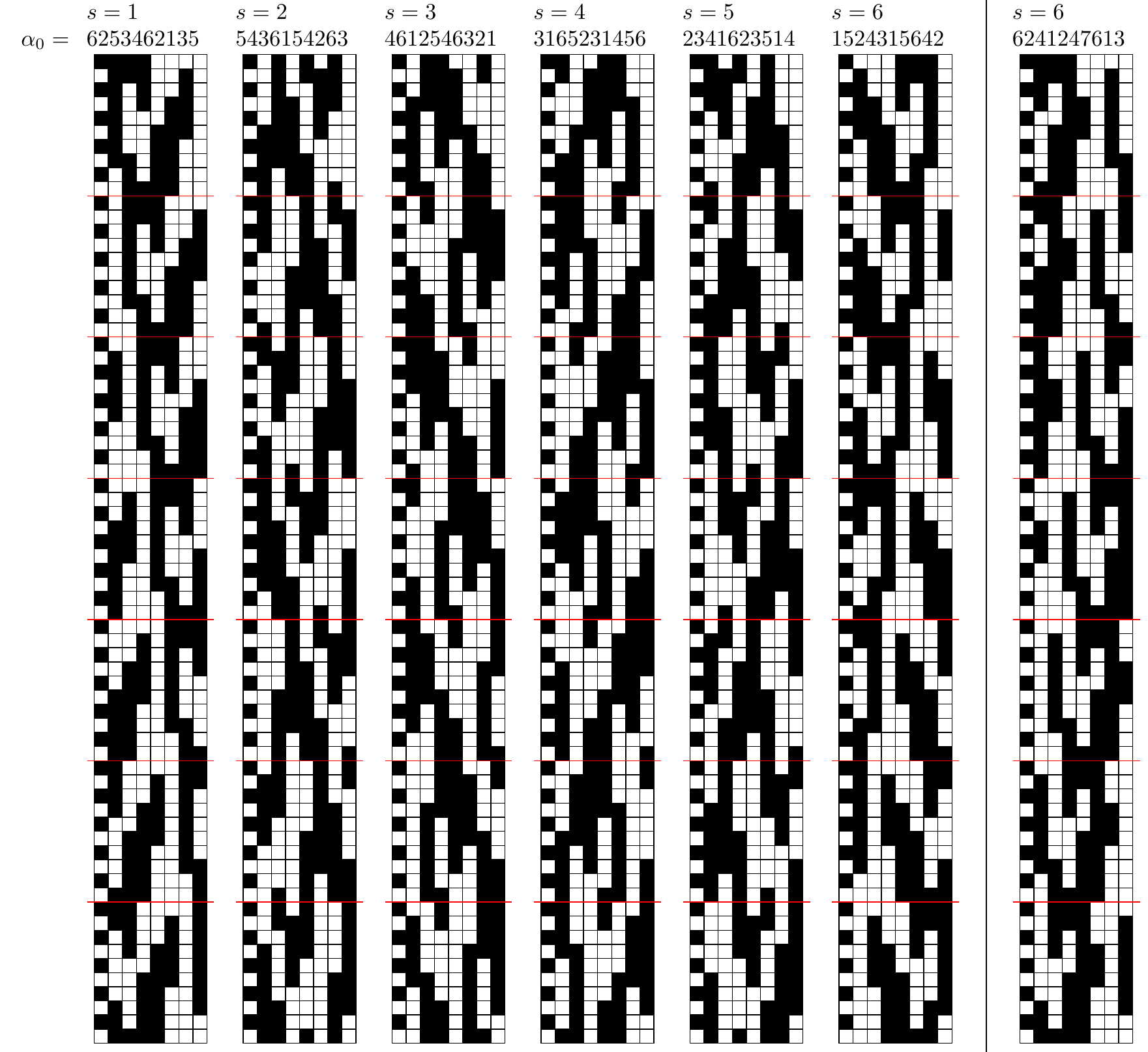}
}
\caption{Star transposition Gray codes for $(4,4)$-combinations obtained from Theorem~\ref{thm:star} for $n=3$ and $s=1,2,\ldots,6$ (two distinct solutions for $s=6$ are shown), where 1-bits are drawn as black squares, 0-bits as white squares.
The initial block~$\alpha_0$ of the flip sequence is shown at the top, and the division of all $N=70$ combinations into $2n+1=7$ blocks of length $2C_n=10$ is highlighted by horizontal lines.
}
\label{fig:c44}
\end{figure}

\begin{theorem}
\label{thm:star}
For any $n\geq 1$ and $1\leq s\leq 2n$ that is coprime to $2n+1$, there is a star transposition ordering of all $(n+1,n+1)$-combinations such that the corresponding flip sequence is $\alpha=(\alpha_0,\alpha_1,\ldots,\alpha_{2n})$, and each block~$\alpha_i$ is obtained from the initial block~$\alpha_0$ by element-wise addition of~$i\cdot s$ modulo~$2n+1$ for all $i=1,\ldots,2n$. 
\end{theorem}

In Section~\ref{sec:ideas} below we explain why the condition on~$s$ to be coprime to $2n+1$ is necessary and cannot be omitted from Theorem~\ref{thm:star}.
Our proof of Theorem~\ref{thm:star} is constructive, and translates into an algorithm that generates $(n+1,n+1)$-combinations by star transpositions efficiently.

\begin{theorem}
\label{thm:algo}
There is an algorithm that computes, for any $n\geq 1$ and $1\leq s\leq 2n$ that is coprime to $2n+1$, a star transposition ordering of all $(n+1,n+1)$-combinations as in Theorem~\ref{thm:star}, with running time~$\cO(n)$ for each generated combination, using $\cO(n)$ memory in total.
\end{theorem}

The initial combination can be chosen arbitrarily in our algorithm, and the initialization time is~$\mathcal{O}(n^2)$.
We implemented this algorithm in C++ and made it available for download and experimentation on the Combinatorial Object Server website~\cite{cos_middle}.
It is open whether our algorithm can be improved to generate each combination in time~$\cO(1)$ instead of~$\cO(n)$.
Also, our algorithm is rather complex, and it remains a challenge to come up with a simple and efficient algorithm, either for the original middle levels conjecture or for Knuth's stronger variant.
We offer a reward of $2^5$~EUR for an algorithm that can be implemented with at most $4000$~ASCII characters of C++ code.

\subsection{Related work}
\label{sec:related}

Let us briefly discuss several results and open questions that are closely related to our work.

\begin{figure}
\makebox[0cm]{ 
\begin{tabular}{cc}
\includegraphics{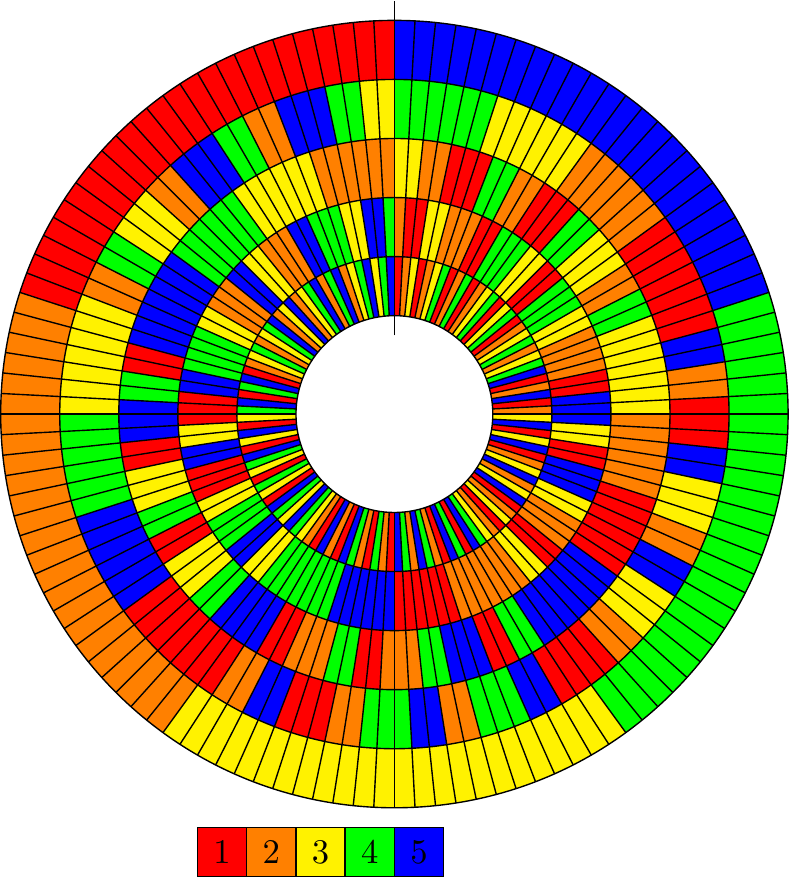} &
\includegraphics{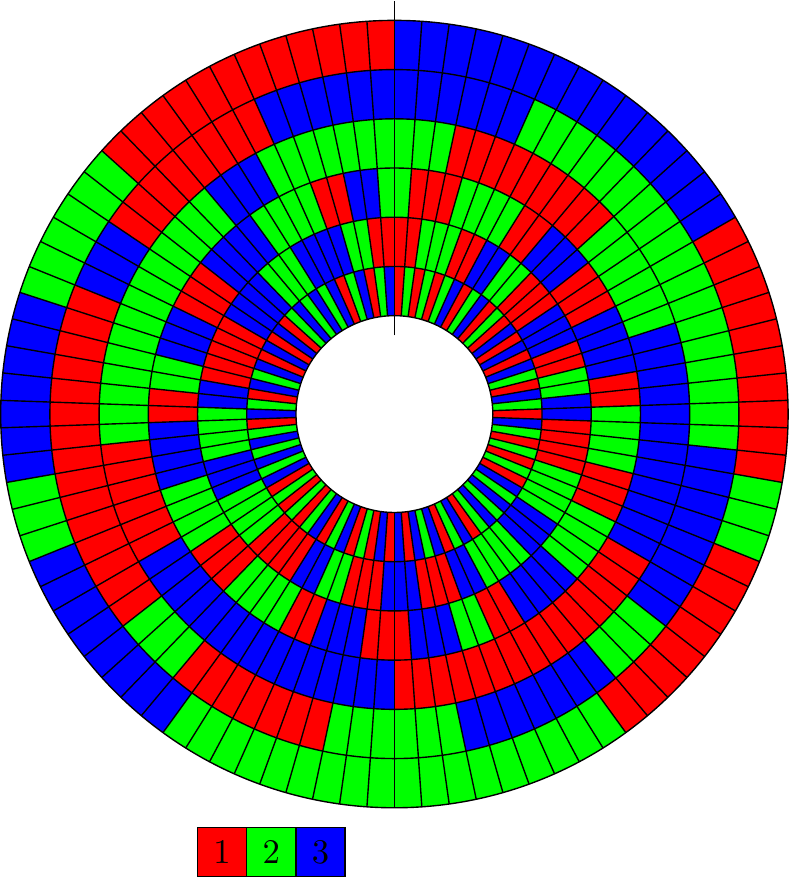} \\
(a) $(a_1,\ldots,a_5)=(1,1,1,1,1)$ & (b) $(a_1,a_2,a_3)=(2,2,2)$
\end{tabular}
}
\caption{Star transposition Gray codes for (a) permutations of length $n=5$ and (b) multiset permutations with frequencies $(a_1,a_2,a_3)=(2,2,2)$.
Permutations are arranged in clockwise order, starting at 12 o'clock, with the first entry on the inner track, and the last entry on the outer track.
The color on the inner track alternates in each step.}
\label{fig:star}
\end{figure}

\subsubsection{Star transpositions for permutations}

In the literature, star transposition orderings of objects other than combinations have been studied intensively.
A classical result, discovered independently by Kompel'maher and Liskovec \cite{MR0498276} and Slater~\cite{MR504868}, is that all \emph{permutations} of~$[n]$ can be generated  (cyclically) by star transpositions, i.e., in every step, the first entry of the permutation is swapped with a later entry.
An efficient algorithm for this task, was found by Ehrlich, and is described as Algorithm~E in Knuth's book~\cite[Section~7.2.1.2]{MR3444818} (see also~\cite{DBLP:journals/endm/ShenW13}).
It has a slight defect of not always guaranteeing a cyclic ordering, i.e., the first and last permutation do not always differ in a star transposition.
The flip sequence produced by Ehrlich's algorithm for $n=4$ with starting permutation~$1234$ is $\alpha=121213212123121213212123$ (indices are again 0-based), and it gives a cyclic ordering.
The ordering of permutations resulting from this algorithm for $n=5$ is shown in Figure~\ref{fig:star}~(a), and it is not cyclic.

In fact, the first two of the aforementioned papers prove that permutations can be generated using \emph{any} set of transpositions that forms a connected graph as a basis, such as star transpositions or adjacent transpositions.
Tchuente~\cite{MR683982} proved more generally that the graph of permutations under such transpositions is Hamilton-laceable for $n\geq 4$, i.e., it has a Hamilton path between any two permutations with opposite signs.
These results are special cases of a more general open problem that asks whether the Cayley graph of the symmetric group on any set of generators has a Hamilton cycle, which in turn is a special case of an even more general conjecture attributed to Lov{\'a}sz about Hamilton cycles in arbitrary connected Cayley graphs~\cite{MR0263646}; see \cite{MR1201997,MR2548567,MR2548568} for more references in this direction.

\subsubsection{Multiset permutations}
\label{sec:related-multiset}

Combinations and permutations are both special cases of \emph{multiset permutations}.
A multiset permutation is a string over the alphabet $\{1,\ldots,n\}$ with a given frequency distribution $(a_1,\ldots,a_n)$, i.e., the symbol~$i$ appears exactly $a_i$ times in the string for all $i=1,\ldots,n$.
For instance, $4412113$ is a multiset permutation for $n=4$ with frequencies $(a_1,a_2,a_3,a_4)=(3,1,1,2)$.
Clearly, for $n=2$ multiset permutations contain only two symbols, so they encode combinations.
On the other hand, for $(a_1,\ldots,a_n)=(1,\ldots,1)$ every symbol appears exactly once, so such multiset permutations are simply permutations of~$[n]$.
Shen and Williams~\cite{shen_williams_2021} raised a beautiful and brave conjecture which asserts that for any integers~$n\geq 2$ and $k\geq 1$, all multiset permutations over the alphabet $\{1,\ldots,n\}$ with frequencies $(a_1,\ldots,a_n)=(k,\ldots,k)$ can be generated by star transpositions.
The only confirmed general cases for this conjecture are the case $n=2$ (the middle levels conjecture) and the case $k=1$ (by the results on permutations mentioned before).
In addition, Shen and Williams gave a solution for $(n,k)=(3,2)$ in their paper, which is shown in Figure~\ref{fig:star}~(b).
We also verified the next two small cases $(n,k)=(3,3)$ and $(n,k)=(4,2)$ by computer.
Moreover, the techniques developed in this paper allowed us to solve Shen and Williams' conjecture for the cases $k\in\{2,3,4\}$ and $n\geq 2$~\cite{gregor-merino-muetze:21}.

\subsubsection{Other algorithms for combination generation}

We also know many efficient algorithms for generating combinations that do not use star transpositions.
Tang and Liu \cite{MR0349274} first showed that all $(k,\ell)$-combinations can be generated by transpositions of a 0-bit with a 1-bit, where neither of the swapped bits is required to be at the boundary.
Their construction arises from restricting the classical binary reflected Gray code to bitstrings with~$k$ 1s, and was turned into a constant-time algorithm by Bitner, Ehrlich, and Reingold~\cite{MR0424386}.
Eades and McKay~\cite{MR782221} showed that $(k,\ell)$-combinations can be generated by transpositions of the form $00\cdots 01\leftrightarrow 10\cdots 00$, i.e., the bits between the swapped~0 and~1 are all~0s.
We can think of this as an algorithm that plays all possible combinations of~$k$ keys out of~$n=k+\ell$ available keys on a piano, without ever crossing any fingers.
Jenkyns and McCarthy~\cite{MR1352777} showed that we can restrict the allowed swaps further and only allow transpositions of the form $01\leftrightarrow 10$ or $001\leftrightarrow 100$; see also~\cite{MR995888}.
Eades, Hickey and Read~\cite{MR821383} and independently Buck and Wiedemann~\cite{MR737262} proved that all $(k,\ell)$-combinations can be generated by using only adjacent transpositions $01\leftrightarrow 10$ if and only if $k\leq 1$ or $\ell\leq 1$ or $k\cdot\ell$ is odd.
An efficient algorithm for this problem was given by Ruskey~\cite{MR936104}.

Another elegant and efficient method for generating combinations based on prefix rotations was described by Ruskey and Williams~\cite{MR2548545}.
An interesting open question in this context is whether all $(k,\ell)$-combinations can be generated by prefix reversals, i.e., in each step, a prefix of the bitstring representation is reversed to obtain the next combination.
Such orderings can be constructed easily for the cases $k\in\{1,2\}$ or $\ell\in\{1,2\}$, but no general construction is known.

\subsection{Proof ideas}
\label{sec:ideas}

In this section we outline the main ideas used in our proof of Theorem~\ref{thm:star}, and in its algorithmization stated in Theorem~\ref{thm:algo}.
We also highlight the new contributions of our work compared to previous papers.

\subsubsection{Flipping through necklaces}
\label{sec:idea-paths}

We start noting that the first bit of a star transposition ordering of combinations alternates in each step (see~Figure~\ref{fig:c44}), so we may simply omit it, and obtain an ordering of all bitstrings of length~$2n+1$ with either exactly~$n$ or~$n+1$ many~1s, such that in every step, a single bit is flipped.
Observe that from a flip sequence~$\alpha_0$ that satisfies the conditions of Theorem~\ref{thm:star} we can uniquely reconstruct the first bitstring of each block, by considering for each $i\in\{1,\ldots,2n+1\}$ the parity of the position of first occurrence of the number~$i$ in~$\alpha$ starting from this block, where indices of~$\alpha$ start with~0.
For example, for the flip sequence~$\alpha$ defined by $\alpha_0:=6253462135$ and $s=1$ (see the left column in Figure~\ref{fig:c44}), the first occurrence of the numbers $i=1,\ldots,7$ in~$\alpha$ is at positions $7,1,3,4,2,0,10$ and the parity of those numbers is the starting bitstring~$1110000$.
In this example, 7 is the first entry of the second block~$\alpha_0+1$, which is at position~10 overall.
As any two consecutive blocks of the flip sequence differ by addition of~$s$, the first bitstrings of the blocks differ by cyclic rotation by~$s$ positions.
From this we obtain that the flip sequence~$\alpha_0$ that operates on these strings of length~$2n+1$ must visit every equivalence class of bitstrings under rotation exactly once, and it must return to a bitstring from the same equivalence class as the starting bitstring.
It also follows that the compression factor $2n+1$ in Knuth's problem is best possible.
Formally, a \emph{necklace~$\neck{x}$} for a bitstring~$x$ is the set of all strings that are obtained as cyclic rotations of~$x$.
For example, the necklace of~$x=1110000$ is $\neck{x}=\{1110000,1100001,1000011,0000111,0001110,0011100,0111000\}$, and there are 10 necklaces for $n=3$, namely $\neck{1110000},\neck{1101000},\neck{1100100},\neck{1100010},\neck{1010100}$ and their complements.
In Figure~\ref{fig:c44}, each of the shown flip sequences~$\alpha_0$ visits exactly one representative from each of the 10~necklaces, and it starts and ends with a bitstring from the same necklace.
Specifically, for each of the solutions for $s=1,\ldots,6$ shown in the figure, the corresponding flip sequence starts at some bitstring~$x$, and after 10 flips arrives at a right-shift of~$x$ by~$s$ positions (first bit omitted).
For example, the flip sequence for $s=1$ starts at~1110000 and arrives at~0111000.
We refer to~$s$ as the \emph{shift} of the flip sequence~$\alpha_0$.
The crucial observation is that every string of length~$2n+1$ with either exactly~$n$ or~$n+1$ many~1s has exactly $2n+1$ distinct cyclic rotations, i.e., every necklace has the same size~$2n+1$.
Consequently, we may apply the shifted flip sequences~$\alpha_i$, obtained from $\alpha_0$ by element-wise addition of~$i\cdot s$ modulo~$2n+1$, one after the other for $i=1,\ldots,2n$, and this will produce the desired star transposition ordering of all~$(n+1,n+1)$-combinations.
Clearly, for this to work the shift~$s$ and~$2n+1$ must be coprime, otherwise we will return to the starting bitstring prematurely before exhausting all bitstrings from the necklaces.
In particular, if $s=0$ we return to the starting bitstring after applying only~$\alpha_0$.
For instance, applying the flip sequence $\alpha_0=6241247617$ to the starting string~$1110000$ visits every necklace exactly once, but returns to the exact same bitstring after 10 flips (every bit is flipped an even number of times by~$\alpha_0$), so this flip sequence has shift~$s=0$.
This explains the condition stated in Theorem~\ref{thm:star} that $s$ and $2n+1$ must be coprime, which is necessary and cannot be omitted.

\subsubsection{Gluing flip sequences together}
\label{sec:idea-gluing}

To construct a flip sequence~$\alpha_0$ that visits every necklace exactly once and returns to the starting necklace, we first construct many disjoint shorter flip sequences that together visit every necklace exactly once.
These basic flip sequences are obtained from a simple bitflip rule based on Dyck words that is invariant under cyclic rotations.
In a second step, these basic flip sequences are glued together to a single flip sequence by local modifications.

\begin{wrapfigure}{r}{0.55\textwidth}
\includegraphics[page=1]{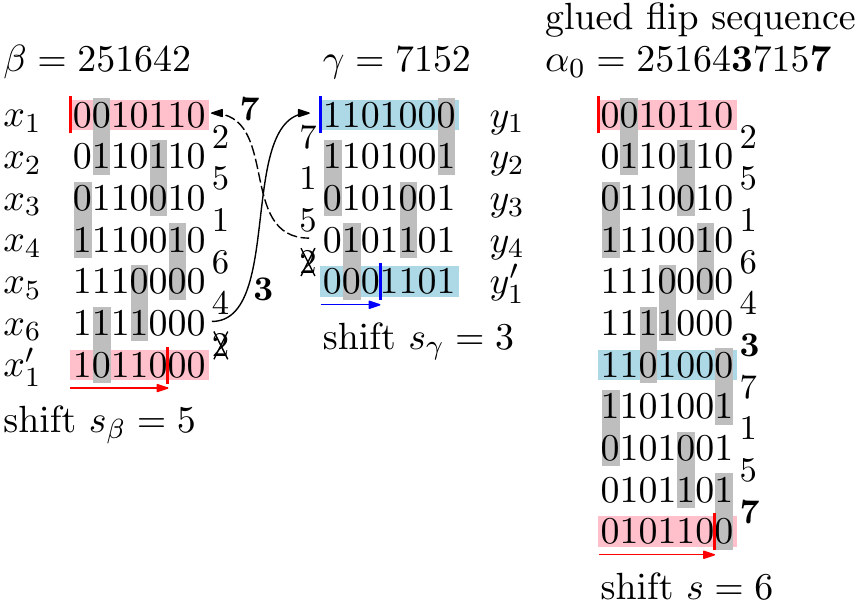}
\caption{Gluing of flip sequences.}
\label{fig:idea-gluing}
\end{wrapfigure}
Figure~\ref{fig:idea-gluing} illustrates this approach for $n=3$.
We may start at the bitstring~$x_1=0010110$ and apply the flip sequence~$\beta=251642$ to generate a sequence of bitstrings~$x_1,x_2,\ldots,x_6,x_1'$, and the final bitstring~$x_1'=1011000$ belongs to the same necklace as~$x_1$, and it differs from~$x_1$ by a right-rotation of~$s_\beta=5$.
Similarly, from the bitstring~$y_1=1101000$ we may apply the flip sequence~$\gamma=7152$ to generate a sequence of bitstrings~$y_1,y_2,y_3,y_4,y_1'$, and the final bitstring $y_1'=0001101$ belongs to the same necklace as~$y_1$, and it differs from~$y_1$ by a right-rotation of~$s_\gamma=3$.
The sets of necklaces $\neck{x_1},\ldots,\neck{x_6}$ and $\neck{y_1},\ldots,\neck{y_4}$ visited by the two sequences are disjoint, and every necklace is contained in exactly one sequence.
As $x_6$ differs from~$y_1$ by a single flip of the 3rd bit, and $y_4$ differs from a cyclic rotation of~$x_1$ by a single flip of the 7th bit, we may replace the last entry of~$\beta$ with~3 and the last entry of~$\gamma$ by~7, and concatenate the resulting sequences, yielding the flip sequence $\alpha_0=25164{\bf 3}715{\bf 7}$, which visits every necklace exactly once.
Moreover, the shift of the resulting flip sequence~$\alpha_0$ turns out to be $s=6$ (which is coprime to $2n+1=7$).
In this example, the gluing of the two flip sequences is achieved by taking their symmetric difference with a 4-cycle of necklaces $(\neck{x_1},\neck{x_6},\neck{y_1},\neck{y_4})$, removing one flip from each of the two original sequences, and adding two flips to transition back and forth between them.
In our proof later, for technical reasons the gluing of flip sequences uses slightly more complicated structures, namely 6-cycles of necklaces, albeit with the same effect of gluing two smaller flip sequences together in each step.
One of the major technical hurdles is to ensure that several of these gluing steps do not interfere with each other.

\begin{wrapfigure}{r}{0.48\textwidth}
\includegraphics[page=2]{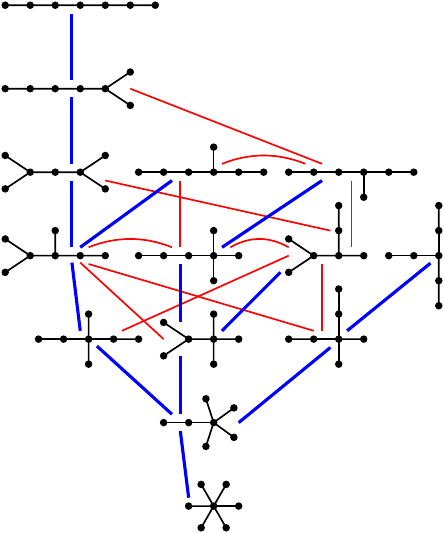}
\caption{Auxiliary graph on plane trees with $n=6$ edges.
Edges correspond to local modification operations on the plane trees.
The bold edges show a spanning tree in the graph.}
\label{fig:idea-aux}
\end{wrapfigure}
The benefit of the gluing approach is that Knuth's generation problem translates into the problem of finding a spanning tree in a suitably defined auxiliary graph:
Specifically, the nodes of this auxiliary graph are the basic flip sequences we start with, and the edges correspond to the gluing 6-cycles that join two of them together.
A spanning tree in the auxiliary graph corresponds to a collection of gluing 6-cycles that glue together all basic flip sequences to a single flip sequence~$\alpha_0$ with the desired properties.
We show that our basic flip sequences can be interpreted combinatorially as the plane trees with $n$ edges via a bijective correspondence (in particular, the number of basic flip sequences is given by the number of these trees).
As a consequence, the aforementioned auxiliary graph has all plane trees with $n$ edges as its nodes; see Figure~\ref{fig:idea-aux}.
Moreover, a gluing 6-cycle can be interpreted as a local modification of the two plane trees involved.
Specifically, a leaf of one plane tree is removed and reattached to a neighbor of the original attachment vertex.
A spanning tree in the auxiliary graph can be obtained by choosing a minimal set of gluing 6-cycles such that the corresponding local modifications allow transforming any two plane trees into each other (see the bold edges in Figure~\ref{fig:idea-aux}).

\subsubsection{Controlling the shift}
\label{sec:idea-shift}

The next step is to control the shift of the resulting flip sequence~$\alpha_0$.
Without controlling the overall shift, we may end up with a shift~$s=0$, or more generally, with a shift~$s$ that is not coprime to~$2n+1$, which is useless as explained above.

The key observation here is that it is enough to construct a flip sequence~$\alpha_0$ for \emph{one particular} shift~$s$ that is coprime to~$2n+1$.
With this solution in hand, a simple transformation gives us \emph{every} shift~$\hs$ that is coprime to~$2n+1$.
Indeed, this transformation is given by multiplying all entries of~$\alpha_0$ by $s^{-1}\hs$ (modulo~$2n+1$), where $s^{-1}$ is the multiplicative inverse of~$s$.
In other words, the Gray code with shift~$\hs$ is obtained from the one with shift~$s$ by permuting columns according to the rule $i\mapsto s^{-1}\hs i$ for all $i=0,\ldots,2n+1$.
In fact, among the first six solutions for $s=1,\ldots,6$ shown in Figure~\ref{fig:c44}, any two are related by such a permutation of columns.
For example, from the flip sequence~$\alpha_0=6253462135$ with shift~$s=1$, we obtain $2\alpha_0=12\;4\;10\;6\;8\;12\;4\;2\;6\;10=5436154263$, which is the solution shown with shift~$s=2$.
This simple but powerful \emph{scaling trick} links our general version of the problem, which allows for an arbitrary shift~$s$ coprime to~$2n+1$, to Knuth's problem, which specifically asks for a shift of~$s=1$.

Still, we are left with the problem of constructing a suitable flip sequence for \emph{one particular} shift that is coprime to~$2n+1$, which we solve by a technique that we call \emph{switching}.
To illustrate the idea behind it, consider the solutions for $s=1$ and the rightmost solution for $s=6$ in Figure~\ref{fig:c44}.
Observe that both of these flip sequences visit the same necklaces in the same order, but they only differ in the chosen necklace representatives, yielding different shift values.
\begin{wrapfigure}{r}{0.4\textwidth}
\includegraphics[page=3]{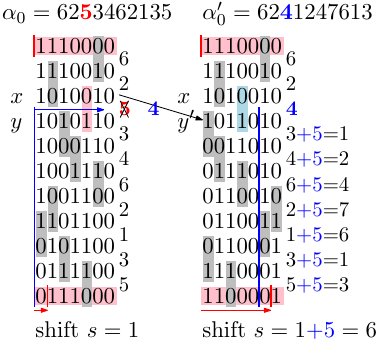}
\caption{Controlling the shift value of flip sequences by switching.}
\label{fig:idea-switch}
\end{wrapfigure}
Specifically, after the first two flips, both of these flip sequences visit the bitstring~$x=1010010$; see Figure~\ref{fig:idea-switch}.
In the third step, one sequence flips the 5th bit of~$x$, yielding the string~$y=1010110$, while the other sequence flips the 4th bit of~$x$, yielding the string~$y'=1011010$, which only differ by cyclic right-rotation by 5~steps.
After this flip, the entries of both flip sequences differ only by the constant~5, and consequently, their shift values differ only by~5 ($s=1$ and $s=6$).
We refer to a bitstring~$x$ that allows flipping two distinct bits to reach two bitstrings~$y$ and~$y'$ in the same necklace $\neck{y}=\neck{y'}$ as a \emph{switch}.
In our approach, we systematically construct switches that allow modifying the shift of flip sequences in a controlled way, so as to achieve some shift value that is coprime to~$2n+1$, while preserving the order of the visited necklaces.

\subsubsection{Efficient computation}
\label{sec:idea-algo}

The biggest obstacle in translating our constructive proof into an efficient algorithm is to quickly compute the resulting shift~$s$ of the flip sequence~$\alpha_0$ that results from the gluing process.
Ideally, gluing two flip sequences with shifts~$s_\beta$ and~$s_\gamma$ should give a flip sequence with shift~$s_\beta+s_\gamma$.

If this additivity would hold, it would allow us to compute the overall shift simply as the sum of shifts of the basic flip sequences, and this sum can be evaluated explicitly to be $s=C_n$, the $n$th Catalan number.
For the algorithm, it is crucial to know the initial shift value~$s$, because only then do we know by how much~$s$ needs to be modified through switching and scaling to achieve the shift value that is specified in the input of the algorithm.
For instance, for $n=16$ we have $s=C_n=35{,}357{,}670=18 \pmod{2n+1}$, which happens not to be coprime to $2n+1=33$, so in this case we need to first apply switching to achieve a coprime shift such as $s=17$, and then apply scaling to achieve the shift value that is specified as input of the algorithm.

In the example from Section~\ref{sec:idea-gluing}, we had $s_\beta=5$ and $s_\gamma=3$, and an overall shift of $s=6$ after the gluing, which is different from $s_\beta+s_\gamma=5+3=8=1$ (modulo~7), so the desired additivity of shifts under gluing does not hold in this example.
In fact, guaranteeing that the shifts behave additively under gluing requires substantial effort, and is achieved by constructing a particularly nice spanning tree in the aforementioned auxiliary graph.

\subsubsection{Comparison to previous work}
\label{sec:comparison}

The general idea of gluing, and the resulting reduction to a spanning tree problem, is very natural and variations of it have been used successfully in several papers before (see e.g.~\cite{MR1308693,MR2925746,MR3599935,DBLP:conf/soda/SawadaW18,MR3826304,DBLP:conf/icalp/GregorMM20}).
As mentioned before, a flip sequence~$\alpha_0$ satisfying the requirements of Knuth's conjecture encodes the entire flip sequence~$\alpha$ by a factor of~$2n+1$ more compactly.
This requires us to perform all the aforementioned steps, i.e., the construction of basic flip sequences and gluing them together, on \emph{necklaces} rather than on \emph{bitstrings}, as was previously done in~\cite{MR3483129,gregor-muetze-nummenpalo:18}, which creates many additional technical complications.
The key innovation of our paper is to develop these necklace-based constructions, and in particular, to control the shift of the resulting flip sequence~$\alpha_0$, using the techniques presented in Section~\ref{sec:idea-shift}.
The flexibility that these methods have will certainly yield other interesting applications in the future.
As evidence for that, recall from Section~\ref{sec:related-multiset} that these techniques allow us to construct solutions for Shen and Williams' conjecture on multiset permutations with frequencies $(a_1,\ldots,a_n)=(k,\ldots,k)$ for the cases $k\in\{2,3,4\}$ and $n\geq 2$~\cite{gregor-merino-muetze:21}.
We believe that with more work, they can help to settle this problem in full generality.

\subsection{Outline of this paper}

In Section~\ref{sec:prel} we introduce some terminology and notation that is used throughout the paper.
In Section~\ref{sec:paths} we explain the construction of the basic flip sequences that together traverse all necklaces.
In Section~\ref{sec:gluing} we discuss the gluing technique that we use to join the basic flip sequences together to a single flip sequence.
These two ingredients are combined in Section~\ref{sec:tree}, where we reduce Knuth's problem to a spanning tree problem in a suitably defined auxiliary graph.
In Section~\ref{sec:switch} we discuss the switching technique, and based on this we present the proof of Theorem~\ref{thm:star}.
In Section~\ref{sec:algo} we redefine the spanning tree in the auxiliary graph, so that shift values behave additively under gluing, which is essential for our algorithms, and we present the proof of Theorem~\ref{thm:algo}.

\section{Preliminaries}
\label{sec:prel}

In this section we introduce some definitions and easy observations that we will use repeatedly in the subsequent sections.

\subsection{Binary strings and necklaces}

For $n\geq 1$, we let $A_n$ and~$B_n$ denote all bitstrings of length $2n+1$ with exactly $n$ or $n+1$ many 1s, respectively.
The \emph{middle levels graph~$M_n$} has $A_n\cup B_n$ as its vertex set, and an edge between any two bitstrings that differ in a single bit.
As mentioned in Section~\ref{sec:idea-paths} before, in a star transposition ordering of all $(n+1,n+1)$-combinations, the first bit alternates between 0 and~1 in each step; see Figure~\ref{fig:c44}.
Consequently, omitting the first bit, we see a Hamilton cycle in the graph~$M_n$ on the remaining $2n+1$ bits.
We index these bit positions by $1,\ldots,2n+1$, and we consider all indices modulo $2n+1$ throughout this paper, with $1,\ldots,2n+1$ as residue class respresentatives (rather than $0,\ldots,2n$).

The empty bitstring is denoted by~$\varepsilon$.
For any bitstring~$x$ and any integer $i\geq 0$, we let $x^i$ denote the bitstring obtained by concatenating $i$ copies of~$x$.
Also, we let $\sigma^i(x)$ denote the bitstring obtained from~$x$ by cyclic left-rotation by~$i$ positions.
As mentioned before, the \emph{necklace of $x$}, denoted~$\neck{x}$, is defined as the set of all bitstrings obtained from~$x$ by cyclic rotations, i.e., we have $\neck{x}=\{\sigma^i(x)\mid i\geq 0\}$.
For example, for $x=11000\in A_2$ we have $\neck{x}=\{11000,10001,00011,00110,01100\}$.
The \emph{necklace graph~$N_n$} has as vertex set all~$\neck{x}$, $x\in A_n\cup B_n$, and an edge between any two necklaces~$\neck{x}$ and~$\neck{y}$ for which~$x$ and~$y$ differ in a single bit; see Figure~\ref{fig:n3}.
Observe that~$N_n$ arises as the quotient of~$M_n$ under the equivalence relation of rotating bitstrings cyclically.
Note that for a given necklace~$\neck{x}$, there may be two distinct bits in the representative~$x$ that reach the same necklace~$\neck{y}$, a fact that we will exploit heavily in Section~\ref{sec:switch}.
Nonetheless, we still consider $N_n$ as a simple graph, and so not all vertices of~$N_n$ have the same degree.
As mentioned before, for any $x\in A_n\cup B_n$, the necklace $\neck{x}$ has size~$2n+1$, i.e., the graph $N_n$ has by a factor $2n+1$ fewer vertices than the graph~$M_n$.
To define the flip sequence~$\alpha_0$ in Theorem~\ref{thm:star}, we will construct a Hamilton cycle in~$N_n$.
This is achieved using paths in the middle levels graph~$M_n$ that have the following property:
A path $P=(x_1,\ldots,x_k)$ in~$M_n$ is called \emph{periodic}, if one can flip a single bit in~$x_k$ to obtain a vertex~$x_{k+1}$ that satisfies $\neck{x_{k+1}}=\neck{x_1}$.

\begin{figure}
\includegraphics{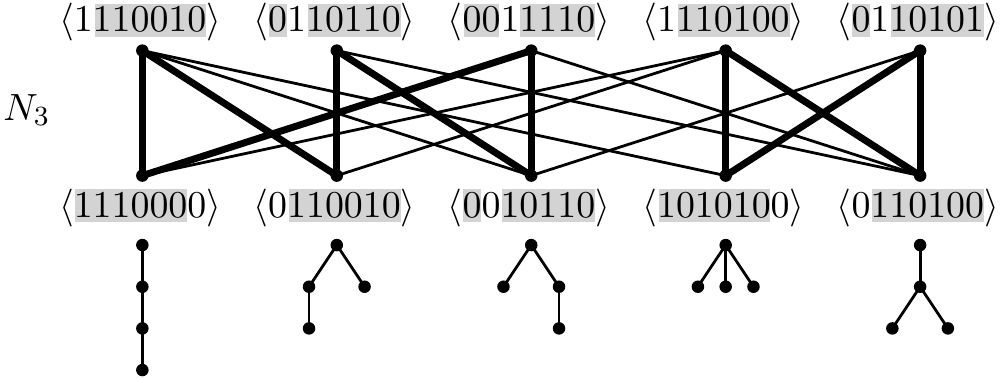}
\caption{The necklace graph $N_3$, with the cycle factor $\cF_3$ highlighted.
The Dyck words in the necklace representatives are highlighted by gray boxes, and the corresponding rooted trees~$t(x)$ for all $\neck{x}$, $x\in A_3$, are shown at the bottom.}
\label{fig:n3}
\end{figure}

\subsection{Operations on sequences}

For any sequence $x=(x_1,\ldots,x_k)$, we let $|x|:=k$ denote the length of the sequence.
For any sequence of integers $x=(x_1,\ldots,x_k)$ and any integer $a$, we define $x+a:=(x_1+a,\ldots,x_k+a)$.
For any sequence of bitstrings $x=(x_1,\ldots,x_k)$, we define $\neck{x}:=(\neck{x_1},\ldots,\neck{x_k})$ and $\sigma^i(x):=(\sigma^i(x_1),\ldots,\sigma^i(x_k))$.

\subsection{Dyck words, rooted trees, and plane trees}

The \emph{excess} of a bitstring~$x$ is the number of~1s minus the number of~0s in~$x$.
If $x$ has excess~0 (i.e., it has the same number of 1s and 0s) and every prefix of~$x$ has non-negative excess, then we call $x$ a \emph{Dyck word}.
For $n\geq 0$, we write $D_n$ for the set of all Dyck words of length~$2n$.
Moreover, we define $D:=\bigcup_{n\geq 0}D_n$.

An (ordered) \emph{rooted tree} is a rooted tree with a specified left-to-right ordering for the children of each vertex.
The trees we consider are unlabeled, i.e., their vertices are indistinguishable.
Every Dyck word~$x\in D_n$ can be interpreted as a rooted tree with $n$ edges as follows; see Figure~\ref{fig:rot}:
If $x=\varepsilon$, then this corresponds to the tree that has an isolated vertex as root.
If $x\neq \varepsilon$, then it can be written uniquely as $x=1\,u\,0\,v$ with $u,v\in D$.
\begin{wrapfigure}{r}{0.42\textwidth}
\includegraphics{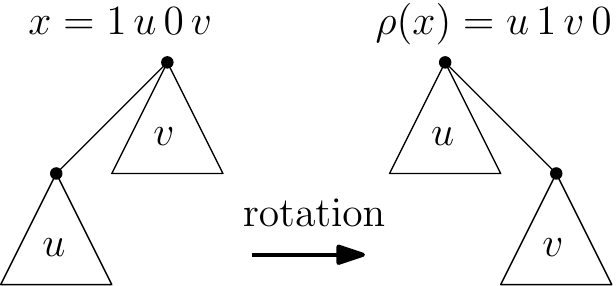}
\captionsetup{margin=3mm}
\caption{Interpretation of Dyck words as rooted trees and definition of tree rotation.}
\label{fig:rot}
\end{wrapfigure}
We then consider the trees~$L$ and~$R$ corresponding to~$u$ and~$v$, respectively, and the tree corresponding to~$x$ has $L$ rooted at the leftmost child of the root, and the edges from the root to all other children except the leftmost one, together with their subtrees, form the tree~$R$.
This mapping is clearly a bijection between~$D_n$ and rooted trees with $n$ edges.

Given a rooted tree~$x\neq \varepsilon$, let $\rho(x)$ denote the tree obtained by rotating the tree to the right, which corresponds to designating the leftmost child of the root of~$x$ as the new root in~$\rho(x)$.
In terms of bitstrings, if $x=1\,u\,0\,v$ with $u,v\in D$, then $\rho(x)=u\,1\,v\,0$; see Figure~\ref{fig:rot}.

A \emph{plane tree} is a tree with a specified cyclic ordering of the neighbors of each vertex.
We think of it as a tree embedded in the plane, where the cyclic ordering is the ordering of the neighbors of each vertex in counterclockwise (ccw) direction around the vertex.
For $n\geq 1$, we let $T_n$ denote the set of all plane trees with $n$ edges.
For any rooted tree~$x$, we let $[x]$ denote the set of all rooted trees obtained from~$x$ by rotation, i.e., we define $[x]:=\{\rho^i(x)\mid i\geq 0\}$, and this can be interpreted as the plane tree underlying~$x$, obtained by `forgetting' the root.
We also define $\lambda(x):=|[x]|$, and for the plane tree $T=[x]$ we define $\lambda(T):=\lambda(x)$.
Note that $\lambda(x)=\min\{i\geq 1\mid \rho^i(x)=x\}$, which implies that the choice of representative from~$[x]$ in the definition of~$\lambda(T)$ is irrelevant, i.e., $\lambda(T)$ is well defined.

For any plane tree~$T$ and any of its edges~$(a,b)$, we let $T^{(a,b)}$ denote the rooted tree obtained from~$T$ by designating~$a$ as root such that~$b$ is its leftmost child.
Moreover, we let $T^{(a,b)-}$ denote the rooted tree obtained from~$T^{(a,b)}$ by removing all children and their subtrees of the root except the leftmost one, and we let $T^{(a,b)--}$ denote the tree obtained as the subtree of~$T^{(a,b)-}$ that is rooted at the vertex~$b$.
Given a vertex~$a$ of~$T$, consider each neighbor~$b_i$, $i=1,\ldots,k$, of~$a$ and define the rooted tree $t_i:=T^{(a,b_i)-}$.
We refer to the rooted trees $t_1,\ldots,t_k$ as the \emph{$a$-subtrees} of~$T$.
Note that we have $T=[(t_1,\ldots,t_k)]$, where $(t_1,\ldots,t_k)$ is the rooted tree obtained from gluing together the trees $t_1,\ldots,t_k$ from left to right in this order at their roots.
In terms of bitstrings, $(t_1,\ldots,t_k)$ is the bitstring obtained from concatenating the bitstring representations~$t_1,\ldots,t_k$ of those trees.

A \emph{leaf} of a rooted or plane tree is a vertex with degree~1.
In particular, the root of a rooted tree is a leaf if and only if it has exactly one child.
We say that a leaf of a tree is \emph{thin} if its unique neighbor in the tree has degree at most~2, otherwise we call the leaf \emph{thick}.
For any rooted or plane tree~$T$, we write~$v(T)$ and~$e(T)$ for the number of vertices or edges of~$T$, respectively.

\subsection{Centroids and potential}

Given a (rooted or plane) tree~$T$, the \emph{potential of a vertex~$a$}, denoted $\varphi(a)$, is the sum of distances from~$a$ to all other vertices in~$T$.
The \emph{potential of the tree~$T$}, denoted $\varphi(T)$, is the minimum of~$\varphi(a)$ over all vertices~$a$ of~$T$.
A \emph{centroid} of~$T$ is a vertex~$a$ with $\varphi(a)=\varphi(T)$.

The notion of centroids was introduced by Jordan~\cite{MR1579443}, and there are several equivalent characterizations for it.
For example, a centroid of an $n$-vertex tree is a vertex whose removal splits the tree into subtrees with at most $n/2$ vertices each.
Our first lemma captures further important properties of a centroid of a tree.
These properties are also well known (see e.g.~\cite{MR396308} or~\cite[Section~2.3.4.4]{MR3077152}), and we restate them here for convenience.

\begin{lemma}
\label{lem:centroid}
Let $T$ be a plane tree.
For every edge~$(a,b)$ of~$T$, we have
\begin{equation}
\label{eq:pot-diff}
\varphi(b)-\varphi(a)=e(T^{(b,a)--})-e(T^{(a,b)--}).
\end{equation}
As a consequence, $T$ has either one centroid or two adjacent centroids.
If $e(T)$ is even, then $T$ has exactly one centroid.
\end{lemma}

\begin{proof}
Comparing the potentials of $b$ and $a$, note that the distance of every vertex in~$T^{(b,a)--}$ to~$b$ differs by~$+1$ from its distance to~$a$.
Conversely, the distance of every vertex in~$T^{(a,b)--}$ to~$b$ differs by~$-1$ from its distance to~$a$.
Combining these observations shows that $\varphi(b)=\varphi(a)+v(T^{(b,a)--})-v(T^{(a,b)--})$.
Using that $v(T)=e(T)+1$ for both trees $T\in\{T^{(b,a)--},T^{(a,b)--}\}$, we obtain~\eqref{eq:pot-diff}.
Consider any path between two leaves of~$T$.
By~\eqref{eq:pot-diff}, the sequence of potential differences along this path forms a strictly decreasing sequence.
It follows that $T$ has either one or two centroids, and if there are two, then they must be adjacent in~$T$.
Moreover, if there are two centroids~$a$ and~$b$, then we must have $\varphi(b)-\varphi(a)=0$ along the edge $(a,b)$ of~$T$, and then \eqref{eq:pot-diff} implies that $e(T^{(b,a)--})=e(T^{(a,b)--})$, i.e., $e(T)=e(T^{(b,a)--})+e(T^{(a,b)--})+1=2e(T^{(a,b)--})+1$ is odd.
\end{proof}

The next lemma describes possible values that the parameter $\lambda(T)$ can take for a plane tree~$T$.

\begin{lemma}
\label{lem:lambda}
Let $T\in T_n$ be a plane tree with $n\geq 1$ edges.
Then $\lambda(T)$ is a divisor of~$2n$.
If $T$ has a unique centroid, then $\lambda(T)$ is even, and if $T$ has two centroids, then $\lambda(T)=2n$ if $n$ is even, and $\lambda(T)\in\{n,2n\}$ if $n$ is odd.
Moreover, for $n\geq 4$ and any even divisor~$k$ of~$2n$ or for~$k=n$ there is a plane tree~$T$ with $\lambda(T)=k$.
\end{lemma}

\begin{proof}
Let $x$ be a rooted tree such that $T=[x]$.
Note that $x=T^{(a,b)}$ for some edge~$(a,b)$ of~$T$.
As there are at most $2n$ choices for the pair~$(a,b)$, we obtain that $\lambda(T)\leq 2n$.
If $\lambda:=\lambda(T)<2n$, then we clearly have $\rho^{\lambda}(x)=x$ and $\rho^{2n}(x)=x$.
If $\lambda$ was not a divisor of $2n$, then there are integers $c\geq 1$ and $1\leq d<\lambda$ such that $2n=c\lambda+d$, and together the previous two equations would yield $\rho^d(x)=x$, and then $d<\lambda$ would contradict the definition of~$\lambda$.
We conclude that $\lambda(T)$ is indeed a divisor of~$2n$.

Now suppose that $T$ has a unique centroid~$c$.
Consider the $c$-subtrees $t_1,\ldots,t_k$ of~$T$, i.e., we have $T=[(t_1,\ldots,t_k)]$.
As the quantity~$\lambda(T)$ counts the number of different ways of rooting~$T$, we may consider the additive contribution of each subtree~$t_i$, $i=1,\ldots,k$, to this count.
Specifically, this contribution is either~0 or~$2e(t_i)$, and it is~0 if and only if $T$ has rotational symmetry around~$c$ such that $t_i$ maps to~$t_j$ for some $j<i$.
This shows that $\lambda(T)$ is even.

It remains to consider the case that $T$ has two centroids~$c$ and~$c'$.
We define the rooted trees $t_c:=T^{(c',c)--}$ and $t_{c'}:=T^{(c,c')--}$.
If $n$ is even, then $n-1$ is odd, implying that $e(t_c)\neq e(t_{c'})$.
This yields in particular that $t_c\neq t_{c'}$, and so we have
\begin{equation}
\label{eq:lambdaT2n}
\lambda(T)=2e(t_c)+2e(t_{c'})+2=2\underbrace{(e(t_c)+e(t_{c'})+1)}_{=n}=2n.
\end{equation}
The $+2$ in~\eqref{eq:lambdaT2n} comes from the two rooted trees~$T^{(c,c')}$ and~$T^{(c',c)}$.
If $n$ is odd, then if $t_c\neq t_{c'}$ we also have \eqref{eq:lambdaT2n}, i.e., $\lambda(T)=2n$, whereas if $t_c=t_{c'}$, then we have
\begin{equation}
\label{eq:lambdaTn}
\lambda(T)=2e(t_c)+1=e(t_c)+e(t_{c'})+1=n.
\end{equation}
The $+1$ in \eqref{eq:lambdaTn} comes from the rooted tree~$T^{(c,c')}=T^{(c',c)}$.

To prove the last part of the lemma, let $k<n$ be an even divisor of $2n$, i.e., we have $2n=kd$ for some integer $d\geq 3$ and $k=2\ell$ for some integer $\ell\geq 1$.
Consider the plane tree~$T$ obtained by gluing together $d$ copies of the path on $\ell$ edges at a common centroid vertex.
This tree has $d\ell=dk/2=n$ edges and satisfies $\lambda(T)=2\ell=k$.
For $k=n$, the path $T$ on $n$ edges satisfies $\lambda(T)=n$.
For $k=2n$, the path $T$ on $n-1$ edges, with an extra edge appended to one of its interior vertices, satisfies $\lambda(T)=2n$ under the assumption that $n\geq 4$.
\end{proof}

\section{Periodic paths}
\label{sec:paths}

In this section we define the basic flip sequences that together visit every necklace exactly once, following the ideas outlined in Sections~\ref{sec:idea-paths} and~\ref{sec:idea-gluing}.
We thus obtain a so-called \emph{cycle factor} in the necklace graph~$N_n$, i.e., a collection of disjoint cycles that visit every vertex of the graph exactly once.
The key properties of these cycles that we will need later are summarized in Proposition~\ref{prop:Fn} below.
Each cycle in the necklace graph~$N_n$ is obtained from a periodic path in the middle levels graph~$M_n$, and we define these periodic paths via a simple bitflip rule based on Dyck words.

The following lemma is well known (see \cite[Problem~7]{bollobas_2006}).

\begin{lemma}
\label{lem:rot}
Let $n\geq 1$.
For any $x\in A_n$, there is a unique integer $\ell=\ell(x)$ with $0\leq\ell\leq 2n$ such that the first $2n$ bits of~$\sigma^\ell(x)$ are a Dyck word.
For any $y\in B_n$, there is a unique integer $\ell=\ell(y)$ with $0\leq\ell\leq 2n$ such that the last $2n$ bits of~$\sigma^\ell(y)$ are a Dyck word.
\end{lemma}

For any $x\in A_n$, we let $t(x)\in D_n$ denote the first $2n$ bits of $\sigma^\ell(x)$ where $\ell:=\ell(x)$, i.e., we have $\sigma^\ell(x)=t(x)\,0$.
Similarly, for any $y\in B_n$, we let $t(y)\in D_n$ denote the last $2n$ bits of $\sigma^\ell(y)$ where $\ell:=\ell(y)$, i.e., we have $\sigma^\ell(y)=1\,t(y)$.
In the following it will be crucial to consider the rooted trees corresponding to~$t(x)$ and~$t(y)$.
By Lemma~\ref{lem:rot}, every bitstring~$x\in A_n$ and every bitstring~$x\in B_n$ can be identified uniquely with the rooted tree~$t(x)$ and the integer~$\ell(x)$. 

In the following we introduce a bijection~$f$ on the set~$A_n\cup B_n$ that defines the basic flip sequences visiting every necklace exactly once.
Consider an $x\in A_n$ with $\ell(x)=0$, i.e., we have
\begin{subequations}
\label{eq:f}
\begin{equation}
\label{eq:xdec}
x=\underbrace{1\,u\,0\,v}_{t(x)}\,0
\end{equation}
with $u,v\in D$.
Then we define
\begin{equation}
\label{eq:ydec}
y:=f(x)=1\,\underbrace{u\,1\,v\,0}_{\rho(t(x))}\in B_n.
\end{equation}
Note that we have $\ell(y)=0$.
We then define
\begin{equation}
\label{eq:ffx0}
f(y)=f(f(x)):=0\,\underbrace{u\,1\,v\,0}_{\rho(t(x))}\in A_n.
\end{equation}
Note that we have $\ell(f(y))=1$.
We extend these definitions to all $x\in A_n\cup B_n$ by setting
\begin{equation}
\label{eq:ffx}
f(x):=\sigma^{-\ell}(f(\sigma^\ell(x))), \text{ where } \ell:=\ell(x).
\end{equation}
\end{subequations}
This definition is illustrated in Figure~\ref{fig:f4}.
It follows directly from these definitions that the mapping $f:A_n\cup B_n\rightarrow A_n\cup B_n$ is invertible.

\begin{figure}
\includegraphics{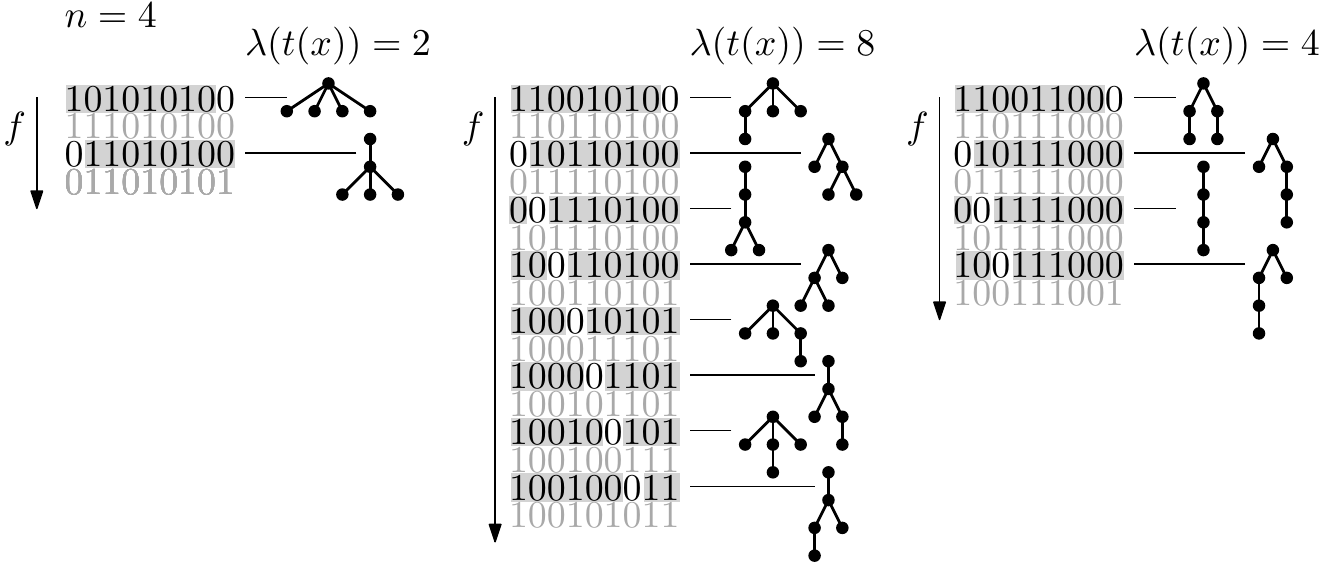}
\caption{Definition of the function~$f$ for $n=4$.
The Dyck words in the bitstrings are highlighted by gray boxes, and the corresponding rooted trees~$t(x)$ for the shown bitstrings $x\in A_4$, are displayed at the side.
Bitstrings from the set~$B_4$ are shown in gray.
Consecutive trees in each column differ by tree rotation.
As rotating the last tree yields the first tree in each column, the cycles in the necklace graph defined by~$f$ wrap around at the bottom and top.}
\label{fig:f4}
\end{figure}

From~\eqref{eq:f} we obtain that for all $x\in A_n$ we have
\begin{equation}
\label{eq:rot}
t(f(f(x)))=t(f(x))=\rho(t(x)).
\end{equation}
Moreover, \eqref{eq:xdec}--\eqref{eq:ffx0} show that if $\ell(x)=0$, then we have $\ell(f(x))=0$ and $\ell(f(f(x)))=1$.
From this and \eqref{eq:ffx}, we thus obtain for all $x\in A_n$ that
\begin{equation}
\label{eq:shift}
\ell(f(x))=\ell(x) \quad \text{and} \quad \ell(f(f(x)))=\ell(x)+1.
\end{equation}

In words, if we cyclically read a bitstring $x\in A_n$ starting at position $p:=\ell(x)+1$ and consider the first $2n$ bits as a rooted tree, ignoring the extra 0-bit, then the bitstring $f(x)$ read starting from position~$p$ has the extra 1-bit plus the rotated tree, and the bitstring $f(f(x))$ read starting from position~$p+1$ is the same rotated tree plus the extra 0-bit.

For any $x\in A_n\cup B_n$, we define the integer
\begin{equation}
\label{eq:kx}
\kappa(x):=\min\big\{i>0\mid \neck{f^i(x)}=\neck{x}\big\}.
\end{equation}

The quantity $\kappa(x)$ expresses how many times $f$ has to applied to~$x$ until we return to the same necklace as~$x$ for the first time.
The following lemma summarizes important properties of the parameter~$\kappa(x)$, showing in particular that it can be expressed in terms of the tree~$t(x)$.

\begin{lemma}
\label{lem:kx}
For any $n\geq 1$ and $x\in A_n\cup B_n$ we have the following:
\begin{enumerate}[label=(\roman*),leftmargin=8mm]
\item For any $y\in\neck{x}$ and any integer $i\geq 0$ we have $\neck{f^i(x)}=\neck{f^i(y)}$.
In particular, we have $\kappa(y)=\kappa(x)$.
\item For any integer $i\geq 0$ we have $\neck{f^i(x)}=\neck{f^{\kappa(x)+i}(x)}$.
\item For any integers $0\leq i<j<\kappa(x)$ we have $\neck{f^i(x)}\neq\neck{f^j(x)}$.
\item For any integer $i\geq 0$ we have $\kappa(f^i(x))=\kappa(x)$.
\item We have $\kappa(x)=2\lambda(t(x))$.
\end{enumerate}
\end{lemma}

\begin{proof}
(i) This follows directly from~\eqref{eq:ffx}.

(ii) By the definition of~$\kappa(x)$, we have $\neck{x}=\neck{f^{\kappa(x)}(x)}$.
Using~\eqref{eq:ffx}, this gives $\neck{f^i(x)}=\neck{f^i(f^{\kappa(x)}(x))}=\neck{f^{\kappa(x)+i}(x)}$.

(iii) Suppose for the sake of contradiction that $y:=f^i(x)$ and $z=f^j(x)$ with $0\leq i<j<\kappa(x)$ satisfy $\neck{y}=\neck{z}$.
Then, using that $f$ is invertible, we obtain from~\eqref{eq:ffx} that $\neck{x}=\neck{f^{-i}(y)}=\neck{f^{j-i}(z)}$ with $j-i<\kappa(x)$, contradicting the definition of~$\kappa(x)$ in~\eqref{eq:kx}.

(iv) It suffices to prove that $\kappa(f(x))=\kappa(x)$.
Observe that we have $\neck{f^{\kappa(x)}(f(x))}=\neck{f^{\kappa(x)+1}(x)}=\neck{f(x)}$ by~(ii).
On the other hand, we have $\neck{f^j(f(x))}\neq\neck{f(x)}$ for all $1\leq j<\kappa(x)$ by~(iii).
Combining these two observations proves that $\kappa(f(x))=\kappa(x)$.

(v) By~\eqref{eq:rot}, two applications of~$f$ correspond to one rotation of the tree~$t(x)$.
The statement now follows from the definition~\eqref{eq:kx}.
\end{proof}

For any $x\in A_n\cup B_n$ we define
\begin{subequations}
\label{eq:Pxn}
\begin{equation}
\label{eq:Px}
\begin{split}
P(x)&:=\big(x,f(x),f^2(x),\ldots,f^{\kappa(x)-1}(x)\big). \\
\end{split}
\end{equation}
By~\eqref{eq:kx}, $P(x)$ is a periodic path in the middle levels graph~$M_n$, and by Lemma~\ref{lem:kx}~(iii), $\neck{P(x)}$ is a cycle in the necklace graph~$N_n$.
For any $y\in\neck{x}$ and any integer $i\geq 0$, combining Lemma~\ref{lem:kx}~(i)+(iv) shows that $\kappa(f^i(y))=\kappa(x)$, and so~$\neck{P(x)}$ and~$\neck{P(f^i(y))}$ describe the same cycle, differing only in the choice of the starting vertex (recall~\eqref{eq:ffx}).
We may thus define a cycle factor in~$N_n$ by
\begin{equation}
\label{eq:Fn}
\cF_n:=\big\{\neck{P(x)}\mid x\in A_n\cup B_n\big\}.
\end{equation}
\end{subequations}
This definition is illustrated in Figures~\ref{fig:n3} and~\ref{fig:f4} for $n=3$ and $n=4$, respectively.
The following proposition summarizes the observations from this section.

\begin{proposition}
\label{prop:Fn}
For any $n\geq 2$, the cycle factor $\cF_n$ defined in~\eqref{eq:Pxn} has the following properties:
\begin{enumerate}[label=(\roman*),leftmargin=8mm]
\item For every $x\in A_n\cup B_n$, the $2i$th vertex~$y$ after~$x$ on the periodic path~$P(x)$ satisfies $t(y)=\rho^i(t(x))$.
Consequently, we can identify the path~$P(x)$ and the cycle~$\neck{P(x)}$ with the plane tree~$[t(x)]$.
\item The number of vertices of the path~$P(x)$ and the cycle~$\neck{P(x)}$ is $2\lambda(t(x))\geq 4$, and we have $\ell(f^{2i}(x))=\ell(x)+i$ for all $i=0,\ldots,\lambda(t(x))$.
\item The cycles of~$\cF_n$ are in bijection with plane trees with $n$ edges.
\end{enumerate}
\end{proposition}

\begin{proof}
Clearly, (i) follows from~\eqref{eq:rot} and the definition~\eqref{eq:Px}.
Moreover, (iii) is an immediate consequence of~(i).
To prove~(ii), note that $|P(x)|=\kappa(x)$ by the definition~\eqref{eq:Px}, and use that $\kappa(x)=2\lambda(t(x))$ by Lemma~\ref{lem:kx}~(v).
As $n\geq 2$, Lemma~\ref{lem:lambda} guarantees that $\lambda(t(x))\geq 2$, and so $|P(x)|=2\lambda(t(x))\geq 4$.
The second part of claim~(ii) follows directly from~\eqref{eq:shift}.
\end{proof}

By Proposition~\ref{prop:Fn}~(iii), the number of cycles of~$\cF_n$ is OEIS sequence~A002995~\cite{oeis}, which for $n\geq 1$ starts as $1,1,2,3,6,14,34,95,280,854,\ldots$.
Asymptotically, we have $|\cF_n|=(1+o(1))4^n/(2\sqrt{\pi}n^{5/2})$ as $n$ tends to infinity, i.e., the number of plane trees and hence the number of cycles of~$\cF_n$ grows exponentially.

\section{Gluing the periodic paths}
\label{sec:gluing}

In this section we implement the ideas outlined in Section~\ref{sec:idea-gluing}, showing how to glue pairs of periodic paths to one longer periodic path.
It turns out that the gluing operation involving two periodic paths can be interpreted as a local modification operation on the two corresponding rooted trees; see Figure~\ref{fig:pull}.
Repeating this gluing process will eventually produce a single periodic path that corresponds to a Hamilton cycle in the necklace graph.
The most technical aspect of this approach is to ensure that multiple gluing steps do not interfere with each other, and the conditions that ensure this are captured in Proposition~\ref{prop:Cxy} below.

We define the two rooted trees $s_n:=1(10)^{n-1}0\in D_n$ for $n\geq 3$ and $s_n':=10s_{n-1}\in D_n$ for $n\geq 4$.
Note that both $s_n$ and $s_n'$ have $n$ edges, $s_n$ is a star, and $s_n'$ is obtained from a star by appending an additional edge to one leaf.
For $n\geq 4$, consider two Dyck words $x,y\in D_n$ with $(x,y)\neq (s_n,s_n')$ of the form
\begin{equation}
\label{eq:gluing}
x=1\,1\,0\,u\,0\,v, \quad y=1\,0\,1\,u\,0\,v, \quad \text{ with } u,v\in D.
\end{equation}

\begin{wrapfigure}{r}{0.42\textwidth}
\includegraphics{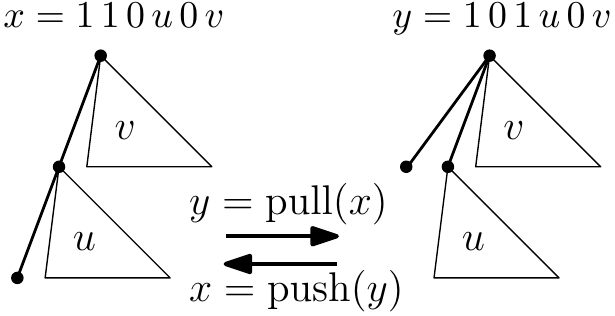}
\caption{A gluing pair~$(x,y)$ and the pull/push operations between the corresponding rooted trees.}
\label{fig:pull}
\end{wrapfigure}
We refer to $(x,y)$ as a \emph{gluing pair}, and we use $G_n$ to denote the set of all gluing pairs $(x,y)$, $x,y\in D_n$.
Considering the corresponding rooted trees, we say that $y$ is obtained from~$x$ by the \emph{pull} operation, and we refer to the inverse operation that transforms~$y$ into~$x$ as the \emph{push} operation; see Figure~\ref{fig:pull}.
We write this as $y=\pull(x)$ and $x=\push(y)$.
We refer to any~$x$ as in~\eqref{eq:gluing} as a \emph{pullable tree}, and to any~$y$ as in~\eqref{eq:gluing} as a \emph{pushable tree}.
Also, we refer to the subtrees~$u$ and~$v$ in~\eqref{eq:gluing} as the \emph{left and right subtree of~$x$ or~$y$}, respectively.
A pull removes the leftmost edge that leads from the leftmost child of the root of~$x$ to a leaf, and reattaches this edge as the leftmost child of the root, yielding the tree~$y=\pull(x)$.
A push removes the leaf that is the leftmost child of the root of~$y$, and reattaches this edge as the leftmost child of the second child of the root of~$y$, yielding the tree~$x=\push(y)$.
Under this viewpoint, we can use the same identifiers for vertices and edges in~$x$ and~$y$.

The next lemma asserts that the centroid(s) of a tree are invariant under certain pull/push operations, and that these operations change the tree potential only by~$\pm 1$.

\begin{lemma}
\label{lem:potential}
Let $(x,y)\in G_n$ be a gluing pair as in~\eqref{eq:gluing}.
If $x$ has a centroid in its right subtree, then this is also a centroid of~$y$ and we have $\varphi(y)=\varphi(x)-1$.
If $y$ has a centroid in its left subtree, then this is also a centroid of~$x$ and we have $\varphi(y)=\varphi(x)+1$.
\end{lemma}

\begin{proof}
Let $a$ be the leaf incident to the edge in which~$x$ and~$y$ differ.
Clearly, $a$ is not a centroid of~$y$.
Moreover, for any vertex~$b$ in the subtree~$u$ of~$x$, the pull operation that transforms~$x$ into~$y$ changes the potential of~$b$ by $+1$.
Similarly, for any vertex~$b$ in the subtree~$v$ of~$x$, the pull operation changes the potential of~$b$ by $-1$.
This proves the first part of the lemma.
The proof of the second part is analogous.
\end{proof}

For a gluing pair~$(x,y)\in G_n$, we define $x^k:=f^k(x\,0)$ and $y^k:=f^k(y\,0)$ for $k\geq 0$.
These sequences agree with the first vertices of the periodic paths~$P(x)$ and~$P(y)$, respectively, defined in~\eqref{eq:Px}.
Using the definition~\eqref{eq:f}, a straightforward calculation yields
\begin{equation}
\label{eq:xyi}
\begin{pmatrix} x^0 \\ x^1 \\ x^2 \\ x^3 \\ x^4 \\ x^5 \\ x^6 \end{pmatrix} = \begin{pmatrix} 1\,1\,0\,u\,0\,v\,0 \\ 1\,1\,0\,u\,1\,v\,0 \\ 0\,1\,0\,u\,1\,v\,0 \\ 0\,1\,1\,u\,1\,v\,0 \\ 0\,0\,1\,u\,1\,v\,0 \\ 1\,0\,1\,u\,1\,v\,0 \\ 1\,0\,0\,u\,1\,v\,0 \end{pmatrix},
\quad
\begin{pmatrix} y^0 \\ y^1 \end{pmatrix} = \begin{pmatrix} 1\,0\,1\,u\,0\,v\,0 \\ 1\,1\,1\,u\,0\,v\,0 \end{pmatrix}.
\end{equation}
From this we obtain
\begin{equation}
\label{eq:Txi}
\begin{pmatrix} t(x^0) \\ t(x^2)=\rho(t(x^0)) \\ t(x^4)=\rho^2(t(x^0)) \\ t(x^6)=\rho^3(t(x^0)) \end{pmatrix} = \begin{pmatrix} 1\,1\,0\,u\,0\,v \\ 1\,0\,u\,1\,v\,0 \\ 1\,u\,1\,v\,0\,0 \\ u\,1\,v\,0\,1\,0 \end{pmatrix},
\quad
t(y^0)=(1\,0\,1\,u\,0\,v\,0)
\end{equation}
(recall Proposition~\ref{prop:Fn}~(i)).

The next lemma shows that the bitstrings listed in \eqref{eq:xyi} all belong to distinct necklaces.

\begin{lemma}
\label{lem:min-length}
Let $(x,y)\in G_n$ be a gluing pair as in~\eqref{eq:gluing}.
Then we have $|P(x^0)|=\kappa(x^0)\geq 8$ and $|P(y^0)|=\kappa(y^0)\geq 4$.
\end{lemma}

Note that if $(x,y)=(s_n,s_n')$ then we have $\kappa(x^0)=4$ and therefore $\neck{x^0}=\neck{x^4}$ and $\neck{x^2}=\neck{x^6}$, so for this case the statement of Lemma~\ref{lem:min-length} would not hold.

\begin{proof}
Note that for any $n\geq 4$, the star $x=s_n$ is the only rooted tree with $\lambda(x)=2$.
For any other tree~$x$ we have $\lambda(x)\geq 4$ by Lemma~\ref{lem:lambda}, so by Lemma~\ref{lem:kx}~(v) we have $\kappa(x^0)=2\lambda(x)\geq 8$.

For $n\geq 4$, we have $\lambda(y)\geq 2$ for any tree $y\in D_n$ by Lemma~\ref{lem:lambda}, so by Lemma~\ref{lem:kx}~(v) we have $\kappa(y^0)=2\lambda(y)\geq 4$.
\end{proof}

Observe from~\eqref{eq:xyi} that
\begin{equation}
\label{eq:Cxy}
C(x,y):=(x^0,x^1,x^6,x^5,y^0,y^1)
\end{equation}
is a 6-cycle in the middle levels graph~$M_n$; see Figure~\ref{fig:cxy}.
The bit positions flipped along this cycle are
\begin{equation}
\label{eq:alpha-Cxy}
\alpha(C(x,y)):=(|u|+4,2,3,|u|+4,2,3).
\end{equation}

\begin{figure}
\includegraphics{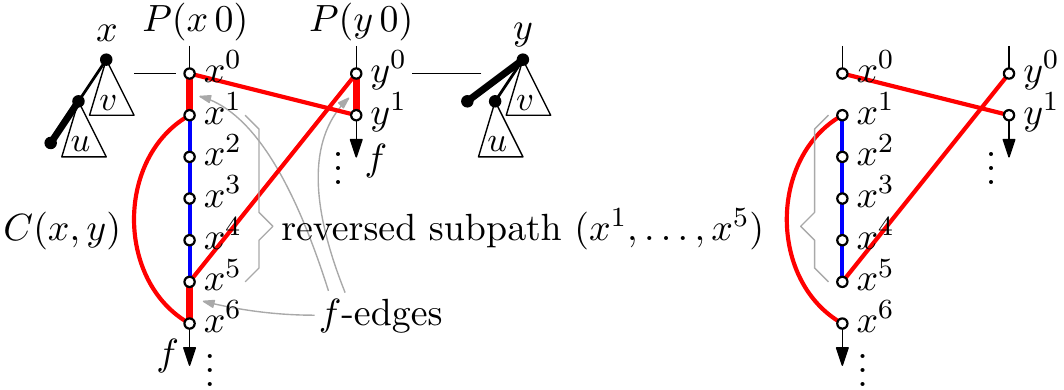}
\caption{Illustration of the gluing cycle $C(x,y)$ and several related definitions.}
\label{fig:cxy}
\end{figure}

By Lemma~\ref{lem:min-length} we have that for all $i\geq 0$ the 6-cycle $\sigma^i(C(x,y))$ has the two edges~$\sigma^i((x^0,x^1))$, $\sigma^i((x^5,x^6))$ in common with the periodic path~$\sigma^i(P(x^0))$, and the edge~$\sigma^i((y^0,y^1))$ in common with the periodic path~$\sigma^i(P(y^0))$.
We refer to these three edges as the \emph{$f$-edges of~$\sigma^i(C(x,y))$} (see Figure~\ref{fig:cxy}), and we refer to $\sigma^i(C(x,y))$ as a \emph{gluing cycle}.
Observe that if $[x]\neq [y]$, then $\neck{P(x^0)}$ and~$\neck{P(y^0)}$ are distinct cycles in the necklace graph~$N_n$ by Proposition~\ref{prop:Fn}~(i), implying that
\begin{equation}
\label{eq:Pxy}
P(x^0)\bowtie P(y^0):=\Big(x^0,y^1,y^2,\ldots,y^{2\lambda(y)-1}, 
\sigma^{-\lambda(y)}\big((y^0,x^5,x^4,x^3,x^2,x^1,x^6,x^7,\ldots,x^{2\lambda(x)-1})\big)\Big)
\end{equation}
is a periodic path in the middle levels graph~$M_n$, and together the $2n+1$ periodic paths $\bigcup_{i\geq 0}\sigma^i\big(P(x^0)\bowtie P(y^0)\big)$ visit all vertices of $\bigcup_{i\geq 0}\sigma^i\big(P(x^0)\cup P(y^0)\big)$.
To see this, recall that $|P(x^0)|=2\lambda(x)$, $|P(y^0)|=2\lambda(y)$, and $\sigma^{\lambda(y)}(y^{2\lambda(y)})=y^0$ by Proposition~\ref{prop:Fn}~(ii).
Note that the edge set of $\bigcup_{i\geq 0}\sigma^i\big(P(x^0)\bowtie P(y^0)\big)$ is the symmetric difference of the edge sets of $\bigcup_{i\geq 0}\sigma^i\big(P(x^0)\cup P(y^0)\big)$ with the gluing cycles $\bigcup_{i\geq 0}\sigma^i(C(x,y))$.
Specifically, the $f$-edges of the gluing cycles~$\sigma^i(C(x,y))$, $i\geq 0$, are removed and replaced by the other edges $\sigma^i((x^1,x^6))$, $\sigma^i((y^0,x^5))$, and $\sigma^i((x^0,y^1))$, for all $i\geq 0$; see the right hand side of Figure~\ref{fig:cxy}

In the necklace graph~$N_n$, the symmetric difference of the edge sets of the two cycles~$\neck{P(x^0)}$ and~$\neck{P(y^0)}$ with the 6-cycle~$\neck{C(x,y)}$ is a single cycle on the same vertex set as $\neck{P(x^0)}\cup \neck{P(y^0)}$.

For all $i\geq 0$, we say that the subpath $\sigma^i((x^1,\ldots,x^5))$ of~$\sigma^i(P(x^0))$ is \emph{reversed by~$\sigma^i(C(x,y))$}; see Figure~\ref{fig:cxy}.
Moreover, we say that two gluing cycles~$\sigma^i(C(x,y))$ and~$\sigma^j(C(\hx,\hy))$ are \emph{compatible}, if they have no $f$-edges in common.
We also say that $\sigma^i(C(x,y))$ and~$\sigma^j(C(\hx,\hy))$ are \emph{nested}, if the $f$-edge~$\sigma^i((y^0,y^1))$ of~$\sigma^i(C(x,y))$ belongs to the path that is reversed by~$\sigma^j(C(\hx,\hy))$; see Figure~\ref{fig:nested}.
Lastly, we say that $\sigma^i(C(x,y))$ and $\sigma^j(C(\hx,\hy))$ are \emph{interleaved}, if the $f$-edge~$\sigma^j((\hx^0,\hx^1))$ of~$\sigma^j(C(\hx,\hy))$ belongs to the path that is reversed by~$\sigma^i(C(x,y))$.

The following key proposition captures the conditions under which a pair of gluing cycles is compatible, interleaved, or nested, respectively.

\begin{proposition}
\label{prop:Cxy}
Let $n\geq 4$ and let $(x,y),(\hx,\hy)\in G_n$ be two gluing pairs with $[x]\neq [y]$, $[\hx]\neq[\hy]$, and $\{[x],[y]\}\neq\{[\hx],[\hy]\}$.
Then for any integers $i,j\geq 0$, the gluing cycles~$\sigma^i(C(x,y))$ and~$\sigma^j(C(\hx,\hy))$ defined in~\eqref{eq:Cxy} have the following properties:
\begin{enumerate}[label=(\roman*),leftmargin=8mm]
\item $\sigma^i(C(x,y))$ and $\sigma^j(C(\hx,\hy))$ are compatible.
\item $\sigma^i(C(x,y))$ and $\sigma^j(C(\hx,\hy))$ are interleaved if and only if $i=j+2$ and $\hx=\rho^2(x)$.
\item $\sigma^i(C(x,y))$ and $\sigma^j(C(\hx,\hy))$ are nested if and only $i=j-1$ and $\hx=\rho^{-1}(y)$.
\end{enumerate}
\end{proposition}

\begin{figure}
\includegraphics{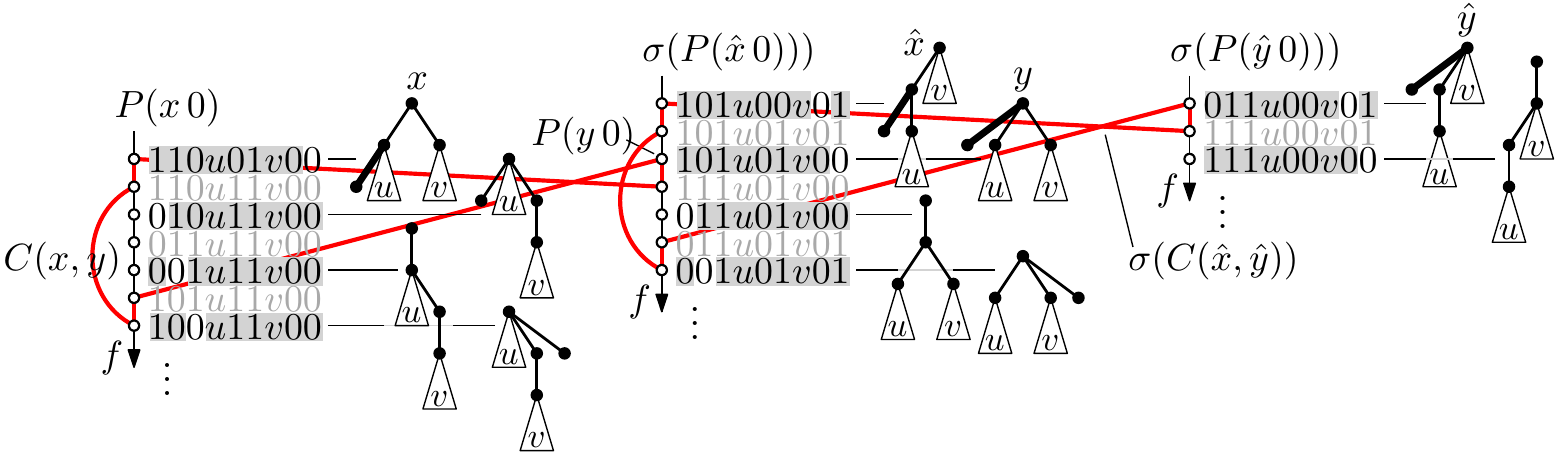}
\caption{Two nested 6-cycles $C(x,y)$ and $C(\hx,\hy)$.
The plane tree~$[\hy]$ is obtained from~$[x]$ by pulling the same edge twice in succession.
This edge is drawn fat in the figure.}
\label{fig:nested}
\end{figure}

Two nested gluing cycles as in case~(iii) of Proposition~\ref{prop:Cxy} can be interpreted as follows:
We start at the tree~$x$, pull an edge towards the root to reach the tree~$y=\pull(x)$, then perform an inverse tree rotation~$\hx=\rho^{-1}(y)$, which makes the previously pulled edge eligible to be pulled again, then pull this edge a second time, reaching the tree~$\hy=\pull(\hx)$.
Consequently, nested gluing cycles occur if and only if the same edge of the underlying plane trees is pulled twice in succession; see Figure~\ref{fig:nested}.

\begin{proof}
It suffices to prove the lemma for $i=0$ and arbitrary $j\geq 0$, so for the rest of the proof we assume that~$i=0$.
We consider the bitstrings $z^k:=f^k(z\,0)$ for all $z\in \{x,y,\hx,\hy\}$ and $k\geq 0$.

By Proposition~\ref{prop:Fn}~(i) and the assumptions $[x]\neq [y]$ and $[\hx]\neq[\hy]$, each of the 6-cycles $\neck{C(x,y)}$ and $\neck{C(\hx,\hy)}$ in~$N_n$ connects two distinct cycles of the cycle factor~$\cF_n$ with each other, and the $f$-edges of~$C(x,y)$ and~$C(\hx,\hy)$ all lie on cycles from the factor.
Consequently, by the assumption that $\{[x],[y]\}\neq\{[\hx],[\hy]\}$, it suffices to verify the following three claims about edges in~$M_n$:
(1)~If $[x]=[\hx]$, then the edge $(x^5,x^6)$ is distinct from~$\sigma^j(\hx^0,\hx^1)$, and the edge~$(x^0,x^1)$ is distinct from~$\sigma^j(\hx^5,\hx^6)$ for all $j\geq 0$.
(2)~If $[x]=[\hx]$, then the edge $\sigma^j(\hx^0,\hx^1)$ does not belong to the path $(x^1,\ldots,x^5)$ for any~$j\geq 0$, with the only possible exception occuring if $\sigma^{-2}(\hx^0,\hx^1)=(x^4,x^5)$ and~$\hx=\rho^2(x)$.
(3)~If $[y]=[\hx]$, then the edge $(y^0,y^1)$ does not belong to the path $\sigma^j(\hx^0,\ldots,\hx^6)$ for any~$j\geq 0$, with the only possible exception occuring if $(y^0,y^1)=\sigma^1(\hx^2,\hx^3)$ and~$\hx=\rho^{-1}(y)$.

We begin observing that $\sigma^j(z^k)\in A_n$ for even~$k$ and $\sigma^j(z^k)\in B_n$ for odd~$k$ and all~$j\geq 0$.
It follows that $(x^5,x^6)\in B_n\times A_n$ and $\sigma^j(\hx^0,\hx^1)\in A_n\times B_n$, which immediately implies~(1).
To prove~(2), we first show that $\sigma^j(\hx^0)\neq x^2$.
This follows from~\eqref{eq:Txi}, by observing that $t(\sigma^j(\hx^0))=t(\hx^0)$ and $t(x^2)$ differ in the second bit.
Note also that the condition $t(\sigma^j(\hx^0))=t(\hx^0)=t(x^4)=\rho^2(t(x^0))$ is equivalent to $\hx=\rho^2(x)$.
Moreover, from~\eqref{eq:xyi} we see that $\ell(\sigma^j(\hx^0))=-j$ and $\ell(x^4)=2$, so $\sigma^j(\hx^0)=x^4$ implies that $j=-2$.

To prove~(3), we first show that $y^0\notin\{\sigma^j(\hx^0),\sigma^j(\hx^4)\}$.
From~\eqref{eq:Txi} we see that $t(y^0)$ and $t(\sigma^j(\hx^0))=t(\hx^0)$ differ in the second and third bit, showing that $y^0$ is different from~$\sigma^j(\hx^0)$.
From~\eqref{eq:Txi} we also obtain that the root of $t(\sigma^j(\hx^4))=t(\hx^4)$ is a leaf, whereas the root of~$t(y^0)$ is not a leaf, proving that $y^0$ is different from~$\sigma^j(\hx^4)$.
Note also that the condition $t(y^0)=t(\sigma^j(\hx^2))=t(\hx^2)=\rho(t(\hx^0))$ is equivalent to $y=\rho(\hx)$.
Moreover, from~\eqref{eq:xyi} we see that $\ell(y^0)=0$ and $\ell(\sigma^j(\hx^2))=1-j$, so $y^0=\sigma^j(\hx^2)$ implies that $j=1$.

This completes the proof.
\end{proof}

\section{Translation to a spanning tree problem}
\label{sec:tree}

In this section we combine the ingredients from the previous two sections, and show how they translate Knuth's Gray code problem into the problem of finding a spanning tree~$\cT_n$ in a suitably defined auxiliary graph~$\cH_n$, following the ideas outlined in Section~\ref{sec:idea-gluing}.
The definitions of the graphs~$\cH_n$ and~$\cT_n$ are given in Sections~\ref{sec:Hn} and~\ref{sec:Tn} below, respectively.
Based on this, we describe how flip sequences are glued together inductively along the spanning tree~$\cT_n$ (Sections~\ref{sec:flip} and~\ref{sec:flip-T}).
This allows us to make a first attempt at proving Theorem~\ref{thm:star} (Section~\ref{sec:first}).
Unfortunately, this attempt does not give a complete proof yet, as we are unable to control the shift value of the flip sequences resulting from the gluing process; recall the discussion from Section~\ref{sec:idea-shift}.

\subsection{Definition of~$\cH_n$}
\label{sec:Hn}

For $n\geq 4$, we let~$\cH_n$ denote the directed arc-labeled multigraph defined as follows:
The node set of~$\cH_n$ is~$T_n$, i.e., all plane trees with $n$ edges.
Moreover, for each gluing pair~$(x,y)\in G_n$, there is an arc labeled~$(x,y)$ from the plane tree~$[x]$ to the plane tree~$[y]$ in~$\cH_n$.
Some pairs of nodes of~$\cH_n$ may be connected by multiple arcs oriented the same way (with different labels), such as $([1100110010],[1010110010])$ and $([1100101100],[1010101100])$.
Some pairs of nodes may be connected by multiple arcs oriented oppositely, such as $([11010100],[10110100])$ and $([11001010],[10101010])$.
There may also be loops in~$\cH_n$, such as $([11010010],[10110010])$.

Let $\cT$ be a simple subgraph of~$\cH_n$, i.e., $\cT$ has no loops and no multiple arcs, neither oriented the same way nor oppositely.
We let $G(\cT)$ be the set of all arc labels of~$\cT$, i.e., the set of all gluing pairs~$(x,y)\in G_n$ that give rise to the arcs in~$\cT$.
As $\cT$ is simple, we clearly have $[x]\neq[y]$, $[\hx]\neq[\hy]$, and $\{[x],[y]\}\neq \{[\hx],[\hy]\}$ for all $(x,y),(\hx,\hy)\in G(\cT)$.
We say that $G(\cT)$ is \emph{interleaving-free} or \emph{nesting-free}, respectively, if there are no two gluing pairs $(x,y),(\hx,\hy)\in G(\cT)$ such that the gluing cycles~$\sigma^i(C(x,y))$ and $\sigma^j(C(\hx,\hy))$ are nested or interleaved for any $i,j\geq 0$.

The next lemma provides a simple sufficient condition guaranteeing interleaving-freeness.

\begin{lemma}
\label{lem:heavy}
If for every gluing pair $(x,y)\in G(\cT)$, the root of the tree~$x$ is not a leaf, then $G(\cT)$ is interleaving-free.
\end{lemma}

\begin{proof}
Suppose there are two gluing pairs $(x,y),(\hx,\hy)\in G(\cT)$ such that the gluing cycles $\sigma^i(C(x,y))$ and $\sigma^j(C(\hx,\hy))$ are interleaved for some $i,j\geq 0$.
By Proposition~\ref{prop:Cxy}~(ii), this implies $i=j+2$ and $\hx=\rho^2(x)$.
However, note that the root of $\rho^2(x)$ is a leaf (recall~\eqref{eq:Txi}), whereas the root of~$\hx$ is not a leaf by the assumption of the lemma, so this is a contradiction.
\end{proof}

\subsection{Pullable/pushable leaves}

The following definitions are illustrated in Figure~\ref{fig:pullable}.
Given a plane tree~$T$ and two vertices $a,b$ of~$T$, we let $d(a,b)$ denote the distance of~$a$ and~$b$ in~$T$, and we let $p^i(a,b)$, $i=0,1,\ldots,d(a,b)$, be the $i$th vertex on the path from~$a$ to~$b$ in~$T$.
In particular, we have $p^0(a,b)=a$ and $p^{d(a,b)}(a,b)=b$.

\begin{figure}[b!]
\includegraphics{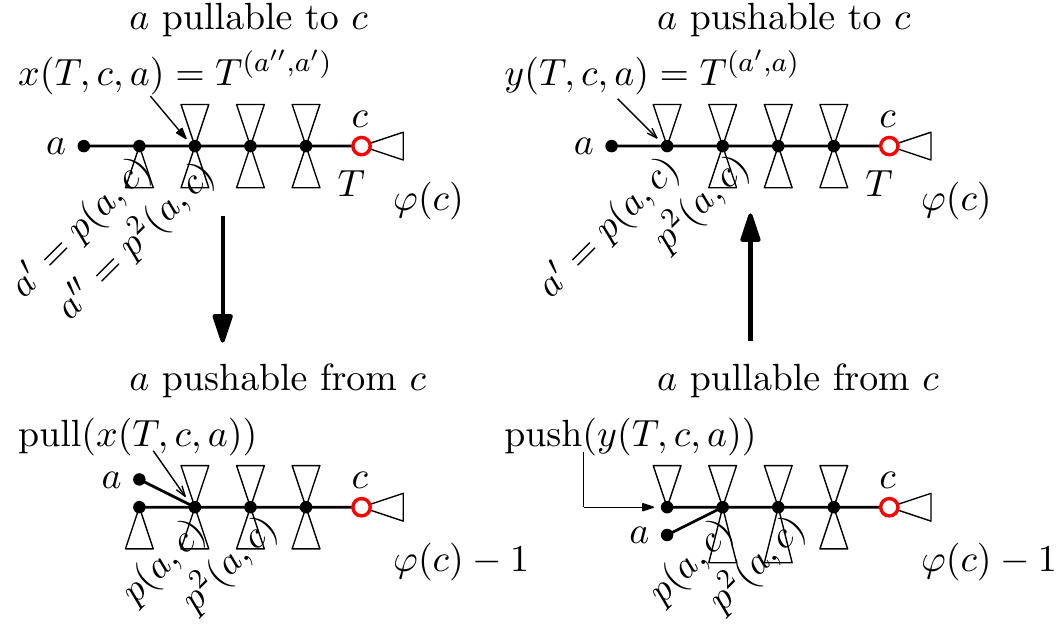}
\caption{Definition of pullable and pushable leaves.}
\label{fig:pullable}
\end{figure}

Consider a vertex~$c$ and a leaf~$a$ of~$T$ with $d(a,c)\geq 2$.
We say that $a$ is \emph{pullable to~$c$}, if $p^1(a,c)$ has no neighbors between~$p^2(a,c)$ and~$a$ in its ccw ordering of neighbors.
We say that $a$ is \emph{pushable to~$c$}, if $p^1(a,c)$ has no neighbors between~$a$ and~$p^2(a,c)$ in its ccw ordering of neighbors.

Consider a vertex~$c$ and a leaf~$a$ of~$T$ with $d(a,c)\geq 1$.
We say that $a$ is \emph{pullable from~$c$}, if $d(a,c)\geq 2$ and $p^1(a,c)$ has at least one neighbor between~$p^2(a,c)$ and~$a$ in its ccw ordering of neighbors, of if $d(a,c)=1$ and $c$ is not a leaf.
We say that $a$ is \emph{pushable from~$c$}, if $d(a,c)\geq 2$ and $p^1(a,c)$ has at least one neighbor between $a$ and~$p^2(a,c)$ in its ccw ordering of neighbors, or if $d(a,c)=1$ and $c$ is not a leaf.

For any odd $n\geq 5$ we define the \emph{dumbbells} $d_n:=1(10)^{(n-1)/2}0(10)^{(n-1)/2}\in D_n$ and $d_n':=\rho^{-2}(d_n):=101(10)^{(n-1)/2}0(10)^{(n-3)/2}\in D_n$.
Each dumbbell has two centroids of degree~$(n+1)/2$ each, and all remaining vertices are leaves.

Given a plane tree~$T$ with a unique centroid~$c$, we refer to every $c$-subtree of~$T$ as \emph{active}.
If $T$ has two centroids~$c,c'$, we refer to every $c$-subtree of~$T$ except the one containing~$c'$, and to every $c'$-subtree of~$T$ except the one containing~$c$ as \emph{active}.
For $n\geq 4$, note that if $T\neq [s_n]$ and $T\neq [d_n]$ for odd~$n$, then $T$ has a centroid with an active subtree that is not a single edge.

The following two lemmas describe certain pull/push operations on plane trees that preserve the centroid(s), and that change the tree potential by~$\pm 1$.

\begin{lemma}
\label{lem:pot-pull}
Let $c$ be a centroid of a plane tree~$T$, let $a$ be a leaf of~$T$ that is pullable to~$c$ and that belongs to an active $c$-subtree unless $n\geq 5$ is odd and $T=[d_n]$.
Then the rooted tree~$x:=x(T,c,a):=T^{(a'',a')}$, where $a':=p^1(a,c)$ and $a'':=p^2(a,c)$, is a pullable tree, the rooted tree $y:=\pull(x)$ satisfies $\varphi(y)=\varphi(x)-1$, and the leaf~$a$ is pushable from~$c$ in~$[y]$.
Moreover, the centroid(s) of~$x$ and~$y$ are identical and contained in the right subtrees of~$x$ and~$y$, unless $n\geq 5$ is odd and $x=d_n$, in which case $x$ has two centroids, namely the roots of its left and right subtree, and the root of the right subtree is the unique centroid of~$y$.
\end{lemma}

\begin{proof}
The statements follow immediately from the definitions given before the lemma, and by Lemma~\ref{lem:potential}.
To see that $x\neq s_n$ note that the star $[s_n]$ has a unique centroid~$c$ and no leaves that are pullable to~$c$.
\end{proof}

\begin{lemma}
\label{lem:pot-push}
Let $c$ be a centroid of a plane tree~$T$, let $a$ be a thick leaf of~$T$ that is pushable to~$c$ and that belongs to an active $c$-subtree unless $n\geq 5$ is odd and $T=[d_n']$.
Then the rooted tree~$y:=y(T,c,a):=T^{(a',a)}$, where $a':=p^1(a,c)$, is a pushable tree, the rooted tree $x:=\push(y)$ satisfies $\varphi(x)=\varphi(y)-1$, and the leaf~$a$ is pullable from~$c$ in~$[y]$.
Moreover, the centroid(s) of~$y$ and~$x$ are identical and contained in the left subtrees of~$y$ and~$x$, unless $n\geq 5$ is odd and $y=d_n'$, in which case $y$ has two centroids, namely the roots of its left and right subtree, and the root of the left subtree is the unique centroid of~$x$.
\end{lemma}

\begin{proof}
The proof is analogous to Lemma~\ref{lem:pot-pull}.
To see that $y\neq s_n'$ note that $a$ is assumed to be thick, unlike the leaf we would push in~$s_n'$ to obtain~$s_n$.
\end{proof}

\subsection{Definition of~$\cT_n$}
\label{sec:Tn}

For $n\geq 4$, we define a subgraph~$\cT_n$ of~$\cH_n$ as follows:
For every plane tree~$T\in T_n$ with $T\neq [s_n]$ and $T\neq [d_n]$ for odd~$n$, we fix a vertex~$c$ that is a centroid of~$T$ and that has at least one active $c$-subtree that is not a single edge.
The leftmost leaf of every such $c$-subtree is pullable to~$c$, and we fix one such leaf~$a$.
If $n$ is odd and $T=[d_n]$, we let~$c$ be one of its centroids, which has exactly one $c$-subtree that is not a single edge, namely the tree~$s_{(n+1)/2}$.
The leftmost leaf~$a$ of this subtree is pullable to~$c$.
In both cases, let $x:=x(T,c,a)$ be the corresponding pullable rooted tree as defined in Lemma~\ref{lem:pot-pull}, and define $y:=\pull(x)$, yielding the gluing pair $(x,y)\in G_n$.
We let $\cT_n$ be the spanning subgraph of~$\cH_n$ that is given by the union of arcs $([x],[y])$ labeled~$(x,y)$ for all gluing pairs~$(x,y)$ obtained in this way.
In the above definition, ties in the case of two centroids or in the case of multiple $c$-subtrees can be broken arbitrarily.

The next lemma shows that the graph~$\cT_n$ defined above is indeed a spanning tree of~$\cH_n$, and moreover the potential of plane trees along every arc of~$\cT_n$ changes by~$-1$.
For any arc~$(T,T')$, we say that~$T'$ is an \emph{out-neighbor} of~$T$, and we say that~$T$ is an \emph{in-neighbor} of~$T'$.

\begin{lemma}
\label{lem:Tn}
For any $n\geq 4$, the graph~$\cT_n$ is a spanning tree of~$\cH_n$, and for every arc~$(T,T')$ in~$\cT_n$ we have $\varphi(T')=\varphi(T)-1$.
Every plane tree~$T$ other than the star~$[s_n]$ has exactly one neighbor~$T'$ in~$\cT_n$ with $\varphi(T')=\varphi(T)-1$, which is an out-neighbor.
Furthermore, $G(\cT_n)$ is interleaving-free.
\end{lemma}

\begin{proof}
Consider the gluing pair $(x,y)\in G(\cT_n)$ added for the plane tree~$T$ with $T=[x]$.
By Lemma~\ref{lem:pot-pull} we have $\varphi(y)=\varphi(x)-1$, i.e., the potential of the trees changes by~$-1$ along every arc of~$\cT_n$.
It follows that in~$\cT_n$, every plane tree~$T$ other than the star~$[s_n]$ has exactly one neighbor~$T'$ with $\varphi(T')=\varphi(T)-1$, which is an out-neighbor.
Consequently, $\cT_n$ has no cycles, regardless of the orientation of arcs along the cycle (in particular, there are no loops).
As from every plane tree~$T\in T_n$ other than~$[s_n]$, we can reach a tree~$T'$ with $\varphi(T')=\varphi(T)-1$, there is a directed path from~$T$ to the star~$[s_n]$, which is the unique plane tree with minimum potential~$n$.
We showed that $\cT_n$ does not contain cycles and is connected, i.e., it is a spanning tree.
By Lemma~\ref{lem:pot-pull}, for any gluing pair $(x,y)\in G(\cT_n)$ the right subtree of~$x$ contains a centroid of~$x$.
As a centroid is never a leaf, the right subtree of~$x$ contains edges, i.e., the root of~$x$ is not a leaf, so we may apply Lemma~\ref{lem:heavy} to conclude that $G(\cT_n)$ is interleaving-free.
\end{proof}

\subsection{Basic operations on flip sequences}
\label{sec:flip}

We now describe some basic operations on flip sequences that will be used heavily in the next section when gluing flip sequences together.

Consider a periodic path $P=(x_1,\ldots,x_k)$ in the middle levels graph~$M_n$.
We say that a sequence of integers $\alpha=(a_1,\ldots,a_k)$ is a \emph{flip sequence} for~$P$, if $a_i$ is the position in which $x_{i+1}$ differs from $x_i$ for all $i=1,\ldots,k-1$, and the vertex~$x_{k+1}$ obtained from~$x_k$ by flipping the bit at position~$a_k$ satisfies $\neck{x_{k+1}}=\neck{x_1}$.
There is unique integer~$\lambda$ modulo~$2n+1$ given by the relation $x_1=\sigma^\lambda(x_{k+1})$.
We define $\lambda(\alpha):=\lambda$, and we refer to this quantity as the \emph{shift} of $\alpha$.
In words, the parameter $\lambda$ describes by how much the necklace representatives get rotated to the right when traversing the periodic path once.

We also define
\begin{subequations}
\begin{equation}
\label{eq:rev}
\begin{split}
\rev(P)&:=\Big(x_1,\sigma^{\lambda(\alpha)}\big((x_k,x_{k-1},\ldots,x_2)\big)\Big), \\
\rev(\alpha)&:=(a_k,a_{k-1},\ldots,a_1)-\lambda(\alpha),
\end{split}
\end{equation}
where indices are considered modulo~$2n+1$, as always.
Note that $\rev(\alpha)$ is a flip sequence for the periodic path~$\rev(P)$ satisfying
\begin{equation}
\label{eq:lambda-rev}
\lambda(\rev(\alpha))=-\lambda(\alpha).
\end{equation}
\end{subequations}

Given $P=(x_1,\ldots,x_k)$ and $\alpha=(a_1,\ldots,a_k)$ as before, we define
\begin{subequations}
\begin{equation}
\label{eq:mov}
\begin{split}
\mov(P)&:=\big(x_2,\ldots,x_k,\sigma^{-\lambda(\alpha)}(x_1)\big), \\
\mov(\alpha)&:=\big(a_2,\ldots,a_k,a_1+\lambda(\alpha)\big).
\end{split}
\end{equation}
Note that $\mov(\alpha)$ is a flip sequence for the periodic path~$\mov(P)$ satisfying
\begin{equation}
\label{eq:lambda-mov}
\lambda(\mov(\alpha))=\lambda(\alpha),
\end{equation}
\end{subequations}
which means that the shift is independent of the choice of the starting vertex along the path.
Similarly, for any integer~$i$ we have that $\alpha+i$ is a flip sequence for the periodic path~$\sigma^{-i}(P)$ satisfying
\begin{equation}
\label{eq:lambda-plus}
\lambda(\alpha+i)=\lambda(\alpha).
\end{equation}

For example, the periodic path $P=(1010100,1110100,0110100,0110101)$ has the flip sequence $\alpha=(2,1,7,2)$ with $\lambda(\alpha)=2$, and the periodic path $\rev(P)=(1010100,1010101,1010001,1010011)$ has the flip sequence $\rev(\alpha)=(7,5,6,7)$ with shift $\lambda(\rev(\alpha))=-2$.
Moreover, $\mov(\alpha)=(1,7,2,4)$ is a flip sequence for the periodic path $\mov(P)=(1110100,0110100,0110101,0010101)$ with $\lambda(\mov(\alpha))=2$, and $\alpha+1=(3,2,1,3)$ is a flip sequence for the periodic path $\sigma^{-1}(P)=(0101010,0111010,0011010,1011010)$ with $\lambda(\alpha+1)=2$.

\subsection{Flip sequences for subtrees of $\cH_n$}
\label{sec:flip-T}

Using the notation introduced in the previous section, we now describe how to glue flip sequences of periodic paths together inductively along subtrees of~$\cH_n$.
Ultimately, this will be done for the entire spanning tree~$\cT_n$.
The key problem in this gluing process is to keep track of the shift value of the flip sequences resulting after each step.

For any $x\in A_n\cup B_n$, with $\kappa(x)$ defined in~\eqref{eq:kx}, we let $\alpha(x)$ be the sequence of positions in which $f^{i+1}(x)$ differs from $f^i(x)$ for all $i=0,\ldots,\kappa(x)-1$.
Clearly, $\alpha(x)$ is a flip sequence for the periodic path~$P(x)$ defined in~\eqref{eq:Px}.
By Proposition~\ref{prop:Fn}~(ii), we have
\begin{equation}
\label{eq:lambda-alpha-x}
\lambda(\alpha(x))=\lambda(t(x)).
\end{equation}

Let $\cT$ be any subtree of~$\cH_n$ such that $G:=G(\cT)$ is interleaving-free.
We define the set of necklaces $N(\cT):=\bigcup_{[x]\in\cT} \neck{P(x\,0)}$.
By Proposition~\ref{prop:Fn}~(i), this is the set of all necklaces visited by cycles $\neck{P(x\,0)}$ in~$N_n$ for which the plane tree~$[x]$ belongs to~$\cT$.
In the following, for any~$z\in N(\cT)$ and any~$x\in z$ we define two periodic paths $\cP_G(x)=\{P,P'\}$ with the same starting vertex~$x$ in the middle levels graph~$M_n$ and flip sequences~$\alpha(P)$ and~$\alpha(P')$ for these two paths such that $P'=\rev(P)$ and $\alpha(P')=\rev(\alpha(P))$.
Moreover, $\neck{P}$ and~$\neck{P'}$ will be oppositely oriented cycles in the necklace graph~$N_n$ with vertex set~$N(\cT)$.
These definitions proceed inductively as follows:

\textbf{Base case:}
If $\cT=[x]$ is an isolated node, then we have $G(\cT)=\emptyset$.
For all~$i,j\geq 0$ we define $y:=\sigma^j(f^i(x\,0))$.
Note that $\alpha(y)$ is a flip sequence for $P(y)$, and so we may define
\begin{equation*}
\cP_\emptyset(y):=\{P(y),\rev(P(y))\}, \quad \alpha(P(y)):=\alpha(y), \quad \alpha(\rev(P(y))):=\rev(\alpha(y)),
\end{equation*}
with reversals as defined in~\eqref{eq:rev}.

\textbf{Induction step:}
For the induction step, we assume that $\cT$ has at least two nodes.
Consider all gluing pairs $(x,y),(\hx,\hy)\in G(\cT)$ for which $\sigma^i(C(x,y))$ and $\sigma^j(C(\hx,\hy))$ are nested for some $i,j\geq 0$, and call the arc of~$\cT$ labeled~$(\hx,\hy)$ \emph{bad}.
By Proposition~\ref{prop:Cxy}~(iii), this is only possible if $i=j-1$ and $\hx=\rho^{-1}(y)$, which implies that the union of all bad arcs is a set of directed paths in~$\cT$ and a proper subset of all of its arcs.
In particular, one of the arcs is not bad, i.e., there is a gluing pair $(\hx,\hy)\in G(\cT)$ satisyfing the following property~(*): $\sigma^i(C(x,y))$ and $\sigma^j(C(\hx,\hy))$ are not nested for any $(x,y)\in G(\cT)$ and $i,j\geq 0$.
Consider the subtrees~$\cT_1$ and~$\cT_2$ obtained by removing the arc $([\hx],[\hy])$ from~$\cT$, and consider the sets of gluing pairs~$G_1:=G(\cT_1)$ and~$G_2:=G(\cT_2)$.
By Proposition~\ref{prop:Cxy}~(i), by induction, and by property~(*) and the assumption that $G(\cT)$ is interleaving-free, there is a periodic path~$P_1=(x_1,\ldots,x_k)\in \cP_{G_1}(\hx\,0)$ that satisfies $(x_1,\ldots,x_7)=(\hx^0,\ldots,\hx^6)$, and a periodic path~$P_2=(y_1,\ldots,y_l)\in \cP_{G_2}(\hy\,0)$ that satisfies $(y_1,y_2)=(\hy^0,\hy^1)$.
Moreover, there are corresponding flip sequences $\alpha(P_1)=:\alpha_1=(a_1,\ldots,a_k)$ and $\alpha(P_2)=:\alpha_2=(b_1,\ldots,b_l)$.
We then define the periodic path
\begin{equation}
\label{eq:P12}
P_1\bowtie P_2:=\Big(x_1,y_2,y_3,\ldots,y_l,\sigma^{-\lambda(\alpha(P_2))}\big((y_1,x_6,x_5,x_4,x_3,x_2,x_7,x_8,\ldots,x_k)\big)\Big)
\end{equation}
in the middle levels graph~$M_n$ (cf.~\eqref{eq:Pxy}).
Together, the $2n+1$ periodic paths $\bigcup_{i\geq 0} \sigma(P_1\bowtie P_2)$ visit all vertices of $\bigcup_{i\geq 0}\sigma^i\big(P_1\cup P_2\big)$.
Moreover, considering the decomposition $\hx=1\,u\,0\,v$ with $u,v\in D$, we define
\begin{equation}
\label{eq:alpha-bowtie}
\alpha_1\bowtie \alpha_2:=\Big(3,b_2,b_3,\ldots,b_l,\big((|u|+4,a_5,a_4,a_3,a_2,2,a_7,a_8,\ldots,a_k)\big)+\lambda(\alpha(P_2))\Big).
\end{equation}
As $\alpha_1\bowtie \alpha_2$ is a flip sequence for the periodic path~$P_1\bowtie P_2$ by~\eqref{eq:alpha-Cxy} and~\eqref{eq:lambda-plus}, we may define
\begin{equation}
\label{eq:P-ind}
\begin{split}
P'&:=\mov^j(\sigma^i(P_1\bowtie P_2)), \quad \alpha':=\mov^j(\alpha_1\bowtie \alpha_2-i), \\
\cP_G(y)&:=\{P',\rev(P')\}, \quad \alpha(P'):=\alpha', \quad \alpha(\rev(P')):=\rev(\alpha'),
\end{split}
\end{equation}
where $y$ is the first vertex of the path~$\mov^j(\sigma^i(P_1\bowtie P_2))$, for all $i,j\geq 0$.
By induction, the sequence of necklaces $\neck{P_i}$, $i\in\{1,2\}$, is a cycle in the necklace graph~$N_n$ with vertex set $N(\cT_i)$.
Consequently, $\neck{P'}$ as defined in~\eqref{eq:P-ind} is a cycle with vertex set $N(\cT_1)\cup N(\cT_2)=N(\cT)$, as desired.

Observe from~\eqref{eq:lambda-mov}, \eqref{eq:lambda-plus}, \eqref{eq:alpha-bowtie} and~\eqref{eq:P-ind} that
\begin{equation*}
\lambda(\alpha(P'))=\lambda(\alpha(P_1))+\lambda(\alpha(P_2)).
\end{equation*}
Unrolling this inductive relation using Proposition~\ref{prop:Fn}~(i)+(iii), \eqref{eq:lambda-rev}, and~\eqref{eq:lambda-alpha-x}, we obtain that
\begin{equation}
\label{eq:total-shift-T}
\lambda(\alpha(P'))=\sum_{T\in\cT} \gamma_T\cdot \lambda(T),
\end{equation}
where the signs $\gamma_T\in\{+1,-1\}$ are determined by which of the gluing cycles $\sigma^i(C(x,y))$ and $\sigma^j(C(\hx,\hy))$ with $(x,y),(\hx,\hy)\in G(\cT)$, $i,j\geq 0$, are nested.

The relation~\eqref{eq:total-shift-T} allows us to compute the shift of flip sequences of periodic paths obtained by the gluing operation~$\bowtie$.
For example, consider the three periodic paths $P_1:=P(x\,0)$, $P_2:=P(\hx\,0)$, and $P_3:=P(\hy\,0)$ shown in Figure~\ref{fig:nested}, and the corresponding gluing cycles $C(x,y)$ and $C(\hx,\hy)$.
Note that the gluing cycles $\bigcup_{i\geq 0}\sigma^i(C(x,y))$ join the paths $\bigcup_{i\geq 0}\sigma^i(P_1\cup P_2)$, and the gluing cycles $\sigma^i(C(\hx,\hy))$ join the paths $\bigcup_{i\geq 0}\sigma^i(P_2\cup P_3)$.
As $\hx=\rho^{-1}(y)$, we have that $\sigma^i(C(x,y))$ and $\sigma^{i+1}(C(\hx,\hy))$ are nested for all $i\geq 0$ by Proposition~\ref{prop:Cxy}~(iii).
For $n=8$, $u=10$ and $v=11101000$ the corresponding flip sequences $\alpha_1:=\alpha(x\,0)$, $\alpha_2:=\alpha(\hx\,0)$, and $\alpha_3:=\alpha(\hy\,0)$ have the shifts $\lambda(\alpha_1)=n=8$, $\lambda(\alpha_2)=\lambda(\alpha_3)=2n=16$ (recall \eqref{eq:lambda-alpha-x}).
The 36th and 37th vertices on the periodic path $P_2\bowtie P_3$ are $\sigma^{-17}(y^1)$ and $\sigma^{-17}(y^0)$, respectively (recall that $y^0=y\,0$ and $y^1=f(y^0)$).
Consequently, $y^0$ and $y^1$ are the first two vertices on the periodic path $P_{23}:=\rev(\mov^{36}(\sigma^{17}(P_2\bowtie P_3)))$ with flip sequence $\alpha_{23}:=\rev(\mov^{36}(\alpha_2\bowtie \alpha_3-17))$.
The resulting flip sequence $\alpha:=\alpha_1\bowtie\alpha_{23}$ for the periodic path~$P:=P_1\bowtie P_{23}$ has shift $\lambda(\alpha)=\lambda(\alpha_1)-\big(\lambda(\alpha_2)+\lambda(\alpha_3)\big)=8-(16+16)=-24$.

\subsection{A first attempt at proving Theorem~\ref{thm:star}}
\label{sec:first}

Let $\cT_n$ be the spanning tree of~$\cH_n$ defined in Section~\ref{sec:Tn}, in particular, the node set of~$\cT_n$ is the set~$T_n$ of all plane trees with $n$ edges.
By Lemma~\ref{lem:Tn}, $G(\cT_n)$ is interleaving-free.
We fix the vertex $x_1:=1^n 0^{n+1}\in A_n\cup B_n$.
The set~$\cP_{G(\cT_n)}(x_1)$ defined in Section~\ref{sec:flip-T} contains a periodic path~$P$ with starting vertex~$x_1$ and second vertex~$f(x_1)$ in the middle levels graph~$M_n$ such that $\neck{P}$ has the vertex set~$N(\cT_n)=\bigcup_{[x]\in T_n} \neck{P(x\,0)}=\{\neck{x}\mid x\in A_n\cup B_n\}$, i.e., $\neck{P}$ is a Hamilton cycle in the necklace graph~$N_n$.
By~\eqref{eq:total-shift-T}, the corresponding flip sequence~$\alpha(P)$ has a shift of
\begin{equation}
\label{eq:total-shift-Tn}
\lambda(\alpha(P))=\sum_{T\in T_n} \gamma_T\cdot \lambda(T)
\end{equation}
for some signs $\gamma_T\in\{+1,-1\}$ that are determined by which gluing cycles encoded by~$\cT_n$ are nested.

With $s:=\lambda(\alpha(P))$ we define the flip sequences
\begin{equation}
\label{eq:alpha-first}
\alpha_0:=\alpha(P), \quad \alpha_i:=\alpha_0+i\cdot s \text{ for } i=1,\ldots,2n.
\end{equation}
If we apply the entire flip sequence $(\alpha_0,\alpha_1,\ldots,\alpha_{2n})$ to the starting vertex~$x_1$ in the middle levels graph~$M_n$, then we reach the vertex~$\sigma^{-i\cdot s}(x_1)$ after applying all flips in $(\alpha_0,\alpha_1,\ldots,\alpha_{i-1})$ for every $i=1,\ldots,2n+1$.
Consequently, if $s$ and $2n+1$ happen to be coprime, then we reach~$x_1$ only after applying the entire flip sequence, and as $\alpha_0=\alpha(P)$ is the flip sequence of the Hamilton cycle~$\neck{P}$ in the necklace graph~$N_n$, the resulting sequence~$C$ of bitstrings is a Hamilton cycle in the middle levels graph~$M_n$.
A star transposition Gray code for $(n+1,n+1)$-combinations satisfying the conditions of Theorem~\ref{thm:star} is then obtained from~$C$ by prefixing every bitstring of~$C$ with~1 or~0, alternatingly.

However, the aforementioned approach requires that $s=\lambda(\alpha(P))$ and $2n+1$ are coprime, which may not be the case.
Consequently, in the next section we present a technique to modify $\alpha(P)$ to another flip sequence $\alpha'$, so that $s':=\lambda(\alpha')$ is coprime to~$2n+1$.

\section{Switches}
\label{sec:switch}

In this section we develop a systematic way to modify the shift value of flip sequences, following the ideas outlined in Section~\ref{sec:idea-shift}, which allows us to prove Theorem~\ref{thm:star}.

\subsection{Switches and their shift}

For two bitstrings that differ in a single bit, we write $p(x,y)$ for the position in which~$x$ and~$y$ differ.
We say that a triple of vertices $\tau=(x,y,y')$ with $x\in A_n$, $y,y'\in B_n$ and $y\neq y'$ is a \emph{switch}, if $x$ differs both from~$y$ and from~$y'$ in a single bit, and $\neck{y}=\neck{y'}$.
In the necklace graph~$N_n$, a switch can be considered as a multiedge~$(\neck{x},\neck{y})=(\neck{x},\neck{y'})$.
The \emph{shift} of a switch~$\tau=(x,y,y')$, denoted $\lambda(\tau)$, is defined as the integer~$i$ such that $y=\sigma^i(y')$.
For example $\tau=(1110000,1110001,1111000)$ is a switch, as we have $\neck{1110001}=\neck{1111000}$, and its shift is $\lambda(\tau)=1$, as $1110001=\sigma^{1}(1111000)$.
We denote a switch $\tau=(x,y,y')$ compactly by writing~$x$ with the 0-bit at position~$p(x,y)$ underlined, and the 0-bit at position~$p(x,y')$ overlined.
The switch~$\tau$ from before is denoted compactly as $\tau=111\ol{0}00\ul{0}$.
Note that for any switch~$\tau=(x,y,y')$, the inverted switch~$\tau^{-1}:=(x,y',y)$ has shift~$\lambda(\tau^{-1})=-\lambda(\tau)$.
For example, for $\tau=111\ol{0}00\ul{0}$, the switch~$\tau^{-1}=111\ul{0}00\ol{0}$ has shift~$\lambda(\tau^{-1})=-1$.
Clearly, cyclically rotating a switch yields another switch with the same shift.
Similarly, reversing a switch yields another switch with the negated shift.
For example, the switch $\sigma(\tau)=11\ol{0}00\ul{0}1$ has shift~$+1$, and its reversed switch $1\ul{0}00\ol{0}11$ has shift~$-1$.

\subsection{Modifying flip sequences by switches}

The idea of a switch~$\tau=(x,y,y')$ is simple and yet very powerful:
Consider a flip sequence~$\alpha=(a_1,\ldots,a_k)$ with shift~$\lambda(\alpha)$ for a periodic path~$P=(x_1,\ldots,x_k)$, and let~$x_{k+1}$ be the vertex obtained from~$x_k$ by flipping the bit at position~$a_k$.
If we have $(x_i,x_{i+1})=(x,y)$ for some $i\in\{1,\ldots,k\}$, then the modified flip sequence
\begin{subequations}
\label{eq:switch}
\begin{equation}
\label{eq:swalpha1}
\alpha':=\big(a_1,\ldots,a_{i-1},p(x,y'),a_{i+1}+\lambda(\tau),\ldots,a_k+\lambda(\tau)\big)
\end{equation}
produces a periodic path~$P'=(x_1',\ldots,x_k')$ that visits necklaces in the same order as~$P$, i.e., we have $\neck{x_i}=\neck{x_i'}$ for all~$i=1,\ldots,k$, and we have
\begin{equation}
\label{eq:swlambda1}
\lambda(\alpha')=\lambda(\alpha)+\lambda(\tau).
\end{equation}
The situation where~$(x_i,x_{i+1})=(x,y')$ is symmetric, and can be analyzed with these equations by considering the inverted switch~$\tau^{-1}$ with~$\lambda(\tau^{-1})=-\lambda(\tau)$.

Similarly, if we have $(x_i,x_{i+1})=(y',x)$ for some $i\in\{1,\ldots,k\}$, then the modified flip sequence
\begin{equation}
\label{eq:swalpha2}
\alpha':=\big(a_1,\ldots,a_{i-1},p(x,y)+\lambda(\tau),a_{i+1}+\lambda(\tau),\ldots,a_k+\lambda(\tau)\big)
\end{equation}
produces a periodic path~$P'=(x_1',\ldots,x_k')$ that visits necklaces in the same order as~$P$, and we have
\begin{equation}
\label{eq:swlambda2}
\lambda(\alpha')=\lambda(\alpha)+\lambda(\tau).
\end{equation}
\end{subequations}
Again, the situation where~$(x_i,x_{i+1})=(y,x)$ is symmetric, and can be analyzed with these equations by considering the inverted switch~$\tau^{-1}$ with~$\lambda(\tau^{-1})=-\lambda(\tau)$.

In particular, if~$\neck{P}$ is a Hamilton cycle in the necklace graph~$N_n$, then $\neck{P'}$ is also a Hamilton cycle in the necklace graph, albeit one whose flip sequence has a different shift (as given by~\eqref{eq:swlambda1} and~\eqref{eq:swlambda2}).

For example, consider the flip sequence~$\alpha=6253462135$, which starting from~$x_1=1110000$ produces the periodic path~$P=(x_1,\ldots,x_{10})$ and the vertex~$x_{11}$ shown on the top left hand side of Figure~\ref{fig:c44} (recall that we omit the first bit here), and we have $\lambda(\alpha)=+1$.
For the switch~$\tau=(x,y,y')=101\ol{0}\ul{0}10$ with $\lambda(\tau)=+5$ we have~$(x_3,x_4)=(x,y)$, and according to~\eqref{eq:swalpha1} the flip sequence~$\alpha'=(6,2,p(x,y'),3+5,4+5,6+5,2+5,1+5,3+5,5+5)=6241247613$ has shift $\lambda(\alpha')=\lambda(\alpha)+\lambda(\tau)=+1+5=+6$ and produces a periodic path~$P'$ that visits necklaces in the same order as~$P$.
The path~$P'$ is shown as the rightmost solution in Figure~\ref{fig:c44}.

\subsection{Construction of switches}

We now describe a systematic way to construct many distinct switches from the canonic switch $\tau=1^n \ol{0}0^{n-1}\ul{0}$, which has shift~$\lambda(\tau)=+1$.

For any integers $n\geq 1$, $d\geq 1$ and $1\leq s\leq d$, the \emph{$(s,d)$-orbit} is the maximal prefix of the sequence~$s+id$, $i\geq 0$, considered modulo~$2n+1$, in which all numbers are distinct.
Clearly, the number of distinct $(s,d)$-orbits for fixed $d$ and $s\geq 1$ is $n_d:=\gcd(2n+1,d)$, and the length of each orbit is $\ell_d:=(2n+1)/\gcd(2n+1,d)$.
Note that both $n_d$ and $\ell_d$ are odd integers.
For example, for $n=10$ and $d=6$ there are $n_d=3$ orbits of length $\ell_d=7$, namely the $(1,6)$-orbit $(1,7,13,19,4,10,16)$, the $(2,6)$-orbit $(2,8,14,20,5,11,17)$, and the $(3,6)$-orbit $(3,9,15,21,6,12,18)$.
For any $n\geq 1$, we let $Z_n$ denote the set of all binary strings of length~$2n$ with exactly $n$ many~0s and $n$ many~1s.
For instance, we have $Z_2=\{1100,1010,1001,0110,0101,0011\}$.

The base case of our definition is the switch $\tau_{n,1}:=1^n \ol{0}0^{n-1}\ul{0}$, which has shift $\lambda(\tau_{n,1})=+1$.
For any integer $2\leq d\leq n$ that is coprime to~$2n+1$, we let $\tau_{n,d}$ denote the sequence whose entries at the positions given by the $(1,d)$-orbit equal the sequence~$\tau_{n,1}$, including the underlined and overlined bit.
In words, $\tau_{n,d}$ is obtained by filling the entries of~$\tau_{n,1}$ one by one into every $d$th position of~$\tau_{n,d}$, starting at the first one.

For any integer $3\leq d\leq n$ that is not coprime to~$2n+1$, we choose an arbitrary bitstring $z=(z_2,\ldots,z_{n_d})\in Z_{(n_d-1)/2}$, and we let $\tau_{n,d,z}$ denote the sequence whose entries at the positions given by the $(1,d)$-orbit equal the sequence $\tau_{(\ell_d-1)/2,1}$, including the underlined and overlined bit, and for $j=2,\ldots,n_d$, all entries at the positions given by the $(j,d)$-orbit equal~$z_j$.
In words, $\tau_{n,d}$ is obtained by filling the entries of~$\tau_{(\ell_d-1)/2,1}$ one by one into every $d$th position of~$\tau_{n,d}$, starting from the first one, and then filling the gaps between these entries by copies of~$z$.
Clearly, the number of choices we have for~$z$ in this construction is~$\binom{n_d-1}{(n_d-1)/2}$.

Note that the construction for coprime~$d$ can be understood as a special case of the construction for non-coprime~$d$ with~$n_d=1$ and $z=\varepsilon$.

These definitions are illustrated in Figure~\ref{fig:switch} for $n=1,\ldots,7$.
The next lemma follows immediately from these definitions.
It asserts that the sequences $\tau_{n,d}$ and $\tau_{n,d,z}$ defined before are indeed switches with shift~$d$.

\begin{lemma}
\label{lem:switch}
Let $n\geq 1$.
For any integer $1\leq d\leq n$ that is coprime to~$2n+1$, the sequence $\tau_{n,d}$ defined before is a switch with~$\lambda(\tau_{n,d})=d$.
For any integer $3\leq d\leq n$ that is not coprime to~$2n+1$ and any bitstring $z\in Z_{(n_d-1)/2}$, the sequence $\tau_{n,d,z}$ defined before is a switch with~$\lambda(\tau_{n,d,z})=d$.
\end{lemma}

\begin{figure}
\includegraphics{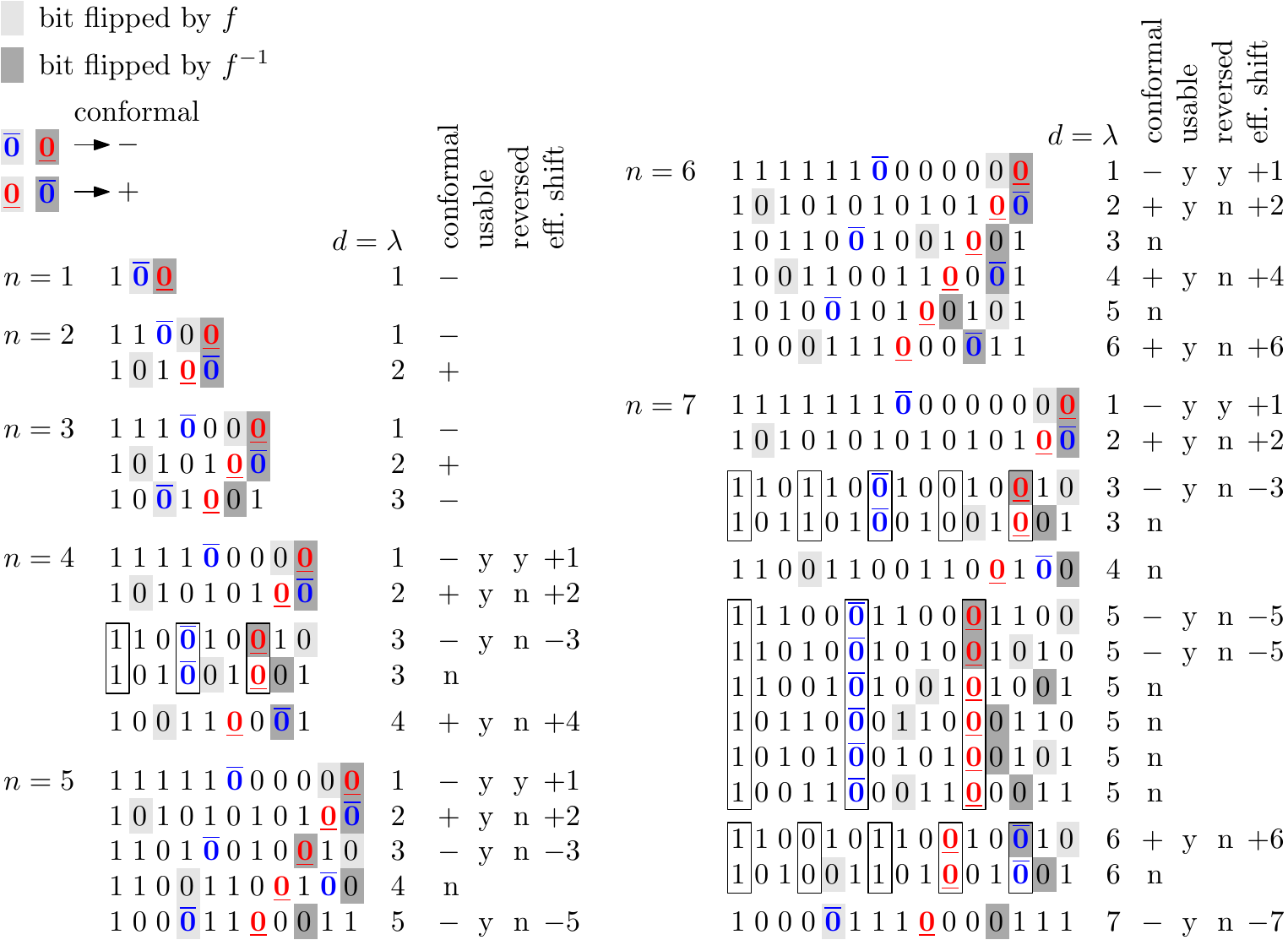}
\caption{All switches for $n=1,\ldots,7$.
The switch $\tau_{n,1}$ is shown as the first switch in each block, and the remaining switches are ordered by increasing~$d$.
The bits flipped by~$f$ and~$f^{-1}$ are marked light gray and dark gray, respectively.
The framed bits belong to a $(1,d)$-orbit for some~$d$ that is not coprime to~$2n+1$.
Whether a switch is $f$-conformal or $f^{-1}$-conformal is indicated by $+$ or $-$, respectively, and by `n' if neither of the two.
Similarly, a switch being usable or reversed is indicated by `y'=yes and `n'=no.
The resulting effective shifts are shown in the rightmost column.
The latter three properties are w.r.t.~$G(\cT_n)$ for the spanning tree~$\cT_n$ defined in Section~\ref{sec:Tn-new}.
}
\label{fig:switch}
\end{figure}

One may ask whether there are switches other than the ones described by Lemma~\ref{lem:switch}.
It can be shown that this is not the case, i.e., every possible switch is obtained in one of the two ways described by the lemma, and by reversal and cyclic rotations.

\subsection{Interaction with other structures}
\label{sec:interaction}

We now describe how the switches defined before interact with the periodic paths introduced in Section~\ref{sec:paths} and with the gluing cycles introduced in Section~\ref{sec:gluing}.

Recall the definition of the function~$f$ from~\eqref{eq:f}.
We say that a switch $\tau=(x,y,y')$ is \emph{$f$-conformal}, if $y=f(x)$ or if $x=f(y')$, and then we refer to~$(x,y)$ or~$(y',x)$, respectively, as the \emph{$f$-edge} of the switch.
Also, we say that~$\tau$ is \emph{$f^{-1}$-conformal}, if the inverted switch $\tau^{-1}$ is $f$-conformal, and we refer to the $f$-edge of~$\tau^{-1}$ also as the $f$-edge of~$\tau$.
A switch being $f$-conformal means that its $f$-edge belongs to a periodic path defined in~\eqref{eq:Px}.

Given a set of gluing pairs~$G\subseteq G_n$, we say that an $f$-conformal or $f^{-1}$-conformal switch~$\tau$ is \emph{usable w.r.t.~$G$}, if for every gluing pair~$(\hx,\hy)\in G$ and all $i\geq 0$, the three $f$-edges of the gluing cycle~$\sigma^i(C(\hx,\hy))$ defined in~\eqref{eq:Cxy}, i.e., the edges $\sigma^i((\hx^0,\hx^1))$, $\sigma^i((\hx^5,\hx^6))$ and $\sigma^i((\hy^0,\hy^1))$ as defined in~\eqref{eq:xyi} are distinct from the $f$-edges of~$\tau$.
Recall from~\eqref{eq:Pxy} and~\eqref{eq:P12} that the three $f$-edges are \emph{removed} when joining periodic paths, so a switch whose $f$-edge is one of the removed edges would not be relevant for us.

\begin{lemma}
\label{lem:usable}
Consider an $f^{-1}$-conformal switch $\tau=(x,y,y')$ with $f$-edge $(y,x)$ for which $t(x)=\cdots 00$.
Then $\tau$ is usable w.r.t.\ any set of gluing pairs $G\subseteq G_n$.
\end{lemma}

\begin{proof}
To show that $\tau$ is usable w.r.t.~$G$, we need to verify that the $f$-edge $(y,x)$ of~$\tau$ is distinct from the $f$-edges of any gluing cycle $\sigma^i(C(\hx,\hy))$ with $(\hx,\hy)\in G$ and $i\geq 0$.
As $y\in B_n$ and $x\in A_n$, it is enough to verify this for the $f$-edge $\sigma^i((\hx^5,\hx^6))$.
Recall from~\eqref{eq:xyi} that $\hx^6=1\,0\,0\,u\,1\,v\,0$ for some $u,v\in D$ and therefore $t(\hx^6)=u\,1\,v\,0\,1\,0$, i.e., this string ends with~$10$.
As the string~$t(x)$ ends with~$00$ by our assumptions in the lemma, we have $t(\hx^6)\neq t(x)$, proving the claim.
\end{proof}

\begin{lemma}
\label{lem:tau12}
Let $n\geq 4$.
The switch $\tau_{n,1}=:(x,y,y')$ has the $f$-edge $(y,x)$ and is $f^{-1}$-conformal.
The switch $\tau_{n,2}=:(x,y,y')$ has the $f$-edge $(y',x)$ and is $f$-conformal.
Moreover, both switches are usable w.r.t.\ any set of gluing pairs $G\subseteq G_n$.
\end{lemma}

\begin{proof}
We first consider the switch $\tau_{n,1}$.
By definition, we have $\tau_{n,1}=1^n\ol{0}0^{n-1}\ul{0}=:(x,y,y')$, i.e., $x\in A_n$ and $y\in B_n$ differ in the last bit.
We clearly have $t(x)=1^n0^n$.
By the definition of~$f:B_n\rightarrow A_n$ in~\eqref{eq:ffx0} we thus have $x=f(y)$, proving that $(y,x)$ is an $f$-edge, which implies that the switch is $f^{-1}$-conformal.
Note that $t(x)=\cdots 00$ for $n\geq 4$, and therefore $\tau_{n,1}$ is usable w.r.t.~$G$ by Lemma~\ref{lem:usable}.

It remains to consider the switch $\tau_{n,2}$.
By definition, we have $\tau_{n,2}=(10)^{n-1}1\ul{0}\ol{0}=:(x,y,y')$, i.e., $x\in A_n$ and $y'\in B_n$ differ in the last bit.
We clearly have $t(x)=(10)^n$.
By the definition of $f:B_n\rightarrow A_n$ in~\eqref{eq:ffx0} we thus have $x=f(y')$, proving that $(y',x)$ is an $f$-edge, which implies that the switch is $f$-conformal.
To show that $\tau_{n,2}$ is usable w.r.t.~$G$, we consider the tree $t':=\rho^{-3}(t(x))=1(10)^{n-1}0=s_n$, which is the star with $n$ edges rooted at a leaf.
It follows trivially from the definition~\eqref{eq:gluing} that for every gluing pair~$(\hx,\hy)\in G$ we have $\hx=t(\hx^0)\neq t'$ and therefore $t(\hx^6)=\rho^3(t(\hx^0))\neq \rho^3(t')=t(x)$ (recall Proposition~\ref{prop:Fn}~(i)).
\end{proof}

\begin{lemma}
\label{lem:taundz}
Let $n\geq 11$ and consider integers $c,d\geq 3$ such that $2n+1=c\cdot d$.
Then the switch $\tau_{n,d,z}=:(x,y,y')$ with $z:=1^{(d-1)/2}0^{(d-1)/2}\in Z_{(d-1)/2}$ has the $f$-edge $(y,x)$ and is $f^{-1}$-conformal.
Moreover, it is usable w.r.t.~$G(\cT_n)$ for the spanning tree~$\cT_n$ of~$\cH_n$ defined in Section~\ref{sec:Tn}.
\end{lemma}

Note that since $2n+1$ is odd, both $c$ and~$d$ are also odd integers.

\begin{proof}
By definition, we have $\tau_{n,d,z}=(1\,z)^{(c-1)/2}\,\ol{0}\,z\,(0\,z)^{(c-3)/2}\,\ul{0}\,z=:(x,y,y')$, i.e., $x\in A_n$ and $y\in B_n$ differ in the bit before the suffix~$z$.
Note that $z$ is a Dyck word by definition, and consequently we have
\begin{equation}
\label{eq:tx1}
t(x)=z\,(1\,z)^{(c-1)/2}(0\,z)^{(c-1)/2}.
\end{equation}
By the definition of $f:B_n\rightarrow A_n$ in~\eqref{eq:ffx0} we thus have $x=f(y)$, proving that $(y,x)$ is an $f$-edge, which implies that the switch is $f^{-1}$-conformal.
From~\eqref{eq:tx1} we see that $t(x)=\cdots z$, so if $d\geq 5$, then $t(x)=\cdots 00$, and therefore $\tau_{n,1}$ is usable w.r.t.~$G(\cT_n)$ by Lemma~\ref{lem:usable}.

It remains to consider the case $d=3$, i.e., we have $c=(2n+1)/3$ and $z=10$.
As $2n+1$ is not divisible by~3 for $n=11$ and $n=12$, we have in particular that $n\geq 13$.
In this case the equation~\eqref{eq:tx1} can be written as
\begin{equation}
\label{eq:tx2}
t(x)=(z\,1)^{(c-1)/2}\,z\,(0\,z)^{(c-1)/2}=(101)^{(c-1)/2}10(010)^{(c-1)/2},
\end{equation}
i.e., the rooted tree~$t(x)$ is obtained from a path of length~$(c+1)/2$ rooted at one of its end vertices by attaching a pending edge to the left and right of each vertex in distance at most~$(c-3)/2$ from the root.
As a sanity check, the total number of edges of the tree~$t(x)$ is $(c+1)/2+2(c-1)/2=c/2+1/2+c-1=(3c-1)/2=n$.
Now consider the rooted tree~$t':=\rho^{-3}(t(x))=110100(101)^{(c-3)/2}10(010)^{(c-3)/2}=1\,1\,0\,u\,0\,1\,0\,w\,1\,0=1\,1\,0\,u\,0\,v$ with $u:=10\in D$, $w:=1(101)^{(c-5)/2}10(010)^{(c-5)/2}0\in D$ and $v:=1\,0\,w\,1\,0\in D$.
Using Lemma~\ref{lem:centroid}, one can verify directly that for $n\geq 13$, the unique centroid~$c$ of the pullable tree~$t'$ (recall~\eqref{eq:gluing}) lies in the subtree~$w$ in distance at least~1 from the root, and consequently the leftmost leaf~$a$ of~$t'$ (which is in distance~2 from the root) is not the leftmost leaf of the $c$-subtree of~$[t']$ containing~$a$.
Consequently, by the definition of~$\cT_n$ in Section~\ref{sec:Tn}, for every gluing pair~$(\hx,\hy)\in G(\cT_n)$ we have $\hx=t(\hx^0)\neq t'$ and therefore $t(\hx^6)=\rho^3(t(\hx^0))\neq \rho^3(t')=t(x)$ (recall Proposition~\ref{prop:Fn}~(i)).
\end{proof}

\subsection{Number-theoretic lemmas}
\label{sec:number}

Recall that the proof attempt presented in Section~\ref{sec:first} may yield a flip sequence with a shift~$s$ that is not coprime to~$2n+1$.
We will correct this by applying one or two switches so that the resulting modified flip sequence has a shift of~$s':=s+d$ or~$s':=s-d$ for some integer~$d\geq 1$, such that~$s'$ is coprime to~$2n+1$.
It is important to realize that while we can choose the value of~$d$ by picking a suitable switch, we have no control whether the modification will result in a shift of~$s+d$ or~$s-d$, as we have no information about whether the $f$-edge of the switch is oriented conformly or oppositely along the flip sequence obtained from gluing (this is determined by the nesting of the gluing cycles).
This is why the following two basic number-theoretic lemmas, which provide the basis for these modifications, always make claims about both numbers~$s+d$ and~$s-d$.

For any integer~$n\geq 1$ we write $\cP(n)$ for the set of prime factors of~$n$ (without multiplicities).
Moreover, for any~$s\in\{0,1,\ldots,n-1\}$ we define $\cP(n,s):=\cP(n)\setminus \cP(s)$ if $s>0$ and $\cP(n,0):=\emptyset$.
For example, for $n=3^2\cdot 5\cdot 11^3=59.895$ and $s=3^2\cdot 7^3\cdot 11=33.957$ we have $\cP(n)=\{3,5,11\}$, $\cP(s)=\{3,7,11\}$ and $\cP(n,s)=\{5\}$.
\todo{There might be some confusion here with the symbol $\cP$ used already earlier in the context of gluing, but I don't know a good solution for this now.}

\begin{lemma}
\label{lem:coprime-change}
Let $n\geq 1$ be such that $2n+1$ is not a prime power, and consider an integer $s\in\{0,\ldots,2n\}$ that is not coprime to~$2n+1$.
If $\cP(2n+1,s)\neq \emptyset$, then for $d:=\prod_{p\in \cP(2n+1,s)}p$ both numbers $s+d$ and $s-d$ are coprime to~$2n+1$.
If $\cP(2n+1,s)=\emptyset$, then for any $d\in \cP(2n+1)$ we have $\cP(2n+1,s+d)=\cP(2n+1)\setminus \{d\}\neq \emptyset$ and $\cP(2n+1,s-d)=\cP(2n+1)\setminus \{d\}\neq \emptyset$.
\end{lemma}

\begin{proof}
We first consider the case $\cP(2n+1,s)\neq \emptyset$, which in particular means that $s\neq 0$.
By the definition of~$d$, we have $s+d=s\neq 0\pmod{p}$ for any $p\in \cP(2n+1,s)$.
Moreover, we have $s+d=d\neq 0\pmod{p}$ for any $p\in \cP(2n+1)\cap \cP(s)$, proving that $s+d$ is coprime to~$2n+1$.
The argument that $s-d$ is coprime to~$2n+1$ is analogous.

We now consider the case $\cP(2n+1,s)=\emptyset$, and we fix some $d\in \cP(2n+1)$.
Observe that $s+d=d\neq 0\pmod{p}$ for any $p\in \cP(2n+1)\setminus \{d\}$, and $s+d=s=0\pmod{d}$, proving that $\cP(2n+1,s+d)=\cP(2n+1)\setminus \{d\}$.
As $2n+1$ is not a prime power, $\cP(2n+1)$ has at least two distinct elements, and therefore $\cP(2n+1)\setminus\{d\}$ is nonempty.
The argument that $s-d$ satisfies $\cP(2n+1,s-d)=\cP(2n+1)\setminus \{d\}\neq \emptyset$ is analogous.
\end{proof}

The following lemma can be easily verified by hand.
We omit the details.

\begin{lemma}
\label{lem:coprime-small}
For any integer $4\leq n\leq 10$ and any integer~$s\in\{0,\ldots,2n\}$ that is not coprime to~$2n+1$, both numbers in at least one of the sets $\{s-1,s+1\}$, $\{s-2,s+2\}$, $\{s-1,s+2\}$, $\{s+1,s-2\}$ are coprime to~$2n+1$.
\end{lemma}

\subsection{Proof of Theorem~\ref{thm:star}}
\label{sec:star-proof}

We now have all ingredients in hand to prove Theorem~\ref{thm:star}.

\begin{proof}[Proof of Theorem~\ref{thm:star}]
By the scaling trick described in Section~\ref{sec:idea-shift}, it is enough to establish the theorem for each $n\geq 1$ for one particular value of~$s$ that is coprime to~$2n+1$.

For $n=1$ we can use the flip sequence $\alpha_0:=21$, which starting from $x_1:=100$ (first bit omitted) has a shift of~$s=1$.
For $n=2$ we can use $\alpha_0:=5135$, which starting from $x_1:=11000$ has a shift of~$s=1$.
For $n=3$ we can use $\alpha_0:=6253462135$, which starting from $x_1:=1110000$ has a shift of~$s=1$ (see the left hand side in Figure~\ref{fig:c44}).

For the rest of the proof we assume that~$n\geq 4$.
We consider the spanning tree~$\cT_n\subseteq\cH_n$ defined in Section~\ref{sec:Tn}.
As explained in Section~\ref{sec:first}, based on the spanning tree~$\cT_n$, we define a periodic path~$P$ with starting vertex~$x_1:=1^n0^{n+1}$ and second vertex~$f(x_1)$ in the middle levels graph~$M_n$, such that~$\neck{P}$ is a Hamilton cycle in the necklace graph~$N_n$, and the shift of the corresponding flip sequence $\alpha(P)$ is given by~\eqref{eq:total-shift-Tn}.
We denote this value by $s:=\lambda(\alpha(P))$.
If $s$ is coprime to~$2n+1$, we are done.
It remains to consider the case that $s$ is not coprime to~$2n+1$.

If $4\leq n\leq 10$, then we consider the switches $\tau_{n,1}$ and $\tau_{n,2}$, which are $f^{-1}$- and $f$-conformal, respectively, and usable w.r.t.\ to $G(\cT_n)$ by Lemma~\ref{lem:tau12}.
By Lemma~\ref{lem:switch}, their shifts are $\lambda(\tau_{n,1})=1$ and $\lambda(\tau_{n,2})=2$, respectively.
Consequently, by modifying the flip sequence~$\alpha(P)$ as described by~\eqref{eq:swalpha1} and~\eqref{eq:swalpha2} using one of the two or both switches, we obtain a flip sequence~$\alpha'$ that has shift
\begin{subequations}
\label{eq:new-shifts}
\begin{equation}
s':=\lambda(\alpha')=s+\chi_1\cdot\gamma_1\cdot 1+\chi_2\cdot\gamma_2\cdot 2
\end{equation}
for some signs $\gamma_1,\gamma_2\in\{-1,+1\}$ (recall~\eqref{eq:swlambda1} and~\eqref{eq:swlambda2} and the remarks at the beginning of Section~\ref{sec:number}), and with indicators $\chi_1,\chi_2\in\{0,1\}$ that are non-zero iff the corresponding switches are used.
If $s-1$ and $s+1$ are coprime to~$2n+1$, or if $s-1$ is coprime to~$2n+1$ and $\gamma_1=-1$, or if $s+1$ is coprime to~$2n+1$ and $\gamma_1=+1$, then we only use the switch~$\tau_{n,1}$, guaranteeing that $s'=s+\gamma_1$ is coprime to~$2n+1$.
Similarly, if $s-2$ and $s+2$ are coprime to~$2n+1$, or if $s-2$ is coprime to~$2n+1$ and $\gamma_2=-2$, or if $s+2$ is coprime to~$2n+1$ and $\gamma_2=+1$, then we only use the switch~$\tau_{n,2}$, guaranteeing that $s'=s+\gamma_2$ is coprime to~$2n+1$.
Otherwise, by Lemma~\ref{lem:coprime-small} we have that either $s-1$ and $s+2$ are coprime to~$2n+1$ and $(\gamma_1,\gamma_2)=(+1,-1)$, or $s+1$ and $s-2$ are coprime to~$2n+1$ and $(\gamma_1,\gamma_2)=(-1,+1)$, and then we use both switches~$\tau_{n,1}$ and~$\tau_{n,2}$.
In the first case $s'=s+1-2=s-1$ is coprime to~$2n+1$, and in the second case $s'=s-1+2=s+1$ is coprime to~$2n+1$.

If $n\geq 11$, then we distinguish three cases:
If $2n+1$ is a prime power, then $s$ is also a power of the same prime number.
We apply the switch~$\tau_{n,1}$, similarly to before, modifying the flip sequence~$\alpha(P)$ so that the resulting flip sequence~$\alpha'$ has shift
\begin{equation}
s'=\lambda(\alpha')=s+\gamma_1\cdot 1,
\end{equation}
for some $\gamma_1\in\{+1,-1\}$, and $s'=s\pm 1$ is coprime to~$2n+1$.

If $2n+1$ is not a prime power and $\cP(2n+1,s)\neq \emptyset$, then we define $d:=\prod_{p\in \cP(2n+1,s)}p$ and $c:=(2n+1)/d$ and we consider the switch~$\tau_{n,d,z}$ defined in Lemma~\ref{lem:taundz}, which is $f^{-1}$-conformal and usable w.r.t.~$G(\cT_n)$.
We apply the switch~$\tau_{n,d,z}$, modifying the flip sequence~$\alpha(P)$ so that the resulting flip sequence~$\alpha'$ has shift
\begin{equation}
\label{eq:spd}
s':=\lambda(\alpha')=s+\gamma_d\cdot d
\end{equation}
\end{subequations}
for some $\gamma_d\in\{+1,-1\}$, and $s'=s\pm d$ is coprime to~$2n+1$ by the first part of Lemma~\ref{lem:coprime-change}.

If $2n+1$ is not a prime power and $\cP(2n+1,s)=\emptyset$, then we pick some $d\in \cP(2n+1)$, define $c:=(2n+1)/d$ and we apply the switch~$\tau_{n,d,z}$, yielding a flip sequence~$\alpha'$ with shift~$s'$ given by~\eqref{eq:spd}, which satisfies $\cP(2n+1,s')\neq \emptyset$ by the second part of Lemma~\ref{lem:coprime-change}.
We then modify the flip sequence a second time as described in the previous case, and the switch used is distinct from the first one, as $d':=\prod_{p\in\cP(2n+1,s\pm d)}p=\prod_{p\in\cP(2n+1)\setminus \{d\}}p$ clearly satisfies $d'\neq d$.

This completes the proof of the theorem.
\end{proof}

The reader may wonder whether the proof presented before would work only with the switches~$\tau_{n,1}$ and~$\tau_{n,2}$, not using the switch $\tau_{n,d,z}$ at all.
However, the statement of Lemma~\ref{lem:coprime-small} becomes invalid for large enough values of~$n$.
The first counterexample is~$n=52$ with $2n+1=105=3\cdot 5\cdot 7$ and $s=5$, for which none of the three numbers $s-2=3$, $s+1=6$ and $s+2=7$ is coprime to~$2n+1$.

\section{Efficient computation}
\label{sec:algo}

The proof of Theorem~\ref{thm:star} presented in the previous section is constructive and translates directly into an algorithm for computing the corresponding flip sequence.
Recall from the discussion in Section~\ref{sec:idea-algo} that there is one major obstacle to make the algorithm efficient, and this is to compute the value of the shift of the flip sequence~$\alpha(P)$ after gluing, i.e., the number~$s:=\lambda(\alpha(P))$ (recall \eqref{eq:total-shift-Tn}), upon initialization of the algorithm.
Only with knowing the value of~$s$, we can decide whether it is coprime to~$2n+1$, and consequently which switches have to be applied in the course of the algorithm to modify~$s$.
Unfortunately, it is not clear how to compute the number~$s$ efficiently, as the signs $\gamma_T\in\{+1,-1\}$ in~\eqref{eq:total-shift-Tn} are determined by which gluing cycles are nested, and the number of cycles, equal to the number of plane trees in~$T_n$, is exponential in~$n$.

However, if we knew that $G(\cT_n)$ is not only interleaving-free, but also \emph{nesting-free}, then this would guarantee that all signs~$\gamma_T$ in~\eqref{eq:total-shift-Tn} are positive, giving
\begin{equation}
\label{eq:lambda-Cn}
s=\lambda(\alpha(P))=\sum_{T\in T_n} \lambda(T)=C_n,
\end{equation}
where $C_n$ is the $n$th Catalan number.
To see the identity used in the last step, recall that $\lambda(T)$ counts all rooted trees whose underlying plane tree is~$T$, so overall we count all rooted trees, which gives the sum~$|D_n|=C_n$.

Unfortunately, the spanning tree~$\cT_n$ defined in Section~\ref{sec:Tn} is not nesting-free in general.
Consequently, to implement the approach outlined before, in the following we define another spanning tree~$\cT_n$ of~$\cH_n$ such that $G(\cT_n)$ is both interleaving-free and nesting-free (see Lemma~\ref{lem:Tn-new} below).
This alternative definition of~$\cT_n$ is considerably more complicated than the one presented in Section~\ref{sec:Tn}.
All other ingredients of the construction will work essentially the same way.

\subsection{Redefinition of~$\cT_n$}
\label{sec:Tn-new}

We define the rooted trees
\begin{equation}
\label{eq:spires}
\begin{alignedat}{4}
q_0&:=10,& q_1&:=1100,& q_2&:=110100,& q_3&:=11100100, \\
q_4&:=11010100,&\quad q_5&:=1110100100,&\quad q_6&:=1110010100,&\quad q_7&:=1110011000, \\
q_8&:=1101011000,&\quad q_9&:=1101010100,& &&& \\
\end{alignedat}
\end{equation}
see Figure~\ref{fig:spire}.

For $n\geq 4$, we define a subgraph~$\cT_n$ of~$\cH_n$ as follows:
For every plane tree~$T\in T_n$ with $T\neq [s_n]$, we define a gluing pair~$(x,y)\in G_n$ with either $T=[x]$ or $T=[y]$.
We let $\cT_n$ be the spanning subgraph of~$\cH_n$ given by the union of arcs~$([x],[y])$ labeled~$(x,y)$ for all gluing pairs~$(x,y)$ obtained in this way.
The definition of the gluing pair~$(x,y)\in G_n$ for a given plane tree~$T\neq [s_n]$ proceeds in the following three steps (T1)--(T3), unless $n$ is odd and $T=[d_n]$, in which case the special rule~(D) is applied.

\textbf{(D) Dumbbell rule.} If $n$ is odd and $T=[d_n]$, we let $c$ be one of its centroids, which has exactly one $c$-subtree that is not a single edge, namely the tree~$s_{(n+1)/2}$.
The rightmost leaf~$a$ of it is thick and pushable to~$c$ in~$T$, so we define $y:=y(T,c,a)=d_n'$ and $x:=\push(y)$ as in Lemma~\ref{lem:pot-push}.

\textbf{(T1) Fix the centroid and subtree ordering.}
If $T$ has two centroids, we let $c$ denote the centroid whose active $c$-subtrees are not all single edges.
If this is true for both centroids, we let $c$ be the one for which all active $c$-subtrees $t_1,\ldots,t_k$, listed in ccw order such that $t_1$ is the first tree encountered after the $c$-subtree containing the other centroid, give the lexicographically minimal string $(t_1,\ldots,t_k)$.
\todo{We do not use this in the proofs yet, but this lexicographically minimal centroid is the same in the entire subtree~$\cS^2$.
This might be a useful property for our algorithms.}

If $T$ has a unique centroid, we denote it by~$c$.
We consider all $c$-subtrees of~$T$, and we denote them by $t_1,\ldots,t_k$, i.e., $T=[(t_1,\ldots,t_k)]$, such that among all possible ccw orderings of subtrees around~$c$, the string $(t_1,t_2,\ldots,t_k)$ is lexicographically minimal.

\textbf{(T2) Select $c$-subtree of~$T$.}
If $T$ has two centroids, we let $t_\hi$ be the first of the trees $t_1,\ldots,t_k$ that is distinct from~$q_0$.

If $T$ has a unique centroid, then for each of the following conditions (i)--(iv), we consider all trees $t_i$ for $i=1,\ldots,k$, and we determine the first tree $t_i$ satisfying the condition, i.e., we only check one of these conditions once all trees failed all previous conditions:
\begin{enumerate}[label=(\roman*),leftmargin=8mm,topsep=0mm]
\item $t_i=q_1$ and $t_{i-1}=q_0$,
\item $t_i\in\{q_2,q_4\}$ and $t_{i+1}\in\{q_0,q_1,q_2\}$,
\item $t_i\notin\{q_0,q_1,q_2,q_4\}$,
\item $t_i\neq q_0$.
\end{enumerate}
Conditions~(i) and~(ii) refer to the previous tree~$t_{i-1}$ and the next tree~$t_{i+1}$ in the ccw ordering of $c$-subtrees, and those indices are considered modulo~$k$.
Note that $T$ is not the star~$[s_n]$, and so at least one $c$-subtree of~$T$ is distinct from~$q_0$ and satisfies the last condition, so this rule to determine $t_i$ is well defined.
We let $t_\hi$ be the $c$-subtree determined in this way.
Clearly, $t_\hi$ has at least two edges.

\begin{figure}
\includegraphics[page=1]{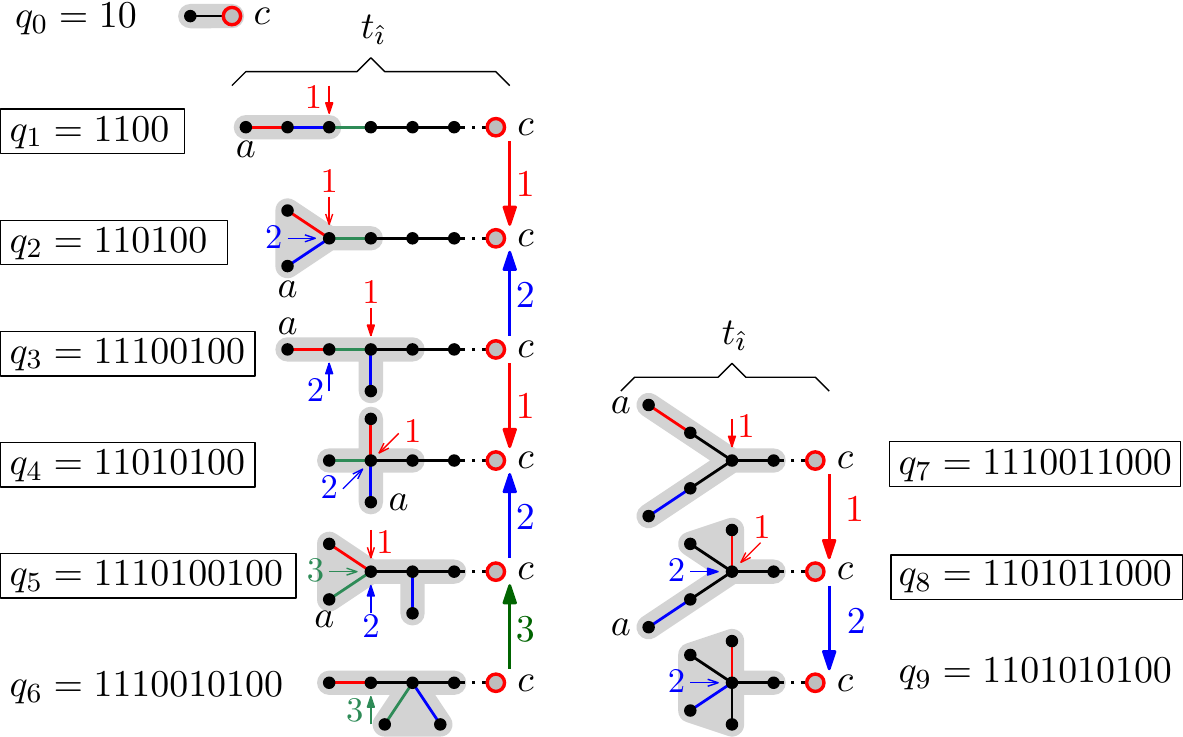}
\caption{Illustration of the trees $q_0,\ldots,q_9$ defined in~\eqref{eq:spires}, which are highlighted in gray, and pull/push operations between them.
In the spanning tree~$\cT_n$, every arc $([x],[y])$ is labeled with a gluing pair~$(x,y)$, and in the figure, the rooted trees~$x$ and~$y$ are obtained by rooting the plane trees~$[x]$ and~$[y]$ at the vertices indicated by the small arrows, which also show the splitting of the cyclic ordering of neighbors of this vertex to obtain the left-to-right ordering of the children of the root.
Moreover, the small arrow at~$[x]$ has a filled head, whereas the small arrow at~$[y]$ has an empty head.
For clarity, every arc and the corresponding two small arrows are marked by the same integer.
The framed trees $\{q_1,\ldots,q_5\}\cup\{q_7,q_8\}$ are treated by separate rules in step~(T2).
}
\label{fig:spire}
\end{figure}

\textbf{(T3) Select leaf to pull/push.}
If $t_\hi=1^l q_j 0^l$ for some $l\geq 0$ and $j\in\{1,\ldots,5\}\cup\{7,8\}$, i.e., $t_\hi$ is a path with one of the trees $q_1,\ldots,q_5$ or $q_7,q_8$ attached to it, then we distinguish four cases; see Figure~\ref{fig:spire}:
\begin{enumerate}[leftmargin=14mm,topsep=0mm]
\item[(q137)] If $j\in\{1,3,7\}$, then we let $a$ be the leftmost leaf of~$t_\hi$, which is thin, and define $x:=x(T,c,a)$ and $y:=\pull(x)$ as in Lemma~\ref{lem:pot-pull}.
Clearly, for $j=1$ we have $y=1^{l-1} q_2 0^{l-1}$ if $l>0$ and $y=q_0^2$ if $l=0$, for $j=3$ we have $y=1^l q_4 0^l$, and for $j=7$ we have $y=1^l q_8 0^l$.
\item[(q24)] If $j\in\{2,4\}$, then we let $a$ be the rightmost leaf of~$t_\hi$, which is thick, and we define $y:=y(T,c,a)$ and $x:=\push(y)$ as in Lemma~\ref{lem:pot-push}.
Clearly, for $j=2$ we have $x=1^{l-1} q_3 0^{l-1}$ if $l>0$ and $x=q_1q_0$ if $l=0$, and for $j=4$ we have $x=1^{l-1} q_5 0^{l-1}$ if $l>0$ and $x=q_2q_0$ if $l=0$.
\item[(q5)] If $j=5$, then we let $a$ be the unique leaf of~$t_\hi$ that is neither the leftmost nor the rightmost one, which is thick, and we define $y:=y(T,c,a)$ and $x:=\push(x)$ as in Lemma~\ref{lem:pot-push}.
We clearly have $x=1^l q_6 0^l$.
\item[(q8)] If $j=8$, then we let $a$ be the rightmost leaf of~$t_\hi$, which is thin, and define $x:=x(T,c,a)$ and $y:=\pull(x)$ as in Lemma~\ref{lem:pot-pull}.
We clearly have $y=1^l q_9 0^l$.
\end{enumerate}

Otherwise we distinguish two cases:
\begin{enumerate}[leftmargin=14mm,topsep=0mm]
\item[(e)] If the potential $\varphi(T)=\varphi(c)$ is even, we let $a$ be the leftmost leaf of~$t_\hi$ and define $x:=x(T,c,a)$ and $y:=\pull(x)$ as in Lemma~\ref{lem:pot-pull}.
\item[(o1)] If the potential $\varphi(T)=\varphi(c)$ is odd and the rightmost leaf~$a$ of~$t_\hi$ is thin, we define $x:=x(T,c,a)$ and $y:=\pull(x)$ as in Lemma~\ref{lem:pot-pull}.
\item[(o2)] If the potential $\varphi(T)=\varphi(c)$ is odd and the rightmost leaf~$a$ of~$t_\hi$ is thick, we define $y:=y(T,c,a)$ and $x:=\push(y)$ as in Lemma~\ref{lem:pot-push}.
\end{enumerate}

This completes the definition of~$\cT_n$.
In Lemma~\ref{lem:Tn-new} below we will show that~$\cT_n$ is indeed a spanning tree of~$\cH_n$.
The spanning trees $\cT_n\subseteq \cH_n$ for $n=4,5,6,7$ are shown in Figures~\ref{fig:t456} and~\ref{fig:t7}.

\begin{figure}
\includegraphics[page=2]{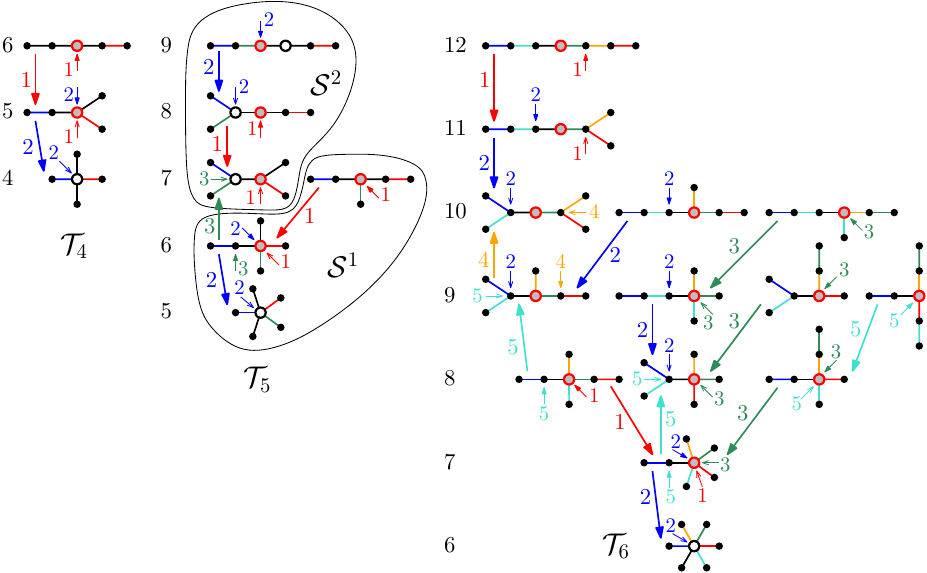}
\caption{Illustration of the spanning trees $\cT_4,\cT_5,\cT_6$.
The subgraphs~$\cS^1,\cS^2\subseteq \cT_n$ with all plane trees that have one or two centroids, respectively, are highlighted.
Centroid(s) are marked with bullets, where the centroid selected in step~(T1) is filled gray.
Plane trees are arranged in levels according to their potential, which is shown on the side.
The arrow markings are explained in Figure~\ref{fig:spire}.
}
\label{fig:t456}
\end{figure}

\begin{figure}
\includegraphics[page=3]{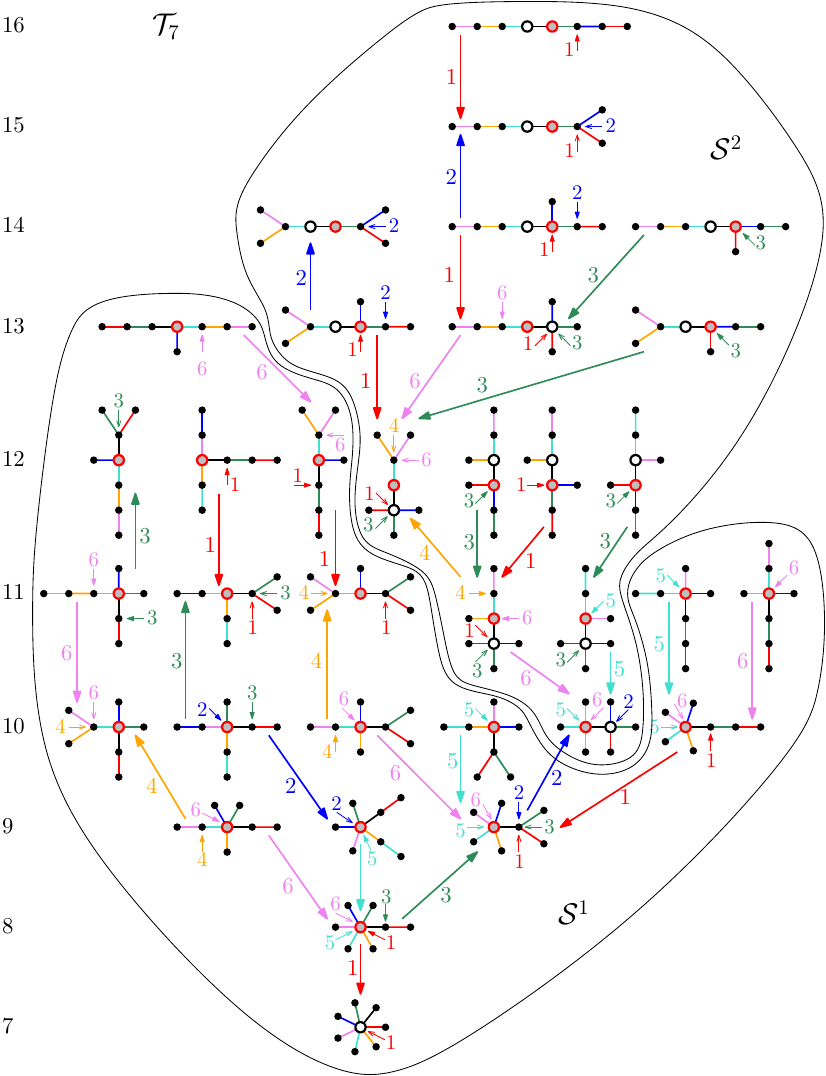}
\caption{Illustration of the spanning tree $\cT_7$.
Notation is as in Figure~\ref{fig:spire}.}
\label{fig:t7}
\end{figure}

In the following, we refer to rules~(q137), (q8), (e), and~(o1) in step~(T3) as \emph{pull rules}, and to rules~(q24), (q5), and~(o2) as \emph{push rules}.
Note that the leaf to which one of the pull rules~(q137), (q8) or~(o1) is applied is always thin, whereas the leaf to which any push rule is applied is always thick.

\subsection{Properties of~$\cT_n$}

The main task of this section is to prove that~$\cT_n$ is a spanning tree of~$\cH_n$ for which $G(\cT_n)$ is interleaving-free and nesting-free (Lemma~\ref{lem:Tn-new} below).
The following lemma is an auxiliary statement that will be used in that proof.

\begin{lemma}
\label{lem:circular}
If $T$ has a unique centroid~$c$, then the $c$-subtree $t_\hi$ selected in step~(T2) satisfies the following conditions:
\begin{enumerate}[label=(\alph*),leftmargin=8mm]
\item If $t_\hi=q_1$, then $t_{\hi-1}=q_0$ or $t_1=t_2=\cdots=t_k=q_1$.
\item If $t_\hi\in\{q_2,q_4\}$, then $t_{\hi+1}\in\{q_0,q_1,q_2\}$ or $t_1=t_2=\cdots=t_k=q_4$.
\end{enumerate}
\end{lemma}

\begin{proof}
We first prove~(a).
Among the conditions (i)--(iv) that are checked in step~(T2), only conditions~(i) and~(iv) lead to selecting a $c$-subtree that is isomorphic to~$q_1$ for~$t_\hi$.
If condition~(i) holds, then we clearly have $t_{\hi-1}=q_0$.
If condition~(iv) applies, then all $c$-subtrees $t_1,\ldots,t_k$ of~$T$ failed all previous conditions~(i)--(iii).
In particular, from~(i) we know that $t_i=q_1$ implies that $t_{i-1}\neq q_0$.
Moreover, from~(ii) we obtain that $t_{i-1}\notin\{q_2,q_4\}$.
Combining this with~(iii) shows that $t_{i-1}=q_1$.
This proves part~(a) of the lemma.

It remains to prove part~(b).
Among the conditions (i)--(iv), only conditions~(ii) and~(iv) lead to selecting a $c$-subtree that is isomorphic to~$q_2$ or~$q_4$ for~$t_\hi$.
If condition~(ii) holds, then we clearly have $t_{\hi+1}\in\{q_0,q_1,q_2\}$.
If condition~(iv) applies, then all $c$-subtrees $t_1,\ldots,t_k$ of~$T$ failed all previous conditions~(i)--(iii).
In particular, from~(ii) we know that $t_i\in\{q_2,q_4\}$ implies that $t_{i+1}\notin\{q_0,q_1,q_2\}$.
Combining this with~(iii) shows that $t_{i+1}=q_4$.
This proves part~(b) of the lemma.
\end{proof}

\begin{lemma}
\label{lem:Tn-new}
For any $n\geq 4$, the graph~$\cT_n$ is a spanning tree of~$\cH_n$, and for every arc $(T,T')$ in~$\cT_n$ we either have $\varphi(T')=\varphi(T)-1$ or $\varphi(T)=\varphi(T')-1$.
Every plane tree $T$ other than the star~$[s_n]$ has exactly one neighbor~$T'$ in~$\cT_n$ with $\varphi(T')=\varphi(T)-1$, which is an out-neighbor or in-neighbor.
Furthermore, $G(\cT_n)$ is interleaving-free and nesting-free.
\end{lemma}

\begin{proof}
Consider a gluing pair~$(x,y)\in G(\cT_n)$ added for a plane tree~$T$ with $T=[x]$.
By Lemma~\ref{lem:pot-pull} we have $\varphi(y)=\varphi(x)-1$, i.e., the potential of the trees changes by~$-1$ along this arc.
On the other hand, consider a gluing pair~$(x,y)$ added for a plane tree~$T$ with $T=[y]$.
By Lemma~\ref{lem:pot-push} we have $\varphi(x)=\varphi(y)-1$, i.e., the potential of the trees changes by~$+1$ along this arc.
It follows that in~$\cT_n$, every plane tree~$T$ other than the star~$[s_n]$ has exactly one neighbor~$T'$ with $\varphi(T')=\varphi(T)-1$, which is an out-neighbor or in-neighbor.
Consequently, $\cT_n$ has no cycles, regardless of the orientation of arcs along the cycle (in particular, there are no loops).
As from every plane tree~$T\in T_n$ other than~$[s_n]$, we can reach a tree~$T'$ with $\varphi(T')=\varphi(T)-1$, there is a path from~$T$ to the star~$[s_n]$, which is the unique plane tree with minimum potential~$n$.
We showed that $\cT_n$ does not contain cycles and is connected, i.e., it is a spanning tree.

We now show that~$G(\cT_n)$ is interleaving-free.
By Lemma~\ref{lem:pot-pull}, for any gluing pair $(x,y)\in G(\cT_n)$ with $\varphi(y)=\varphi(x)-1$ the right subtree of~$x$ contains a centroid of~$x$.
As a centroid is never a leaf, the right subtree of~$x$ contains edges, i.e., the root of~$x$ is not a leaf.
For any gluing pair $(x,y)\in G(\cT_n)$ with $\varphi(x)=\varphi(y)-1$, as $y$ is obtained from~$x$ by a push rule, the pushed leaf in~$y$ is thick, i.e., the right subtrees of~$y$ and~$x$ contain edges, and hence the root of~$x$ is not a leaf.
We can thus apply Lemma~\ref{lem:heavy} to conclude that $G(\cT_n)$ is interleaving-free.

The remainder of the proof is devoted to showing that~$G(\cT_n)$ is nesting-free.

We let $\cS^1$ and $\cS^2$ denote the subgraphs of~$\cT_n$ induced by all plane trees with a unique centroid, or with two centroids, respectively.
By Lemma~\ref{lem:centroid}, $\cS^1=\cT_n$ and $\cS^2=\emptyset$ for even~$n$, whereas $\cS^1$ and $\cS^2$ are both nonempty for odd~$n$; see Figures~\ref{fig:t456} and~\ref{fig:t7}.
For any plane tree~$T\neq [s_n]$ in $\cT_n$, consider the tree~$T'$ with $\varphi(T')=\varphi(T)-1$ that is connected to~$T$ in~$\cT_n$.
By Lemmas~\ref{lem:pot-pull} and~\ref{lem:pot-push}, if $T$ has a unique centroid, then $T'$ also has a unique centroid.
Similarly, if $T$ has two centroids, then $T'$ also has two centroids, unless $n$ is odd and $T=[d_n]=[d_n']$, in which case $T'=[\push(d_n')]$ has a unique centroid.
Consequently, $\cS^1$ and $\cS^2$ are subtrees of~$\cT_n$, and for odd~$n$ the subtree~$\cS^2$ is nonempty and $\cS^1$ and~$\cS^2$ are connected via the arc~$([\push(d_n')],[d_n'])$.

The following arguments are illustrated in Figure~\ref{fig:nest-free}.
Suppose for the sake of contradiction that~$G(\cT_n)$ is not nesting-free.
Then by Proposition~\ref{prop:Cxy}~(iii), there are gluing pairs $(x,y),(\hx,\hy)\in G(\cT_n)$ with $\hx=\rho^{-1}(y)$.
We consider the plane trees~$T:=[\hx]=[y]$, $T':=[\hy]$, and $T'':=[x]$.
We let $a$ denote the leaf in which~$\hx$ and~$\hy$ differ, which is also the leaf in which~$x$ and~$y$ differ.
Moreover, we let $b$ denote the root of~$\hx$, $b'$ the root of~$y$, and $b''$ the leftmost child of the root of~$x$.
As $\hx$ and~$x$ are pullable trees, we have $\hx=1\,1\,0\,u'\,0\,w$ and $x=1\,1\,0\,u\,0\,v'$ for $u',w,u,v'\in D$ (recall~\eqref{eq:gluing}).
Combining these relations with $\hx=\rho^{-1}(y)$ shows that if $u'=\varepsilon$, then we have $u=w$ and $v'=\varepsilon$ (in particular, $b=b''$), whereas if $u'\neq \varepsilon$, then we have $u'=1\,u\,0\,v$ and $v'=v\,1\,w\,0$.
The vertex identifiers~$a,b,b',b''$ and the subtree identifiers~$u,v,w$ apply to the rooted trees~$x,y,\hx,\hy$, but also to the plane trees~$T,T',T''$.
Note that $\hx=T^{(b,b')}$, $y=T^{(b',a)}$, $\hy=T'^{(b,a)}$, and $x=T''^{(b',b'')}$.
We let $c,c',c''$ denote the centroids of $T,T',T''$, respectively, selected in step~(D) or~(T1), and we let $t,t',t''$ denote the subtrees selected in step~(D) or~(T2).

\begin{figure}
\includegraphics[page=4]{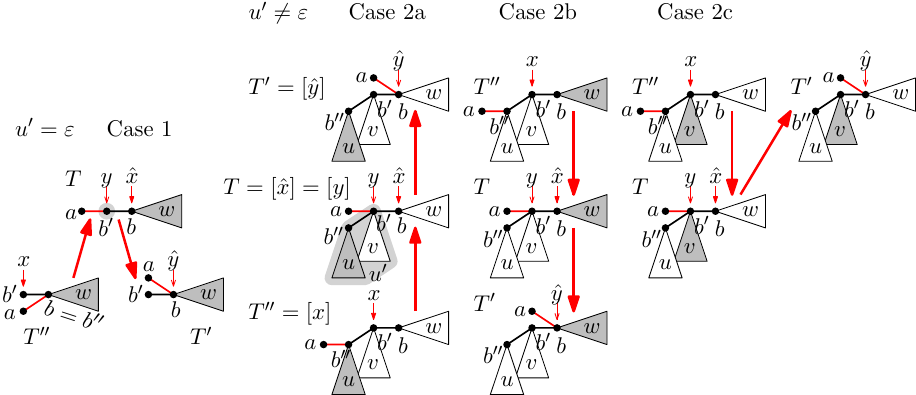}
\caption{Notations used in the proof that~$G(\cT_n)$ is nesting-free.
The gray subtrees contain the centroid(s).
}
\label{fig:nest-free}
\end{figure}

\textbf{Case 1:} We first consider the case $u'=\varepsilon$.
In this case the leaf~$a$ of~$T=[\hx]=[y]$ is thin, and so both $T'=[\hy]=[\pull(\hx)]$ and $T''=[x]=[\push(y)]$ have smaller potential than~$T$, i.e., we have $\varphi(T')=\varphi(T'')=\varphi(T)-1$.
This is impossible however, as $T$ has only a single neighbor in~$\cT_n$ with potential~$\varphi(T)-1$.

We now consider the case $u'\neq \varepsilon$.

\textbf{Case 2a:} $\varphi(T')=\varphi(T)+1$ and $\varphi(T'')=\varphi(T)-1$.
In this case, the leaf $a$ is thick and pushable to~$c'$ in~$T'$, implying that $w\neq\varepsilon$, and the leaf~$a$ is thick and pushable to~$c$ in~$T$.

\textit{Subcase 2a(i):} $T,T',T''\in\cS^1$ or $T,T',T''\in\cS^2$.
By Lemmas~\ref{lem:pot-pull} and~\ref{lem:pot-push}, the centroid(s) of $T,T',T''$ are identical.
As $a$ is pushable to~$c$ and in an active $c$-subtree in~$T$, the centroid(s) of~$T$ are in~$u$, in particular, $c$ is in~$u$.
As $a$ is in an active $c'$-subtree of~$T'$, we must have $c'=c$.

As $\hx$ is obtained from~$\hy$ by a push applied to the leaf~$a$, one of the push rules~(q24), (q5) or~(o2) in step~(T3) applies to~$t'$ in~$T'$.

However, rule~(q5) does not apply, as the rightmost leaf of~$q_5$ is missing in~$t'$.

If rule~(q24) applies to~$t'$, i.e., $t'$ is a path with $q_2$ or $q_4$ attached to it (in particular, $v=\varepsilon$), then $t$ is a path with $q_3$ or~$q_5$ attached to it, respectively, i.e., rule~(q137) or~(q5) apply to~$t$ in~$T$.
However, rule~(q137) is a pull rule, not a push rule, a contradiction.
If rule~(q5) applies to~$t$, then this rule applies a push to a leaf in~$w$, but not to~$a$, a contradiction.

If rule~(o2) applies to~$t'$ in~$T'$, i.e., $\varphi(T')$ is odd and $a$ is the rightmost leaf of~$t'$, then $\varphi(T)=\varphi(T')-1$ is even, so rule~(o2) does not apply to~$t$ in~$T$.
It remains to check that none of the push rules~(q24) or~(q5) applies to~$t$, either.
Rule~(q5) does not apply, as the rule does not push~$a$, which is the rightmost leaf of~$t$, but rather a leaf that is neither the rightmost nor the leftmost leaf of~$t$.
Rule~(q24) does not apply either, as $w\neq \varepsilon$.
In each case, we arrive at a contradiction.

\textit{Subcase 2a(ii):} $n$ is odd, $T=[d_n]\in\cS^2$ and $T''=[\push(d_n')]\in\cS^1$.
By rule~(D), it suffices to consider the case that the centroids of~$T$ are~$b'$ and~$b''$ and $c=b''$, i.e., we have $u=q_0^{(n-1)/2}$, $v=q_0^{(n-5)/2}$, and $w=\varepsilon$.
Then we have $c'=b'$ and rule~(q137) of step~(T3) applies to the $c'$-subtree $1w100=1100=q_1$ of~$T'$, but this is a pull rule, not a push rule, a contradiction.

\textit{Subcase 2a(iii):} $n$ is odd, $T'=[d_n]\in\cS^2$ and $T=[\push(d_n')]\in\cS^1$.
By rule~(D), it suffices to consider the case that the centroids of~$T'$ are $b$ and~$b'$ and $c'=b'$, i.e., we have $u=\varepsilon$ and $v=w=q_0^{(n-3)/2}$.
Then $c=b'$ is the unique centroid of~$T$, and the leaf~$a$ is not pushable to~$c$ in~$T$, a contradiction.

\textbf{Case 2b:} $\varphi(T'')=\varphi(T)+1$ and $\varphi(T')=\varphi(T)-1$.
In this case, the leaf $a$ is pullable to~$c''$ in~$T''$, and the leaf~$a$ is thick (due to the vertex~$b''$) and pullable to~$c$ in~$T$.

\textit{Subcase 2b(i):} $T,T',T''\in\cS^1$ or $T,T',T''\in\cS^2$.
By Lemmas~\ref{lem:pot-pull} and~\ref{lem:pot-push}, the centroid(s) of $T,T',T''$ are identical.
As $a$ is pullable to~$c$ and in an active $c$-subtree in~$T$, the centroid(s) of~$T$ are in~$w$, in particular, $c$ is in~$w$.
As $a$ is in an active $c''$-subtree of~$T''$, we must have $c''=c$.

As $y$ is obtained from~$x$ by a pull applied to the leaf~$a$, one of the pull rules~(q137), (q8), (e) or~(o1) in step~(T3) applies to~$t''$ in~$T''$.

However, rule~(q8) does not apply, as the leftmost leaf of~$q_8$ is missing in~$t''$.

If rule~(q137) applies to~$t''$, i.e., $t''$ is a path with $q_1$, $q_3$, or $q_7$ attached to it (in particular, $u=\varepsilon$), then $t$ is a path with $q_2$, $q_4$, or $q_8$ attached to it, respectively, i.e., rule~(q24) or~(q8) apply to~$t$ in~$T$.
However, rule~(q24) is a push rule, not a pull rule, a contradiction.
Also, rule~(q8) applies a pull to a leaf in~$v$, but not to~$a$, a contradiction.

If rule~(e) applies to~$t''$ in~$T''$, i.e., $\varphi(T'')$ is even and $a$ is the leftmost leaf of~$t''$, then $\varphi(T)=\varphi(T'')-1$ is odd, so rule~(e) does not apply to~$t$ in~$T$.
As $a$ is not the rightmost leaf of~$t$ in~$T$ due to the edge~$(b',b'')$, rule~(o1) does not apply to~$t$, either.
Moreover, none of the remaining pull rules~(q137) or~(q8) apply to~$t$, as they only apply to thin leaves, whereas $a$ is thick in~$T$.
We arrive at a contradiction.

If rule~(o1) applies to~$t''$ in~$T''$, i.e., $\varphi(T'')$ is odd and $a$ is the rightmost leaf of~$t''$, then $\varphi(T)=\varphi(T'')-1$ is even, so rule~(o1) does not apply to~$t$ in~$T$.
None of the pull rules~(q137) or~(q8) apply to~$t$ due to the fact that $a$ is thick in~$T$, as argued before.
Suppose that rule~(e) applies to $t$ in~$T$, i.e., $a$ is the leftmost leaf of~$t$.
However, if $a$ is the rightmost leaf of~$t''$ and the leftmost leaf of~$t$, then as $c''=c$ we obtain that $t$ is a path with $q_2$ attached to it, in which case the push rule~(q24) applies to~$t$, a contradiction.

\textit{Subcase 2b(ii):} $n$ is odd and $T''=[d_n]$ or $T=[d_n]$.
These cases are impossible as the tree~$[d_n]$ has an incoming arc from the tree $[\push(d_n')]$ with lower potential that it is connected to in~$\cT_n$, and no outgoing arcs to any such tree.

\textbf{Case 2c:} $\varphi(T')=\varphi(T'')=\varphi(T)+1$.
In this case, the leaf $a$ is pullable to~$c''$ in~$T''$, and the leaf~$a$ is thick and pushable to~$c'$ in~$T'$.

\textit{Subcase 2c(i):} $T,T',T''\in\cS^1$ or $T,T',T''\in\cS^2$.
By Lemmas~\ref{lem:pot-pull} and~\ref{lem:pot-push}, the centroid(s) of $T,T',T''$ are identical.
As $a$ is pullable to~$c''$ and in an active $c''$-subtree in~$T''$, the centroid(s) of~$T''$ are in~$v$ or~$w$.
Moreover, as $a$ is pushable to~$c'$ in~$T'$ and in an active $c'$-subtree of~$T'$, the centroid(s) of~$T'$ are in~$v$ or~$u$.
Combining these observations shows that the centroid(s) are in~$v$ and $c'=c''$.
As $y$ is obtained from~$x$ by a pull applied to the leaf~$a$, one of the pull rules~(q137), (q8), (e) or~(o1) in step~(T3) applies to~$t''$ in~$T''$.

In the following we first assume that the centroid is unique, i.e., $T,T',T''\in\cS^1$, and subsequently we explain how to modify these arguments if~$T,T',T''\in\cS^2$.

If rule~(q137) applies to~$t''$, then due to the edge~$(b',b)$, we must have $u=\varepsilon$, $c''=b'$, and $t''=q_1$.
Using Lemma~\ref{lem:circular}~(a), it follows that $w=\varepsilon$, or $n$ is even and all $c''$-subtrees of~$T''$ are copies of~$q_1$.
In the first case, the leaf~$a$ of~$T'$ is thin, a contradiction.
In the second case, we have $w=q_0$ and $v=q_1^{(n-4)/2}$.
If $n=4$, then $b$ is the unique centroid of~$T'$, and not $c'=b'$, a contradiction.
If $n\geq 6$, then $v$ consists of at least one copy of~$q_1$, and in~$T'$, rule~(i) in step~(T2) applies to the $c'$-subtree given by the leftmost such copy (as $u=\varepsilon$), and this rule has higher priority than rule~(ii) that selects the $c'$-subtree $1w100=110100=q_2$, a contradiction.

If rule~(q8) applies to~$t''$ in~$T''$, i.e., $t''$ is a path with $q_8$ attached to it, then we have $w=\varepsilon$, i.e., $a$ is thin in~$T'$, a contradiction.

If rule~(e) applies to~$t''$ in~$T''$, i.e., $\varphi(T'')$ is even and $a$ is the leftmost leaf of~$t''$, then for $a$ to be leftmost in~$t''$, we must have $c''=b'$ due to the edge~$(b',b)$.
Also we have $u\notin\{\varepsilon,q_0,q_0^2\}$, otherwise the $c''$-subtree $110u0$ would be equal to $q_1,q_2$, or $q_4$, respectively, and then rule~(q137) or~(q24) would apply to~$t''$ instead of rule~(e).
Clearly, the push rule~(o2) does not apply to~$t'$ in~$T'$, as $\varphi(T')=\varphi(T'')$ is even.
The push rule~(q5) does not apply to~$t'$ either, as this rule would apply a pull to a leaf in~$w$, and not to~$a$.
If the push rule~(q24) applies to~$t'$, then we have $1w100\in\{q_2,q_4\}$, i.e., $w=q_0$ or $w=q_0^2$.
By Lemma~\ref{lem:circular}~(b), the ccw next $c'$-subtree of $1w100$ in~$T'$, namely the tree~$1u0$, is from $\{q_0,q_1,q_2\}$, or $n$ is a multiple of~4 and all $c'$-subtrees of~$T'$ are isomorphic to~$q_4$.
In the first case, we get $u\in\{\varepsilon,q_0,q_0^2\}$, a contradiction to the conditions on~$u$ derived before.
In the second case, we have $w=q_0^2$, i.e., the $c''$-subtree $1w0$ of~$T''$ is isomorphic to~$q_2$.
Moreover, we have $u=q_0^3$, and $v$ consists of $(n-8)/4$ copies of~$q_4$.
If $n=8$, then $v=\varepsilon$ and the unique centroid of~$T''$ is $b''$, not $c''=b'$, a contradiction.
If $n\geq 12$, then the rightmost copy of~$q_4$ in~$v$ has $1w0=q_2$ as its ccw next $c''$-subtree, implying that rule~(ii) in step~(T2) applies to this subtree.
However, the $c''$-subtree $t''=110u0=1101010100$ is distinct from $q_0,q_1,q_2,q_4$, so it was selected by rule~(iii), which has lower priority, a contradiction.

If rule~(o1) applies to~$t''$ in~$T''$, i.e., $\varphi(T'')$ is odd and $a$ is the rightmost leaf of~$t''$, then we have $u=\varepsilon$.
Moreover, we have $c''\neq b'$, as otherwise rule~(q137) would apply to~$t''$ instead of rule~(o1), i.e., the centroid of~$T''$ is in~$v$, but not at the root of this subtree.
Consequently, the push rule~(o2) does not apply to~$t'$ in~$T'$, as $a$ is not the rightmost leaf of~$t'$ due to the edge~$(b',b'')$.
The push rule~(q24) does not apply to~$t'$ either, again due to the edge~$(b',b'')$, which is missing in~$q_2$ and~$q_4$.
If the push rule~(q5) applies to~$t'$, i.e., $t'$ is a path with $q_5$ attached to it, then we have $u=\varepsilon$, $w=q_0$ and $t''$ is a path with $q_7$ attached to it.
However, then rule~(q7) applies to~$t''$ in~$T''$, and not rule~(o1), a contradiction.

If $T,T',T''\in\cS^2$, then the above four cases for the pull rules applied to~$t''$ in~$T''$ can be adapted as follows:
The cases where rule~(q8) or~(o1) applies to~$t''$ are the same, only the cases where the rule~(q137) or~(e) applies have to be modified, due to the usage of Lemma~\ref{lem:circular}, which only applies if the centroid is unique.

If rule~(q137) or~(e) applies to~$t''$, then we have $c''=b'$ as before.
Now $w=\varepsilon$ follows from the fact that in step~(T2), $t''$ is selected as the first active $c''$-subtree in ccw order that is not a single edge.
But then $a$ is thin in~$T'$, a contradiction.

\textit{Subcase 2c(ii):} $n$ is odd, $T'=[d_n]\in\cS^2$ and $T=[\push(d_n')]\in\cS^1$.
By rule~(D), it suffices to consider the case that the centroids of~$T'$ are $b$ and~$b'$ and $c'=b'$, i.e., we have $u=\varepsilon$ and $v=w=(10)^{(n-3)/2}$.
Then $c''=b'$ is the unique centroid of~$T''$, and we have $11u0=1100=q_1$.
However, this $c''$-subtree violates the conditions of Lemma~\ref{lem:circular}~(a), as $w\neq \varepsilon$ and $v$ contains at least one $c''$-subtree that is a single edge, by the assumption $n\geq 4$.

This completes the proof.
\end{proof}

\subsection{Interaction with switches}

In this section we desribe how the newly defined spanning tree~$\cT_n$ interacts with the switches~$\tau_{n,1}$, $\tau_{n,2}$ and $\tau_{n,d,z}$ used in the proof of Theorem~\ref{thm:star}.

Recall the definition of conformal and usable switches given in Section~\ref{sec:interaction}, the latter of which was based on a set of gluing pairs~$G\subseteq G_n$.
For the following definitions, we now also assume that $G$ is nesting-free.
We say that a usable switch~$\tau$ is \emph{reversed}, if the $f$-edge of~$\tau$ lies on the reversed path of one of the gluing cycles~$\sigma^i(C(\hx,\hy))$, $(\hx,\hy)\in G$, for some $i\geq 0$, i.e., on the path $\sigma^i((\hx^1,\ldots,\hx^5))$.

The assumption that $G$ is nesting-free guarantees that along the periodic paths constructed from gluing the $f$-edges of all switches that are not reversed are oriented conformly, and oppositely to the $f$-edges of all reversed switches.
Consequently, for an $f$-conformal usable switch~$\tau$ that is not reversed, the modifications to the flip sequence described by~\eqref{eq:switch} change the shift by $+\lambda(\tau)$, and for an $f^{-1}$-conformal usable switch that is not reversed, they change the shift by $+\lambda(\tau^{-1})=-\lambda(\tau)$.
On the other hand, if the switch is reversed, then the sign of these changes is inverted.
We refer to this quantity as the \emph{effective shift} of~$\tau$; see Figure~\ref{fig:switch}.
The effective shift of~$\tau$ is the shift~$\lambda(\tau)$ with the sign determined by $f$-conformality (multiplied by~$-1$ iff $f^{-1}$-conformal) and whether the switch is reversed (multiplied by~$-1$ iff reversed).
The effective shift describes precisely, including the correct sign, by how much the shift of the flip sequence along a periodic path changes when applying the switching technique using this switch.

\begin{lemma}
\label{lem:reversed}
Let $\cT_n$ be the spanning tree of~$\cH_n$ defined in Section~\ref{sec:Tn-new}.
For $n\geq 4$, the switch $\tau_{n,1}$ is reversed w.r.t.~$G(\cT_n)$, and the switch~$\tau_{n,2}$ is not reversed w.r.t.\ any set of gluing pairs~$G\subseteq G_n$.
For $n,d,z$ as in Lemma~\ref{lem:taundz}, the switch $\tau_{n,d,z}$ is usable and not reversed w.r.t.~$G(\cT_n$).
\end{lemma}

\begin{proof}
We first consider the switch $\tau_{n,1}=1^n\ol{0}0^{n-1}\ul{0}=:(x,y,y')$.
Note that $x\in A_n$ and $y\in B_n$ differ in the last bit, and $(y,x)$ is the $f$-edge of~$\tau_{n,1}$ by Lemma~\ref{lem:tau12}.
To show that $\tau_{n,1}$ is reversed, consider the gluing pair $(\hx,\hy)=(11001^{n-2}0^{n-2},10101^{n-2}0^{n-2})$, which is in $G(\cT_n)$ by rule~(q137) in step~(T3) of the definition of~$\cT_n$.
As we have $t(\hx^4)=\rho^2(t(\hx^0))=t(x)$ (recall Proposition~\ref{prop:Fn}~(i)), the $f$-edge $(y,x)$ equals $\sigma^i((\hx^3,\hx^4))$ for some $i\geq 0$, proving that it lies on the reversed path of $\sigma^i(C(\hx,\hy))$.

We now consider the switch $\tau_{n,2}=(10)^{n-1}1\ul{0}\ol{0}=:(x,y,y')$.
Note that $x\in A_n$ and $y'\in B_n$ differ in the last bit, and $(y',x)$ is the $f$-edge of~$\tau_{n,2}$ by Lemma~\ref{lem:tau12}.
As $x\in A_n$ and $y'\in B_n$, it suffices to show that the $f$-edge of~$\tau_{n,2}$ is distinct from the $f$-edges $\sigma^i((\hx^1,\hx^2))$ and $\sigma^i((\hx^3,\hx^4))$, $i\geq 0$, of any gluing cycle $C(\hx,\hy)$ with $(\hx,\hy)\in G$.
We have $t(x)=(10)^n$ and therefore $\rho^{-1}(t(x))=1(10)^{n-1}0=s_n$ and $\rho^{-2}(t(x))=t(x)$.
However, neither $s_n$ nor $t(x)=10\cdots$ are pullable trees by the definition~\eqref{eq:gluing}, so we must have $t(\hx^2)=\rho(t(\hx^0)))\neq \rho(s_n)=t(x)$ and $t(\hx^4)=\rho^2(t(\hx^0))\neq \rho^2(t(x))=t(x)$ (recall Proposition~\ref{prop:Fn}~(i)).

Lastly, we consider the switch $\tau_{n,d,z}=(1\,z)^{(c-1)/2}\,\ol{0}\,z\,(0\,z)^{(c-3)/2}\,\ul{0}\,z=:(x,y,y')$.
Note that $x\in A_n$ and $y\in B_n$ differ in the bit before the suffix~$z$, and $(y,x)$ is the $f$-edge of~$\tau_{n,d,z}$ by Lemma~\ref{lem:taundz}.
The tree $t(x)$ is given by~\eqref{eq:tx1}.
The argument that~$\tau_{n,d,z}$ is usable for the spanning tree~$\cT_n$ defined in Section~\ref{sec:Tn-new} is analogous to the argument given in the proof of Lemma~\ref{lem:taundz}, as all leaves selected in step~(T3) of the definition of~$\cT_n$ are either leftmost or rightmost leaves of the selected subtree, or this subtree is isomorphic to a path with~$q_5$ attached to it (but $t(x)$ has no such subtrees).
It remains to prove that the switch $\tau_{n,d,z}$ is not reversed.

We first consider the case $d\geq 5$.
We have $\rho^{-1}(t(x))=111\cdots$ and $\rho^{-2}(t(x))=111\cdots$, i.e., none of these two trees is pullable by the definition~\eqref{eq:gluing}.
It follows that $t(\hx^2)\neq t(x)$ and $t(\hx^4)\neq t(x)$ for all $(\hx,\hy)\in G(\cT_n)$, proving that $\tau_{n,d,z}$ is not reversed w.r.t.~$G(\cT_n)$.
We now consider the case $d=3$.
We see from~\eqref{eq:tx2} that the tree $\rho^{-1}(t(x))$ has a leaf as its root, and we have $\rho^{-2}(t(x))=10\cdots$.
As all pullable trees~$\hx$ with $(\hx,\hy)\in G(\cT_n)$ have a root that is not a leaf, and the latter tree is not pullable by the definition~\eqref{eq:gluing}, we have $t(\hx^2)\neq t(x)$ and $t(\hx^4)\neq t(x)$ for all $(\hx,\hy)\in G(\cT_n)$, proving that $\tau_{n,d,z}$ is not reversed w.r.t.~$G(\cT_n)$.
\end{proof}

\subsection{Proof of Theorem~\ref{thm:algo}}

Our algorithm to compute a star transposition ordering of $(n+1,n+1)$-combinations is a faithful implementation of the constructive proof of Theorem~\ref{thm:star} presented in Section~\ref{sec:star-proof}, which also works with the spanning tree~$\cT_n$ of~$\cH_n$ defined in Section~\ref{sec:Tn-new} (in particular, the switch~$\tau_{n,d,z}$ is usable by Lemma~\ref{lem:reversed}).
The effective shifts of the switches $\tau_{n,1}$, $\tau_{n,2}$ and $\tau_{n,d,z}$ used in the proof, i.e., the signs $\gamma_1,\gamma_2$ and~$\gamma_d$ in~\eqref{eq:new-shifts}, can now be determined explicitly.
Specifically, from Lemma~\ref{lem:tau12}, \ref{lem:taundz} and~\ref{lem:reversed} we obtain that
\begin{equation}
\label{eq:gamma12d}
\gamma_1=(-1)\cdot (-1)=+1, \quad \gamma_2=(+1)\cdot (+1)=+1, \quad \gamma_d=(-1)\cdot (+1)=-1.
\end{equation}
In those products, the first factor is $-1$ iff the switch is $f^{-1}$-conformal and the second factor is $-1$ iff the switch is reversed.

\begin{proof}[Proof of Theorem~\ref{thm:algo}]

In the following we outline the key data structures and computation steps performed by our algorithm.
For more details, see the C++ implementation available at~\cite{cos_middle}.

The input of the algorithm is the integer $n\geq 1$, the initial combination~$\hx$ and the desired shift~$\hs$ coprime to~$2n+1$.
Upon initialization, we first compute the value of the $n$th Catalan number~$C_n$ modulo~$2n+1$ in time~$\cO(n^2)$, using Segner's recurrence relation.
By~\eqref{eq:lambda-Cn}, this gives us the shift~$s$ of the flip sequence obtained from gluing without modifications.
We then test whether~$s$ is coprime to~$2n+1$, and compute an appropriate set of one or two switches such that the shift~$s'$ of the modified flip sequence, computed using~\eqref{eq:new-shifts} and~\eqref{eq:gamma12d}, is coprime to~$2n+1$.
In particular, the definition of~$d$ used in~\eqref{eq:spd} involves computing the prime factorization of~$2n+1$.
From~$s'$ we compute the scaling factor~$s'^{-1}\hs$ and the corresponding initial combination~$x$ such that~$\hx$ is obtained from~$x$ by permuting columns according to the rule $i\mapsto s'^{-1}\hs i$ (simply apply the inverse permutation).
These remaining initialization steps can all be performed in time~$\cO(n)$.
All further computations are then performed with~$x$, and whenever a flip position is computed for~$x$, it is scaled by $s'^{-1}\hs$ before applying it to~$\hx$.

Throughout the algorithm, we maintain the following data structures:
\begin{itemize}[itemsep=0ex,parsep=0.5ex,leftmargin=2ex]
\item the bitstring representation~$x\in A_n\cup B_n$ of the current $(n+1,n+1)$-combination;
\item the position $\ell(x)$ from where to read the rooted tree~$t(x)$ in~$x$;
\item the plane tree~$T=[t(x)]$ and its centroid(s).
\end{itemize}
The space required by these data structures is clearly~$\cO(n)$.

There are two types of steps that we encounter in our algorithm:
An \emph{$f$-step} is simply an application of the mapping~$f$ defined in~\eqref{eq:f} to the current bitstring~$x$, which corresponds to following one of the basic flip sequences.
Such a step incurs only a rotation of the tree~$t(x)$ (recall~\eqref{eq:rot}), and therefore the plane tree~$T=[t(x)]=[\rho(t(x))]$ is not modified.
On a subpath that is reversed by a gluing cycle, we apply~$f^{-1}$ and inverse tree rotation~$\rho^{-1}(t(x))$ instead.
A \emph{pull/push step} is more complicated, and corresponds to following one of the edges of a gluing cycle~$\sigma^i(C(\hx,\hy))$, $i\geq 0$, $(\hx,\hy)\in G(\cT_n)$ for some arc~$([\hx],[\hy])$ in the spanning tree~$\cT_n$.
Such a step also modifies the plane tree~$T=[t(x)]$ by applying a pull or push operation to one of its leaves.
All these updates can easily be done in time~$\cO(n)$.

For deciding whether to perform an $f$-step or a pull/push step, the following computations are performed on the current plane tree~$T=[t(x)]$, following the steps~(T1)--(T3) described in Section~\ref{sec:Tn-new}:
\begin{itemize}[itemsep=0ex,parsep=0.5ex,leftmargin=2ex]
\item compute a centroid~$c$ of~$T$ and its potential~$\varphi(c)$ as in step~(T1) in time~$\cO(n)$ (see~\cite{MR396308});
\item compute the lexicographic subtree ordering as in step~(T1) in time~$\cO(n)$.
In the case where the centroid is unique, this is achieved by Booth's algorithm~\cite{MR585391}.
Specifically, to compute the lexicographically smallest ccw ordering $(t_1,\ldots,t_k)$ of the $c$-subtrees of~$T$ we insert $-1$s as separators between the bitstring representations $t_1,\ldots,t_k$ of the subtrees, i.e., we consider the string~$z:=(-1,t_1,-1,\ldots,-1,t_k)$.
This trick makes Booth's algorithm return a cyclic rotation of~$z$ that starts with~$-1$, and it is easy to check that this rotation is also the one that minimizes the cyclic subtree ordering~$(t_1,\ldots,t_k)$.
\item compute a $c$-subtree of~$T$ and one of its leaves as in steps~(T2) and~(T3) in Section~\ref{sec:Tn-new} in time~$\cO(n)$.
\end{itemize}
Overall, the decision which type of step to perform next takes time~$\cO(n)$ to compute.

Whenever we encounter a switch in the course of the algorithm, which can be detected in time~$\cO(n)$, we perform a modified flip as described by~\eqref{eq:switch}.
Each time this happens, the position~$\ell(x)$ has to be recomputed, while the plane tree~$T=[t(x)]$ does not change.

Summarizing, this algorithm runs in time~$\cO(n)$ in each step, using $\cO(n)$ memory in total, and it requires time~$\cO(n^2)$ for initialization.
\end{proof}

We remark here that if $2n+1$ is prime, then a straightforward calculation shows that $C_n=2(-1)^n \pmod{2n+1}$.
In particular, $C_n$ is coprime to~$2n+1$ in this case, and no switches are necessary.
Unfortunately, in general the behavior of the sequence $C_n\bmod{2n+1}$ seems quite hard to understand.
The first few entries are $-2,2,-2,-4,-2,2,-6,2,-2,-4,-2,12,-5,\allowbreak 2,-2,\ldots=1,2,5,5,9,2,9,2,17,17,21,12,22,2,29,\ldots$.

\section*{Acknowledgements}

We thank the anonymous reviewers of this paper for their numerous helpful comments.
In particular, one referee pointed out the scaling trick now described in Section~\ref{sec:idea-shift}, which led to substantial simplifications of our proofs.

\bibliographystyle{alpha}
\bibliography{../refs}

\end{document}